\documentclass[
11pt,  reqno]{amsart}
\usepackage{amsmath,amssymb,amscd,amsthm}

\usepackage{color}
\usepackage{hyperref}

\usepackage{bm}

\allowdisplaybreaks[3]

\newtheorem{theorem}{Theorem} [section]

\newtheorem{lemma}[theorem]{Lemma}
\newtheorem{proposition}[theorem]{Proposition}
\newtheorem{remark}[theorem]{Remark}
\newtheorem{definition}[theorem]{Definition}
\newtheorem{corollary}[theorem]{Corollary}

\DeclareMathOperator*{\esssup}{ess \ sup}

\newcommand{\Z}{\mathbb{Z}}
\newcommand{\R}{\mathbb{R}}
\newcommand{\C}{\mathbb{C}}
\newcommand{\T}{\mathbb{T}}

\newcommand{\al}{\alpha}

\newcommand{\dl}{\delta}

\newcommand{\Ta}{\Theta}

\newcommand{\eps}{\varepsilon}

\newcommand{\Ld}{\Lambda}

\newcommand{\ft}{\widehat}
\newcommand{\Ft}{{\mathcal{F}}}
\newcommand{\wt}{\widetilde}
\newcommand{\cj}{\overline}
\newcommand{\dx}{\partial_x}

\newcommand{\dt}{\partial_t}
\newcommand{\br}{\breve}

\newcommand{\pd}{\partial}

\newcommand{\ta}{\theta}

\renewcommand{\l}{\ell}

\newcommand{\les}{\lesssim}
\newcommand{\ges}{\gtrsim}

\newcommand{\jb}[1]
{\langle #1 \rangle}

\newcommand{\Ml}{\mathcal{M}}
\newcommand{\Nl}{\mathcal{N}}
\newcommand{\Kl}{\mathcal{K}}
\newcommand{\Rl}{\mathcal{R}}

\newcommand{\N}{\mathbb{N}}

\renewcommand{\o}{\omega}

\newcommand{\Com}[2]{\mathrm{Com}^{#1} [#2]}

\newcommand{\I}{\mathrm{I}}
\newcommand{\II}{\mathrm{I \! I}}
\newcommand{\III}{\mathrm{I \! I \! I}}

\newcommand{\regwp}{s_\ast (\al)}

\numberwithin{equation}{section}
\numberwithin{theorem}{section}

\let\Re=\undefined\DeclareMathOperator*{\Re}{Re}
\let\Im=\undefined\DeclareMathOperator*{\Im}{Im}

\begin{document}
\baselineskip = 14pt

\title[Well- and Ill-posedness for derivative fNLS on the torus]
{Well- and Ill-posedness of the Cauchy problem for derivative fractional nonlinear Schr\"odinger equations
on the torus}

\author[T.K.~Kato, T.~Kondo, and M.~Okamoto]
{Takamori Kato, Toshiki Kondo, and Mamoru Okamoto}

\address{
Takamori Kato\\
Mathematical Science Course\\
Faculty of Science and Engineering\\
Saga University\\
1 Honjo-machi, Saga-city, Saga, 840-8502, Japan
}
\email{tkkato@cc.saga-u.ac.jp}

\address{
Toshiki Kondo\\
Department of Mathematics\\
Graduate School of Science\\ Osaka University\\
Toyonaka\\ Osaka\\ 560-0043\\ Japan
}
\email{u534463k@ecs.osaka-u.ac.jp}

\address{
Mamoru Okamoto\\
Department of Mathematics\\\
Graduate School of Advanced Science and Engineering\\
Hiroshima University\\
1-3-1 Kagamiyama,
Higashi-Hiroshima, 739-8526
Japan}
\email{mokamoto@hiroshima-u.ac.jp}

\subjclass[2020]{35Q55}

\keywords{fractional Schr\"odinger equation;
well-posedness; modified energy; ill-posedness; non-existence}

\begin{abstract}
We consider the Cauchy problem for derivative fractional Schr\"odinger equations (fNLS) on the torus $\T$. 
Recently, the second and third authors established a necessary and sufficient condition on the nonlinearity for well-posedness of semi-linear Schr\"odinger equations on $\T$.
In this paper, we extend this result to derivative fNLS.
More precisely, we prove that the necessary and sufficient condition on the nonlinearity is the same as that for semi-linear Schr\"odinger equations.
However, since we can not employ a gauge transformation for derivative fNLS,
we use the modified energy method to prove well-posedness.
We need to inductively construct correction terms for the modified energy when the fractional Laplacian is of order between $1$ and $\frac 32$.
For the ill-posedness, we prove the non-existence of solutions to the Cauchy problem by exploiting a Cauchy-Riemann-type operator that appears in nonlinear interactions.

\end{abstract}

%
\maketitle

\tableofcontents

\newpage

\section{Introduction}

We consider the Cauchy problem for the following fractional nonlinear Schr\"odinger equation (fNLS):
\begin{equation} 
\left\{
\begin{aligned}
&\dt u + i D^\al u = F(u, \dx u, \cj u, \cj{\dx u}), \\ 
&u|_{t=0} = \phi,
\end{aligned}
\right.
\label{fNLS}
\end{equation}
where $\al > 2$, $\T := \R /(2\pi \Z)$, $u = u(t,x) : \R \times \T \to \C$ is an unknown function,
and $\phi : \T \to \C$ is a given function.
Here, $D^\al$ is the fractional derivative defined by
\[
D^\al f(t,x) = \sum_{k \in \Z} |k|^\al \ft f(t,k) e^{ikx},
\] 
$F(\zeta, \o, \cj{\zeta}, \cj{\o})$ is a polynomial in $\zeta, \o, \cj \zeta, \cj \o \in \C$,
and $\cj \zeta$ denotes the complex conjugate of $\zeta \in \C$.
The fNLSs appear in fractional quantum mechanics,
a generalization of the standard quantum mechanics extending the Feynman path integral to L\'evy processes \cite{La00}.
The aim in this paper is to determine a necessary and sufficient condition on $F$ for well-posedness of \eqref{NLS} in the Sobolev space.

When $\al = 2$, the equation \eqref{fNLS} is the semi-linear Schr\"odinger equation:
\begin{equation}
\dt u - i \dx^2 u =  F(u, \dx u, \cj u, \cj{\dx u}).
\label{NLS}
\end{equation}
There are many results on \eqref{NLS}.
See, for instance, \cite{Tak99, Chi02, Her06, DNY21, KNV23, BaPe24, HOV25}.
In particular,
the second and third authors \cite{KoOk25b} established a necessary and sufficient condition on $F$ to be well-posed for \eqref{NLS} in the Sobolev space $H^s(\T)$ for $s>\frac 52$.

When $\al \neq 2$,
fNLSs without derivative nonlinearity have been studied by many researchers, even when $\al<2$.
Since there are many results, we will cite here only a few recent results on $\T$:
\cite{SRR25, AlTz25, BLLZ} for $\al>2$ and
\cite{EGT19, LiWa24, YZL25} for $\al<2$.
See also \cite{Mizu06, Chi15, TaTs22} for higher-order linear Schr\"odinger equations.

However, there are a few results on fNLS with derivative nonlinearity on $\T$ even when $\al>2$.
Segata \cite{Se12, Se15} proved well-posedness for fourth-order nonlinear Schr\"odinger equations (4NLS).
The first author \cite{Kato25} proved the unconditional well-posedness for 4NLS.
Onodera \cite{Ono24} studied a system of 4NLS.
Liu and Zhang \cite{LiZh24} constructed quasi-periodic solutions to fNLS for $\al \in (1,2)$.
In particular,
to the best of the authors' knowledge, there are no well-posedness results in periodic Sobolev spaces for fNLS with derivative nonlinearity when $\al<4$, except for $\al=2$.
In contrast to fNLS,
for fractional Korteweg-de Vries-type equations, especially when the unknown function is real-valued,
there has been a lot of research even on $\T$.
See \cite{KaTo06, KePi16, Kwa18, Sch20, Gas21, Roy22, MoTa22, Joc24}, for example.

In the case of the real line, the situation changes drastically.
Thanks to the local smoothing effect on $\R$, we are able to recover the derivative losses unless the nonlinearity includes derivative quadratic terms.
In fact, when $\al=2$, Kenig, Ponce, and Vega \cite{KPV93} proved that \eqref{NLS} is well-posed in $H^s(\R)$ for $s>\frac 72$ if $F$ is a cubic or higher-order polynomial.
When $\al>2$, the dispersion effect becomes stronger, so their result can be extended to $\al>2$ as well.

Before stating the main result, we define some notations.

\begin{definition}
\label{def:sol}
Let $s>\frac 32$, $T > 0$, and $\phi \in H^s(\T)$.
We say that $u$ is a solution to \eqref{fNLS} in $H^s(\T)$ on $[0,T]$
if $u$ satisfies the following two conditions:

\begin{enumerate}
\item
$u \in L^\infty ([0,T]; H^s(\T))$;

\item
For any $\chi \in C_c^\infty([0,T) \times \T)$,
\footnote{Here, $C_c^\infty([0,T) \times \T)$ denotes the space of smooth functions with compact support in $[0,T) \times \T$.
}
we have
\begin{align*}
&- \int_0^T \int_\T u (t,x) \dt \chi(t,x) dx dt
- \int_\T \phi (x) \chi (0,x) dx
\\
&\quad
+ i \int_0^T \int_\T u(t,x) D^\al \chi(t,x) dx dt
\\
&=
\int_0^T \int_\T F(u(t,x), \dx u(t,x), \cj{u(t,x)}, \cj{\dx u(t,x)}) \chi (t,x) dx dt.
\end{align*}
\end{enumerate}
A solution on $[-T,0]$ is defined in a similar manner.
\end{definition}

The second condition (ii) in Definition \ref{def:sol} means that $u$ satisfies \eqref{fNLS} in the sense of distribution.
The condition $s>\frac 32$, together with Corollary \ref{cor:bili3} below, guarantees that the nonlinear term $F(u, \dx u, \cj u , \cj{\dx u})$ is well-defined.

\begin{remark}
\label{rem:sold}
\rm
If $s >\frac 32$ and $u \in C([0,T]; H^s(\T))$ is a solution to \eqref{fNLS},
Corollary \ref{cor:bili3} below yields that
\[
F( u, \dx u, \cj u , \cj{\dx u}) \in C([0,T]; H^{s-1}(\T)).
\]
Hence,
$u \in C^1([0,T]; H^{s-\al}(\T))$
and
\[
\dt u + i D^\al u = F( u, \dx u, \cj u , \cj{\dx u})
\]
holds in $H^{s-\al}(\T)$ for $t \in [0,T]$.
\end{remark}

We say that \eqref{fNLS} is (locally in time) well-posed in $H^s(\T)$ if the following two conditions are satisfied:
\begin{enumerate}
\item
For any $\phi \in H^s(\T)$,
there exist $T>0$ and a unique solution $u \in C([-T,T]; H^s(\T))$ to \eqref{NLS}.%
\footnote{
Corollary \ref{cor:diff2} yields the unconditional uniqueness.
In this paper, we only consider the case for $s$ satisfying the assumption of Corollary \ref{cor:diff2},
so we will also consider uniqueness in $C([-T,T]; H^s(\T))$ here.
}

\item
If $\phi_n$ converges to $\phi$ in $H^s(\T)$,
then the corresponding solution $u_n$ to \eqref{NLS} with the initial data $\phi_n$ converges to $u$ in $C([-T,T]; H^s(\T))$.
\end{enumerate}
That is,
well-posedness means that the solution map can be defined
and that it is continuous with respect to the initial data.

For $\zeta = a+ib \in \C$ ($a,b \in \R$),
we write
\[
\frac{\pd F}{\pd \zeta} := \frac 12 \Big( \frac{\pd F}{\pd a} - i \frac {\pd F}{\pd b} \Big), \quad
\frac{\pd F}{\pd \cj{\zeta}} := \frac 12 \Big( \frac{\pd F}{\pd a} + i \frac {\pd F}{\pd b} \Big).
\]
We also use the following abbreviations:
\[
F_\zeta = \frac{\pd F}{\pd \zeta}, \quad
F_\o = \frac{\pd F}{\pd \o}, \quad
F_{\cj{\zeta}} = \frac{\pd F}{\pd \cj{\zeta}}, \quad
F_{\cj{\o}} = \frac{\pd F}{\pd \cj{\o}}.
\]
Set
\begin{align}
\regwp
&:=\max\Big(\frac \al2 +1 , \frac 52\Big).
\label{condA}
\end{align}
The following is the main result in this paper.

\begin{theorem}
\label{thm:equiv}
Let $\al > 2$ and  $s \in \R$ satisfy
\[
s>
\regwp,
\]
where $\regwp$ is defined in \eqref{condA}.
Then, the Cauchy problem \eqref{fNLS} is well-posed in $H^s(\T)$ if and only if
\begin{equation}
\int_\T \Im F_\o ( \psi(x),\dx \psi(x),\cj{\psi(x)},\cj{\dx \psi(x)} ) dx = 0
\label{Fbez}
\end{equation}
for any $\psi \in H^s (\T)$.
\end{theorem}

For the ill-posedness in Theorem \ref{thm:equiv},
we prove the non-existence of solutions to \eqref{fNLS}.

\begin{theorem}
\label{thm:nonexi}
Let $\al > 2$ and  $s \in \R$ satisfy
\[
s>
\regwp,
\]
where $\regwp$ is defined in \eqref{condA}.
Suppose that there exists $\psi \in H^s (\T)$ such that
\begin{equation}
\int_\T \Im F_\o ( \psi(x),\dx \psi(x),\cj{\psi(x)},\cj{\dx \psi(x)} ) dx \neq 0.
\label{Fbez2}
\end{equation}
Then,
there exists $\phi  \in H^s (\T)$ such that for no $T > 0$
the Cauchy problem \eqref{fNLS} has a solution in $C([0,T]; H^s(\T))$ or $C([-T,0]; H^s(\T))$.
\end{theorem}

Due to nonlinear interactions, resonant terms appear where dispersion effects vanish.
The condition \eqref{Fbez} cancels a problematic resonant part.
Theorem \ref{thm:nonexi} states that this part is the only problematic aspect in demonstrating the existence of a solution to \eqref{fNLS}.

The conditions \eqref{Fbez} and \eqref{Fbez2} are the same as those for $\al=2$.
See \cite{KoOk25b}.
In particular, under \eqref{Fbez2},
regardless of how large $\al$ is, the derivatives in the nonlinear term can not be controlled by the dispersive effect.
Consequently, without structural conditions such as \eqref{Fbez} on the nonlinear term,
dispersive equations on $\T$ do not admit solutions.

\begin{remark}
\rm

From Theorems \ref{thm:equiv} and \ref{thm:nonexi},
the condition \eqref{Fbez} is equivalent to the existence of a solution $u \in C([0,T]; H^s(\T))$ or $u \in C([-T,0]; H^s(\T))$ for some $T>0$.
We note that
the conditions for solvability depend on the function space in which the solutions are sought.
In fact, by the same argument as in Section 4 of \cite{KiTs18},
in the space of real-analytic functions, \eqref{fNLS} has a solution even without assuming \eqref{Fbez}.
Moreover, when $F$ depends only on $\zeta$ and $\o$ (i.e., is independent of $\cj \zeta$ and $\cj \o$),
Nakanishi and Wang \cite{NaWa25} proved that \eqref{fNLS} is globally well-posed in a space of
distributions whose Fourier support is in the half-space.
\end{remark}

\begin{remark}
\rm
Since we need to deal with fractional derivatives in the construction of the correction term (see \eqref{modene1} below) in the proof of well-posedness
and Proposition \ref{prop:FGH} in the proof of non-existence,
this paper only considers polynomial nonlinearities to avoid certain technical difficulties arising from fractional derivatives.
However, the same result holds for real-analytic nonlinearity $F$ with respect to $\Re \zeta$, $\Im \zeta$, $\Re \o$, and $\Im \o$.
\end{remark}

Here, we give examples of nonlinear terms.

\begin{enumerate}
\item[(a)]
If $F$ is independent of $\o$,
we then have $F_\o = 0$.
This implies \eqref{Fbez}.
The typical nonlinear terms in this case are
\[
|u|^{2m} u, \quad \cj{u}^m \cj{\dx u}
\]
for some $m \in \N$.
These correspond to the cases of
$\zeta^{m+1} \cj{\zeta}^{m}$ and $\cj \zeta^m \cj \o$, respectively.

\item[(b)]
Let $c \in \C$ and $m \in \N$.
The nonlinear term
\[
c u^m \dx u
\]
corresponds to $F(\zeta, \o, \cj \zeta, \cj \o) = c \zeta^m \o$.
Since
$F_\o(\zeta, \o, \cj \zeta, \cj \o) = c \zeta^m$,
the condition \eqref{Fbez} holds if and only if $c =0$.
Moreover, when $c \neq 0$,
the condition \eqref{Fbez2} holds with $\psi(x)= c^{-\frac 1m} e^{\frac \pi {2m} i}$.
While the second and third authors \cite{KoOk25a} proved
norm inflation for \eqref{fNLS} with $\al>3$ and this nonlinearity,
Theorem \ref{thm:nonexi} implies the non-existence of a solution,
which is a stronger result of ill-posedness.

\item[(c)]
Let $c \in \C$.
The nonlinear term
\[
c \dx (|u|^2 u)
\]
corresponds to $F(\zeta, \o, \cj \zeta, \cj \o) = 2c |\zeta|^2 \o + c \zeta^2 \cj \o$.
Since
$F_\o(\zeta, \o, \cj \zeta, \cj \o) = 2c |\zeta|^2$,
the condition \eqref{Fbez} holds if and only if $c$ is real.
Moreover, when $c$ is not real,
the condition \eqref{Fbez2} holds with $\psi(x)=1$.

\item[(d)]
Let $c_1, c_2 \in \C$.
The nonlinear term
\[
c_1 (\dx u)^2 \cj u + c_2 |\dx u|^2 u
\]
corresponds to
$F(\zeta, \o, \cj \zeta, \cj \o) = c_1 \o^2 \cj \zeta + c_2 |\o|^2 \zeta$.
Since
$F_\o(\zeta, \o, \cj \zeta, \cj \o)
= 2 c_1 \o \cj \zeta + c_2 \cj{\o} \zeta$,
we have
\begin{align*}
\qquad
&\int_\T \Im F_\o ( \psi,\dx \psi,\cj{\psi},\cj{\dx \psi} ) dx
\\
&= \int_\T \Im \big( 2 c_1 \cj{\psi} \dx \psi + c_2 \psi \cj{\dx \psi} \big) dx
\\
&= \Im (2c_1 + c_2) \int_\T \Re (\cj{\psi} \dx \psi) dx
+ \Re (2c_1-c_2) \int_\T \Im (\cj{\psi} \dx \psi) dx.
\end{align*}
By $\Re (\cj{\psi} \dx \psi) = \frac 12 \dx (|\psi|^2)$, the first term on the right-hand side vanishes.
Hence,
the condition \eqref{Fbez} holds if and only if $\Re (2c_1-c_2)=0$.
Note that the nonlinear term where $c_1$ and $c_2$ are purely imaginary is treated in \cite{Se12}.
Moreover,
when $\Re (2c_1-c_2) \neq 0$,
the condition \eqref{Fbez2} holds with $\psi(x)= e^{ix}$.

In this example, it is further observed that there are no solutions to \eqref{fNLS} with the individual nonlinear terms
$(\dx u)^2 \cj u$ and $|\dx u|^2 u$,
but there is a solution to \eqref{fNLS} with the nonlinearity given by the linear combination $(\dx u)^2 \cj u + 2 |\dx u|^2 u$.
Thus, the structure of the nonlinear term can change due to the sum. 
\end{enumerate}

For the well-posedness of \eqref{fNLS},
we employ a modified energy method.
To reveal the most problematic term in the nonlinearity,
we consider the system for $u$ and $v = \dx u$.
It follows from \eqref{fNLS} that $v$ satisfies the following equation:
\begin{equation}
\begin{aligned}
\dt v + i D^\al v
&=
F_\o (u,v, \cj u, \cj v) \dx v
+ F_{\cj \o} (u,v, \cj u, \cj v) \cj{\dx v}
\\
&\quad
+ \text{(a polynomial in $u,v, \cj u, \cj v$)}.
\end{aligned}
\label{eqvv}
\end{equation}
We can not directly apply an energy estimate to the first term on the right-hand side of \eqref{eqvv}.
In contrast, the second term on the right-hand side is harmless.

When $\al = 2$,
the gauge transformation works well to handle the problematic term.
See \cite{Chi02, KoOk25b}.
The gauge transformation means choosing an appropriate $\Ld$ such that $V := e^\Ld v$ satisfies an equation with harmless terms.
We give a rough idea of how to choose $\Ld$.
With \eqref{eqvv},
a direct calculation yields that
\begin{align*}
(\dt - i \dx^2) (e^\Ld v)
=
\big( F_\o (u,v, \cj u, \cj v) -2i \dx \Ld \big) e^\Ld \dx v
+ \text{(other terms)}.
\end{align*}
To nearly cancel the first term on the right-hand side,
we choose
\[
\Ld = 
- \frac i2 \dx^{-1}
F_\o (u,v, \cj u, \cj v).
\]
Then,
\[
V := \exp \Big( - \frac i2 \dx^{-1} F_\o (u,v,\cj u, \cj v) \Big) v
\]
satisfies the following:
\begin{equation*}
\dt V - i \dx^2 V
=
\int_\T  F_\o (u,v, \cj u, \cj v) dx
\cdot \dx V
+ \text{(harmless terms)}.
\end{equation*}
If \eqref{Fbez} holds, we can apply the (usual) energy method to this equation.

On the other hand,
when $\al >2$,
a gauge transformation can no longer be applied to eliminate the problematic term.
Indeed,
since
the fractional Leibniz rule (see \cite{Li19}, for example)
formally
yields that
\begin{equation}
\begin{aligned}
D^\al \big( fg \big)
&= (D^{\al} f) g + f D^{\al}g
\\
&\quad
- \al ( D^{\al-2} \dx f ) \dx g
- \al ( \dx f) D^{\al-2} \dx g
\\
&\quad
+\text {(other terms)},
\end{aligned}
\label{fLeib1}
\end{equation}
with \eqref{eqvv},
we have
\begin{align*}
(\dt + i D^\al) (e^\Ld v)
&=
\big( F_\o (u,v, \cj u, \cj v) e^\Ld - i \al D^{\al-2} \dx e^\Ld \big) \dx v
\\
&\quad
- i \al \big( \dx e^\Ld \big) D^{\al-2} \dx v
+ \text{(other terms)}.
\end{align*}
The first term on the right-hand side can be almost canceled by appropriately choosing $\Ld$.
However,
the second term on the right-hand side contains $D^{\al-2} \dx v$, which has a derivative of order 
$\al-1 > 1$.
As a result, applying a gauge transformation to \eqref{fNLS} when $\al > 2$ generates a new problematic term.
Therefore, 
for $\al > 2$, an alternative approach must be used in place of the gauge transformation.

To overcome this difficulty, we employ a modified energy method.
In this approach, correction terms are added to the energy functional to cancel the problematic contributions. 
Roughly speaking, the use of a correction term allows us to recover derivatives of order $\al-2$.
In our setting, since the nonlinear term contains one derivative, a single correction term suffices for $\al \ge 3$.
However, for $\al \in (2,3)$,
the correction terms must be constructed inductively.
See \eqref{modene1} below.

We employ a refined Kato–Ponce-type commutator estimate (Proposition \ref{prop:comm3}),
which is a variant of the fractional Leibniz rule,
in the construction of the correction terms.
Since $\al>2$,
terms analogous to the third term on the right-hand side of \eqref{fLeib1} appear when applying this commutator estimate.
The term generated by the $n$-th correction term cancels with the term generated by the ($n+1$)-th correction term.
Exploiting this cancellation property enables us to construct the sequence of correction terms.
See Lemma \ref{lem:modi}.

To prove well-posedness for \eqref{fNLS} when $\al>2$,
we apply the modified energy method, which differs from the approach used for $\al=2$.
Nevertheless, the two methods are closely related.
Indeed, if we formally set $\al=2$ in the correction terms we construct,
we recover the gauge transformation employed in \cite{KoOk25b}.
In this sense, the gauge transformation can be interpreted as a modified energy method with an infinite sequence of correction terms.
See Remark \ref{rem:gemod} for details.

As mentioned above, the number of correction terms changes at $\al=3$.
We also note that $\regwp$ defined in \eqref{condA}
changes its value at the same point.
This regularity threshold $\regwp$ arises from applying the modified energy method to equation \eqref{eqvv}.
Since \eqref{eqvv} contains one derivative in the nonlinear term,
the regularity of $v$ must exceed $\frac 32$ to estimate $\| \dx v \|_{L^\infty}$.
Moreover, when computing the time derivative of the modified energy (more precisely, of the correction terms), we rewrite $\dt v$ using \eqref{eqvv}.
At that stage, a derivative of order $\al$ appears
(e.g., we need to estimate $|(v, D^\al v)_{L^2}|$), and the regularity must exceed $\frac \al2$ to distribute this derivative evenly.
Since $v=\dx u$,
the corresponding regularity requirements for $u$ become 
$\frac 52 = \frac 32+1$ and $\frac \al2+1$.

\begin{remark}
\rm

As the primary aim of this paper is to establish the necessary and sufficient condition on $F$ for well-posedness,
we do not attempt to determine the optimal regularity conditions.
In particular, for certain special nonlinear terms, well-posedness can hold under weaker regularity assumptions. Indeed, from the Strichartz estimate, one can easily deduce that the cubic fNLS is well-posed in $L^2(\T)$ for $\al>2$.
Moreover, by renormalizing the cubic nonlinearity,
Brun et al. \cite{BLLZ} proved well-posedness in $H^s(\T)$ for $\al>2$ and $s> \frac{2-\al}6$.

\end{remark}

If \eqref{Fbez} does not hold,
the term $i \int_\T \Im F_\o (u,v, \cj u ,\cj v) dx$ remains
in the first term on the right-hand side of \eqref{eqvv}.
In this case,
the Cauchy-Riemann-type elliptic operator $\dt + i a \dx$ ($a \in \R$) appears in \eqref{eqvv},
suggesting that the Cauchy problem for \eqref{eqvv} may be ill-posed.

The simplest example is a linear case $F(\zeta, \o, \cj \zeta, \cj \zeta) = i \o$.
Then, \eqref{fNLS} becomes
\[
\left\{
\begin{aligned}
&\dt u + i D^\al u = i \dx u, \\ 
&u|_{t=0} = \phi.
\end{aligned}
\right.
\]
If $u \in C([0,T]; H^s(\T))$ for some $s \in \R$ is a solution to this Cauchy problem,
we have
\[
\ft u(t,k) = \ft \phi (k) e^{(-i |k|^\al - k)t}
\]
for $t \in [0,T]$ and $k \in \Z$.
Since $|\ft u(t,k)| = |\ft \phi(k)| e^{-kt}$,
we can choose $\phi$ so that
$|\ft u(t,k)|$ grows exponentially as $k \to -\infty$ when $t>0$.
In that case,
$u(t, \cdot)$ is not a distribution on $\T$,
and hence the solution does not exist.
We note that $\phi$ must be chosen appropriately.
In particular, in Theorem \ref{thm:nonexi}, we can not generally take $\psi = \phi$.
Indeed,
in this example, for any nonzero real number $c$,
the function $\psi(x)=c$ satisfies condition \eqref{Fbez2}, and $u(t,x) = c$ is a solution.

In the nonlinear case, the influence of a similar Cauchy–Riemann-type operator leads to the non-existence of solutions.
Unlike the linear case, however, we can not explicitly express the solution.
Moreover, we must control the derivatives of the nonlinear terms.
Therefore, following \cite{KiTs18, KoOk25b}, we show that Cauchy-Riemann-type operators arise from resonant interactions
and that, if a solution exists, the initial data has regularity, which yields non-existence by choosing appropriate initial data.
We note that the regularity assumption $s>\regwp$ is needed to estimate the nonlinear terms.

In the case $\al=2$ \cite{KoOk25b}, the gauge transformation was also essential in proving non-existence.
By contrast, when $\al>2$, terms except for the problematic resonant terms can be controlled by dispersive effects.
In particular, it is unnecessary to introduce multiple correction terms as in the proof of well-posedness
See Proposition \ref{prop:FGH} and its proof.

The rest of the paper is organized as follows.
In Section \ref{SEC:Bili}, we state the bilinear estimates used in the proof of Theorem \ref{thm:equiv}.
In Section \ref{SEC:wp}, after constructing the modified energy,
we prove the existence of a solution to \eqref{fNLS} and then show the uniqueness of the solution and continuous dependence on initial data, namely, the well-posedness result of Theorem \ref{thm:equiv}.
In Section \ref{SEC:nonex}, we prove the ill-posedness result of Theorem \ref{thm:equiv} by showing Theorem \ref{thm:nonexi}.

We conclude this section with notation.
For an integrable function $f$ on $\T$,
we define
\begin{align*}
\int_\T f dx := \frac 1{2\pi} \int_{-\pi}^\pi f(x) dx,
\\
\ft f(k) := \int_\T f(x) e^{-ikx} dx
\end{align*}
for $k \in \Z$.
We also denote by $\ft f(k)$ of $\Ft [f](k)$ the Fourier coefficient of a distribution $f$ on $\T$.

For $s \in \R$ and a distribution $f$ on $\T$,
we define
\[
\jb{D}^s f (x) := \sum_{k \in \Z} \jb{k}^s \ft f(k) e^{ikx},
\]
where
$\jb{k} := \sqrt{1+|k|^2}$.
We denote $H^s(\T)$ by the $L^2$-based Sobolev space on $\T$
equipped with the norm
\[
\| f \|_{H^s} :=
\bigg( \sum_{k \in \Z} \jb{k}^{2s} |\ft f(k)|^2 \bigg)^{\frac 12}
= \| \jb{D}^s f \|_{L^2}
.
\]
We also denote $\l^2_s$ by the space of square-summable sequences with the weight $\jb{k}^s$,
equipped with the norm
\[
\| a \|_{\l^2_s} :=
\bigg( \sum_{k \in \Z} \jb{k}^{2s} |a(k)|^2 \bigg)^{\frac 12}.
\]
Set $\l^2 := \l^2_0$.
We note that
\[
\| f \|_{H^s} = \| \ft f \|_{\l^2_s}.
\]

For a distribution $f$ on $\T$,
we define
\begin{align*}
P_0 f
&:=
\ft f(0),
&
P_{\neq 0} f
&:=
f - P_0 f,
\\
P_+ f (x)
&:=
\sum_{k \in \N} \ft f(k) e^{ikx},
&
P_- f (x)
&:=
\sum_{k \in \N} \ft f(-k) e^{-ikx}.
\end{align*}
Moreover, we define
\[
\dx^{-1} f (x)
:=
\sum_{k \in \Z \setminus \{0\}} \frac 1{ik} \ft f(k) e^{ikx}.
\]
When $f$ is an integrable function, we have
\begin{equation}
\dx^{-1} f (x)
= \int_\T \int_y^x P_{\neq 0} f(x') dx' dy.
\label{antder1}
\end{equation}

For $T>0$ and a Banach space $X$,
we denote the space of essentially bounded $X$-valued functions by $L^\infty([0,T]; X)$ equipped with the norm
\[
\| u \|_{L_T^\infty X}
:=
\esssup_{t \in [0,T]} 
\| u (t) \|_{X}.
\]

We use $A \les B$ to denote $A \le C B$ with a constant $C>0$.
When the $C>0$ is sufficiently small, we write $A \ll B$.
We also use $A \sim B$ to denote $A \les B$ and $A \ges B$.

\section{Bilinear estimates}
\label{SEC:Bili}

In this section, 
we state some bilinear estimates.
First, we recall the basic estimate.
See Corollary 3.16 on page 855 in \cite{Tao01} for example.
While the bilinear estimate on $\R$ is proved in \cite{Tao01},
it follows from a straightforward modification of the proof.

\begin{proposition}
\label{prop:bili}
Let $s_0,s_1,s_2 \in \R$ satisfy
\[
\min (s_0+s_1, s_1+s_2, s_2+s_0) \ge 0 , \quad
s_0+s_1+s_2 > \frac 12
\]
or
\[
\min (s_0+s_1, s_1+s_2, s_2+s_0) > 0 , \quad
s_0+s_1+s_2 \ge \frac 12.
\]
Then, we have
\[
\| fg \|_{H^{-s_0}}
\les \| f \|_{H^{s_1}} \| g \|_{H^{s_2}}
\]
for $f \in H^{s_1}(\T)$ and $g \in H^{s_2}(\T)$.
\end{proposition}

Moreover, the following bilinear estimate holds.
See Lemma B.1 in \cite{II01} for example.

\begin{proposition}
\label{prop:bili2}
Let $s \ge 0$ and $r>\frac 12$.
Then, we have
\[
\| fg \|_{H^s}
\les \| f \|_{H^s} \| g \|_{H^{r}} + \| f \|_{H^{r}} \| g \|_{H^s}
\]
for any $f,g \in H^s(\T) \cap H^{r}(\T)$.
\end{proposition}

By Proposition \ref{prop:bili2}, we get the following estimate.
\begin{corollary}
\label{cor:bili3}
Let
\[
s> \frac 32,
\quad
r \ge 1.
\]
Then,
we have
\begin{align}
&\|F(f,\dx f, \cj f, \cj{\dx f})\|_{H^{r-1}} 
\le C(\|f\|_{H^s}) (1 + \|\dx f\|_{H^{r-1}}),
\label{bilinn1}
\\
&
\begin{aligned}
&\| F(f,\dx f, \cj f, \cj{\dx f}) - F(g,\dx g, \cj g, \cj{\dx g})\|_{H^{r-1}} 
\\
&\quad \le C(\|f\|_{H^{\max (s,r)}}, \|g\|_{H^{\max (s,r)}}) \|f -g \|_{H^{r}}
\end{aligned}
\label{bilinn2}
\end{align}
for $f,g \in H^s(\T) \cap H^r(\T)$.
\end{corollary}

\begin{proof}
Recall that $F(f,\dx f, \cj f, \cj{\dx f})$ is a polynomial.
By Proposition \ref{prop:bili2} and $s > \frac32$, we have
\[
\begin{aligned}
\|F(f,\dx f, \cj f, \cj{\dx f}) \|_{H^{r-1}} 
&\le C( \|f\|_{H^s})(1 + \|f\|_{H^{r-1}} + \|\dx f\|_{H^{r-1}})
\\
&\le C( \|f\|_{H^s})(1 + \|\dx f\|_{H^{r-1}}).
\end{aligned}
\]
This shows \eqref{bilinn1}.

Next, we consider \eqref{bilinn2}.
Set
\[
\Ta_\zeta [f] := F_\zeta(f,\dx f, \cj f, \cj{\dx f}),
\quad
\Ta_\o [f] := F_\o (f,\dx f, \cj f, \cj{\dx f})
\]
for simplicity.
A direct calculation shows that
\begin{align*}
&F(f,\dx f, \cj f, \cj{\dx f}) - F(g,\dx g, \cj g, \cj{\dx g})
\\
&=
(f-g) \cdot \int_0^1 \Ta_\zeta [\theta f + (1-\theta) g] d\theta
\\
&\quad
+ \cj{(f-g)} \cdot \int_0^1 \Ta_{\cj \zeta} [\theta f + (1-\theta) g] d\theta
\\
&\quad
+
\dx (f-g) \cdot \int_0^1 \Ta_\o [\theta f + (1-\theta) g] d\theta
\\
&\quad
+ \cj{\dx (f-g)} \cdot \int_0^1 \Ta_{\cj \o} [\theta f + (1-\theta) g] d\theta.
\end{align*}
When $r \ge s$,
Proposition \ref{prop:bili2} yields \eqref{bilinn2}.
When $r<s$, Proposition \ref{prop:bili} and \eqref{bilinn1} yield that
\begin{align*}
&\bigg\| \cj{\dx (f-g)} \cdot \int_0^1 \Ta_{\cj \o} [\theta f + (1-\theta) g] d\theta \bigg\|_{H^{r-1}}
\\
&\les
\| f-g \|_{H^r}
\int_0^1
\| \Ta_{\cj \o} [\theta f + (1-\theta) g] \|_{H^{s-1}} d\theta
\\
&\le
C( \| f \|_{H^s}, \| g \|_{H^s})
\| f-g \|_{H^r}.
\end{align*}
Since the remaining parts are similarly handled,
we obtain \eqref{bilinn2}.
\end{proof}

Next, we state some commutator estimates.
Define
\[
[ \jb{D}^s, f] g
:=
\jb{D}^s (fg) - f \jb{D}^s g
\]
for $s \in \R$ and $f,g \in H^s(\T) \cap L^2(\T)$.
This operator is continuously extended as follows:

\begin{proposition}
\label{prop:commC}
Let $s \ge 0$ and $\eps>0$.
Then,
we have
\begin{align*}
\| [ \jb{D}^s, f] \dx g \|_{L^2}
&\les
\| f \|_{H^{\frac 32+\eps}} \| g \|_{H^s}
+ \| f \|_{H^{s}} \| g \|_{H^{\frac 32+\eps}}
\end{align*}
for any $f,g \in H^{\max (s,\frac 32+\eps)}(\T)$.
\end{proposition}

While the proof is a slight modification of that of Proposition 2.5 in \cite{KoOk25b},
we give a proof for completeness.

\begin{proof}
It suffices to prove the estimate for
$f,g \in H^{s+\frac 32+\eps}(\T)$.
A direct calculation shows that
\begin{equation}
\begin{aligned}
&\Ft \big[ \jb{D}^s (f \dx g) - f \jb{D}^s \dx g \big] (k)
\\
&= \sum_{k' \in \Z \setminus \{ 0 \}} \big( \jb{k}^s  - \jb{k'}^s ) ik' \ft f(k-k')\ft g(k') \\
&=
\bigg( \sum_{\substack{k' \in \Z \setminus \{ 0 \} \\ |k'| \ge 2|k|}}
+ \sum_{\substack{k' \in \Z \setminus \{ 0 \} \\ |k'|< \frac{|k|}2}}
+ \sum_{\substack{k' \in \Z \setminus \{ 0 \} \\ \frac{|k|}2 \le |k'|< 2|k|}} \bigg)
\big( \jb{k}^s  - \jb{k'}^s ) ik' \ft f(k-k')\ft g(k')
\\
&=: \I (k) + \II (k) + \III (k)
\end{aligned}
\label{dec1}
\end{equation}
for $k \in \Z$.
We treat these three cases separately.

When $|k'| \ge 2|k|$,
it holds that
$|k-k'| \sim |k'| \ges |k|$.
With $s \ge 0$ and $\eps>0$,
we have
\begin{align*}
|\I (k)|
&\les
\sum_{\substack{k' \in \Z \\ |k'| \ge 2|k|}}
\jb{k-k'} |\ft f(k-k')| \jb{k'}^s |\ft g(k')|\\
&\les
\bigg( \sum_{\substack{k' \in \Z \\ |k'| \ge 2|k|}}
\jb{k-k'}^2 |\ft f(k-k')|^2 \bigg)^{\frac 12}
\| g \|_{H^s}
\\
&\les
\jb{k}^{-\frac 12-\eps}
\| f \|_{H^{\frac 32+\eps}}
\| g \|_{H^s}.
\end{align*}
This yields that
\begin{equation}
\| \I \|_{\l^2}
\les
\| f \|_{H^{\frac 32+\eps}}
\| g \|_{H^s}.
\label{In1}
\end{equation}

When $0<|k'|< \frac{|k|}2$,
it holds that
$|k| \sim |k-k'| \ges |k'|$.
With $s \ge 0$,
we have
\begin{align*}
|\II (k)|
&\les
\sum_{\substack{k' \in \Z \\ |k'| < \frac{|k|}2}}
\jb{k-k'}^{s} |\ft f(k-k')|
\jb{k'}|\ft g(k')|.
\end{align*}
By Minkowski's integral inequality (with counting measures),
we have
\begin{equation}
\begin{aligned}
\| \II \|_{\l^2}
\les
\| f \|_{H^{s}}
\sum_{k' \in \Z} \jb{k'} |\ft g(k')|
\les
\| f \|_{H^{s}} \| g \|_{H^{\frac 32+\eps}}
\end{aligned} 
\label{IIn1}
\end{equation}
for $\eps>0$.

When
$\frac{|k|}2 \le |k'| < 2|k|$,
it holds that
\begin{align*}
\big| \jb{k}^s - \jb{k'}^s \big|
\les \jb{k-k'} \jb{k'}^{s-1}.
\end{align*}
Indeed,
if $|k-k'| \ge \frac{|k'|}2$,
the inequality follows from $|k-k'| \sim |k| \sim |k'|$.
If $|k-k'| < \frac{|k'|}2$,
the fundamental theorem of calculus yields that
\begin{align*}
\big| \jb{k}^s - \jb{k'}^s \big|
&= \bigg| (k-k') \int_0^1 s \jb{k' + \theta (k-k')}^{s-2} (k' + \theta (k-k')) d\theta \bigg|
\\
&\sim |k-k'| \jb{k'}^{s-1}.
\end{align*}
Hence,
with $\wt k = k-k'$,
we have
\begin{align*}
|\III (k)|
&\les
\sum_{\substack{k' \in \Z \\ \frac{|k|}2 \le |k'| < 2|k|}}
\jb{k-k'} |\ft f(k-k')| \jb{k'}^s |\ft g(k')|\\
&=
\sum_{\substack{\wt k \in \Z \\ \frac{|k|}2 \le |k-\wt k| < 2|k|}}
\jb{\wt k} |\ft f(\wt k)| \jb{k-\wt k}^s |\ft g(k-\wt k)|.
\end{align*}
By Minkowski's integral inequality (with counting measures),
we have
\begin{equation}
\| \III \|_{\l^2}
\les
\sum_{\wt k \in \Z}
\jb{\wt k} |\ft f(\wt k)|
\cdot
\| g \|_{H^s}
\les
\| f \|_{H^{\frac 32+\eps}}
\| g \|_{H^s}
\label{IIIn1}
\end{equation}
for $\eps>0$.

From
\eqref{dec1}--\eqref{IIIn1},
we obtain the desired estimate.
\end{proof}

We also need the estimate with negative regularity.

\begin{proposition}
\label{prop:commne}
Let $s < 0$ and $\eps>0$.
Then,
we have
\begin{align*}
\| [\jb{D}^s, f] \dx g \|_{L^2}
&\les
\| f \|_{H^{\frac 32+\eps}}
\| g \|_{L^2}
\end{align*}
for any
$f \in H^{\frac 32+\eps}(\T)$ and $g \in L^2(\T)$.
\end{proposition}

\begin{proof}
It suffices to prove the estimate for
$f,g \in H^{\frac 32+\eps}(\T)$.
We decompose
$\Ft \big[ \jb{D}^s (f \dx g) - f \jb{D}^s \dx g \big] (k)$
into three parts as in \eqref{dec1}.
We treat these three cases separately.

When $|k'| \ge 2|k|$,
it holds that
\[
|k-k'| \sim |k'| \ges |k|.
\]
With $s < 0$ and $\eps>0$,
we have
\begin{align*}
|\I (k)|
&\les
\sum_{\substack{k' \in \Z \\ |k'| \ge 2|k|}}
|k-k'| |\ft f(k-k')|  |\ft g(k')|
\\
&\les
\bigg( \sum_{\substack{k' \in \Z \\ |k'| \ge 2|k|}}
\jb{k-k'}^2 |\ft f(k-k')|^2 \bigg)^{\frac 12}
\| g \|_{L^2}
\\
&\les
\jb{k}^{-\frac 12-\eps}
\| f \|_{H^{\frac 32+\eps}}
\| g \|_{L^2}.
\end{align*}
This yields that
\begin{equation}
\| \I \|_{\l^2}
\les
\| f \|_{H^{\frac 32+\eps}}
\| g \|_{L^2}.
\label{In1n}
\end{equation}

When $0<|k'|< \frac{|k|}2$,
it holds that
\[
|k| \sim |k-k'| \ges |k'|.
\]
It follows from $s<0$ that
\[
|k|^s \les |k'|^s \les 1.
\]
Hence,
we have
\begin{align*}
|\II (k)|
&\les
\sum_{\substack{k' \in \Z \\ |k'| < \frac{|k|}2}}
\jb{k-k'} |\ft f(k-k')| |\ft g(k')|\\
&\les
\bigg( \sum_{\substack{k' \in \Z \\ |k'| < \frac{|k|}2}}
\jb{k-k'}^2 |\ft f(k-k')|^2 \bigg)^{\frac 12}
\| g \|_{L^2}
\\
&\les
\jb{k}^{-\frac 12-\eps}
\| f \|_{H^{\frac 32+\eps}}
\| g \|_{L^2}.
\end{align*}
A similar calculation as in \eqref{In1n} shows that
\begin{equation}
\| \II \|_{\l^2}
\les
\| f \|_{H^{\frac 32+\eps}} \| g \|_{L^2}.
\label{IIn1n}
\end{equation}

Note that
in the proof of Proposition \ref{prop:commC},
we did not use the condition $s \ge 0$ to obtain the estimate $\III$.
Hence, the same calculation works well in this case
and we obtain \eqref{IIIn1} for $\eps>0$.
Hence,
from
\eqref{In1n}, \eqref{IIn1n}, and \eqref{IIIn1},
we obtain the desired estimate.
\end{proof}

We also use a commutator estimate with a higher-order term.
Li \cite{Li19} proved commutator estimates (Theorem 6.2 in \cite{Li19}),
but the form we wish to use in this paper is not presented as such.
Therefore, here we describe a variant of the refinements of the commutator estimate that suits our purpose.

We define operator 
\begin{equation}
\Com{s}{f}(g):= \jb{D}^s(fg) - f \jb{D}^sg + s (\dx f) \jb{D}^{s-2}\dx g
\label{comm}
\end{equation}
for $s>0$ and $f,g \in H^s (\T)$.
This commutator operator
is continuously (uniquely) extended as follows:

\begin{proposition}
\label{prop:comm3}
Let $s > 2$ and $\eps >0$. Then, we have
\[
\| \Com{s}{f} (g) \|_{L^2}
\les
\| f \|_{H^{\max (s,\frac 52+\eps)}} \| g \|_{H^{\min (s-2, \frac 12+\eps)}}
+ \|f \|_{H^{\frac52+\eps}} \|g\|_{H^{s-2}}
\]
for any $f \in H^{\max(s,\frac52+\eps)}(\T)$ and  $g \in H^{s-2}(\T)$.
\end{proposition}

\begin{proof}
It suffices to prove the estimate
for any $f \in H^{s+\frac52+\eps}(\T)$ and  $g \in H^{s}(\T)$.
A direct calculation shows that
\begin{equation}
\begin{aligned}
&\Ft [ \jb{D}^s(f g) - f \jb{D}^s g + s \, \dx f \jb{D}^{s-2} \dx g](k)
\\
&=
\sum_{k' \in \Z}
\Big( \jb{k}^s 
-
\jb{k'}^s
-
s \, (k - k') k' \jb{k'}^{s-2} \Big)
\ft f(k - k') \ft g(k')
\\
&=
\bigg( \sum_{\substack{k' \in \Z \\ |k'| \ge \frac{|k|}2}}
+ \sum_{\substack{k' \in \Z \\ |k'|< \frac{|k|}2}}
\bigg)
\Big( \jb{k}^s 
-
\jb{k'}^s
-
s \, (k - k') k' \jb{k'}^{s-2} \Big)
\\
&\hspace*{50pt}
\times
\ft f(k-k') \ft g(k')
\\
&=:
\I_1(k) + \I_2(k)
\end{aligned}
\label{comft}
\end{equation}
for $ k \in \Z$.
Note that
\begin{equation}
\begin{aligned}
&\jb{k}^s - \jb{k'}^s - s \, (k - k') k' \jb{k'}^{s-2} 
\\
&=
s (k - k')
\int_0^1 \big( (\ta (k-k') + k') \jb{\ta (k-k') + k'}^{s-2} - k' \jb{k'}^{s-2} \big) d\ta
\\
&=
s (k - k')^2 \int_0^1 \ta \int_0^1 \big( 1+ (s-1) (\eta \ta (k-k') +k')^2 \big)
\\
&\hspace*{50pt}
\times
\jb{\eta \ta (k-k') + k'}^{s-4} d\eta  d\ta.
\end{aligned}
\label{conb}
\end{equation}

First, we consider $\I_1(k)$.
With $\wt k = k-k'$,
\eqref{comft}, \eqref{conb}, and $s>2$ yield that
\[
\begin{aligned}
|\I_1(k)|
&\les
\sum_{\substack{\wt k \in \Z \\ |k-\wt k| \ge \frac{|k|}2}}
\jb{\wt k}^2 |\ft f(\wt k)| \jb{k-\wt k}^{s-2} |\ft g(k-\wt k)|.
\end{aligned}
\]
By Minkowski's integral inequality,
we have
\begin{equation}
\| \I_1 \|_{\l^2}
\les
\sum_{\substack{\wt k \in \Z}}
\jb{\wt k}^2 |\ft f(\wt k)|
\| g \|_{H^{s-2}}
\les
\|f \|_{H^{\frac 52+\eps}}
\| g \|_{H^{s-2}}
\label{comIk}
\end{equation}
for $\eps>0$.

Second,
we consider $\I_2(k)$. Note that $\I_2(0) = 0$.
When $|k'| < \frac{|k|}2$, it holds that $|k - k'| \sim |k| \ges |k'|$.
It follows from \eqref{comft} that
\[
\begin{aligned}
|\I_2(k)|
&\les
\sum_{\substack{k' \in \Z \\ |k'| < \frac{|k|}2}}
\jb{k - k'}^s |\ft f(k - k')|
|\ft g(k')|.
\end{aligned}
\]
Minkowski's integral inequality yields that
\begin{equation}
\begin{aligned}
\| \I_2 \|_{\l^2}&\les
\sum_{\substack{k' \in \Z}}
\Big(
\sum_{\substack{ k \in \Z \\ |k|>2|k'|}}
\jb{k - k'}^{2s} |\ft f(k - k')|^2
\Big)^{\frac 12}
|\ft g(k')|
\\
&\les
\| f \|_{H^{\max (s,\frac 52+\eps)}}
\sum_{\substack{k' \in \Z}}
\jb{k'}^{\min (0, s-\frac 52-\eps)}
|\ft g(k')|
\\
&\les
\| f \|_{H^{\max (s,\frac 52+\eps)}}
\| g \|_{H^{\min (s-2, \frac 12+\eps)}}
\end{aligned}
\label{comIIk}
\end{equation}
for $\eps>0$.

Combining \eqref{comft}, \eqref{comIk}, and \eqref{comIIk}, we get the desired bound.
\end{proof}

\begin{remark}
\rm
The estimate
\[
\| \Com{s}{f} (g) \|_{L^2}
\les
(\|f \|_{H^{s}} + \|f \|_{H^{\frac52+\eps}}) \|g\|_{H^{s-2}}
\]
is enough to prove our main result.
Note that ``$\| g \|_{H^{\frac 12+\eps}}$'' does not appear on the right-hand side.
If such a term were present,
an additional condition $\al \ge 3$ would be required in the main result.
See \eqref{remi3} below, for example.
\end{remark}

\section{Well-posedness}
\label{SEC:wp}

In this section,
we prove the well-posedness part of Theorem \ref{thm:equiv} using the modified energy method.
For simplicity, we only consider a solution to \eqref{fNLS} in a positive time interval.
Throughout this section,
we assume that $F$ satisfies \eqref{Fbez}.

\subsection{Energy estimates}
\label{SUBSEC:energy}

In this subsection,
we prove the existence of a solution to \eqref{fNLS} as a limit of the following
the parabolic regularization of \eqref{fNLS}:
\begin{equation} 
\left\{
\begin{aligned}
&\dt u^{\eps} + iD^\al u^{\eps} - \eps \dx^2 u^\eps = F(u^{\eps}, \dx u^{\eps}, \cj {u^{\eps}}, \cj{\dx u^{\eps}}), \\ 
&u^{\eps}(0,x) = \phi(x)
\end{aligned}
\right.
\label{RfNLS}
\end{equation}
for $\eps>0$, $t>0$, and $x \in \T$.
Owing to the parabolic regularization,
the existence of a solution to \eqref{RfNLS} can be established via the contraction mapping argument.

\begin{proposition}
\label{prop:exRfNLS}
Let $\eps>0$ and $s> \frac 32$.
Then,
\eqref{RfNLS} is well-posed in $H^s(\T)$.
In particular,
for any $\phi \in H^s(\T)$,
there exist $T_\eps>0$ and a unique solution $u^\eps \in C([0,T_\eps]; H^s(\T))$ to \eqref{RfNLS}.
Moreover,
the solution depends continuously on the initial data.
\end{proposition}

\begin{proof}
We consider the corresponding integral equation to \eqref{RfNLS}:
\[
u^\eps(t) = U_\eps (t) \phi + \int_0^t U_\eps(t-t') F(u^{\eps}, \dx u^{\eps}, \cj {u^{\eps}}, \cj{\dx u^{\eps}})(t') dt' ,
\]
where $U_\eps(t) = e^{t(-i D^\al + \eps \dx^2)}$.

Let $T > 0$. For $\phi \in H^s(\T)$ and $u^\eps \in C([0,T];H^s(\T))$, we define
\[
\Phi(u^\eps)(t):= U_\eps (t) \phi + \int_0^t U_\eps(t-t') F(u^{\eps}, \dx u^{\eps}, \cj {u^{\eps}}, \cj{\dx u^{\eps}}) (t') dt' .
\]
By a direct calculation, we have
\begin{equation}
\|U_\eps(t) \phi \|_{H^s}
=
\bigg( \sum_{k \in \Z} \jb{k}^{2s} e^{-2\eps tk^2} |\ft \phi |^2 \bigg)^{\frac12}
\le
\| \phi \|_{H^s}.
\label{line}
\end{equation}
Since
\[
\jb{k} e^{-\eps (t-t')k^2}
\le
1 + \frac1{(\eps(t-t'))^{\frac12}},
\]
we have
\begin{equation}
\begin{aligned}
&\bigg\| \int_0^t U_\eps(t-t') F(u^{\eps}, \dx u^{\eps}, \cj {u^{\eps}}, \cj{\dx u^{\eps}}) (t') dt' \bigg\|_{H^s}
\\
&\les
\int_0^t
\Big( 1 + \frac 1{(\eps (t-t'))^{\frac 12}} \Big)
\| F(u^{\eps}, \dx u^{\eps}, \cj {u^{\eps}}, \cj{\dx u^{\eps}}) (t') \|_{H^{s-1}} dt'
\\
&\le
C (\eps)
(t+t^{\frac12}) \sup_{t' \in [0,t]} \| F(u^{\eps}, \dx u^{\eps}, \cj {u^{\eps}}, \cj{\dx u^{\eps}}) (t') \|_{H^{s-1}}.
\end{aligned}
\label{dua}
\end{equation}
From \eqref{line}, \eqref{dua}, and Corollary \ref{cor:bili3}, there exists $C(\eps,\|u^\eps\|_{L_T^\infty H^s}) > 0$ such that
\[
\| \Phi(u^\eps)\|_{L_T^\infty H^s} 
\le
\|\phi\|_{H^s} + C(\eps,\|u^\eps\|_{L_T^\infty H^s})(T + T^{\frac12}).
\]
Set
\[
B_T:= \{ u \in C([0,T];H^s(\T)) \mid \| u \|_{L_T^\infty H^s} \le 2 \| \phi \|_{H^s} \}.
\]
By taking $T_{\eps,1}>0$ as 
\[
T_{\eps,1} + T_{\eps,1}^{\frac12} \le \frac{\|\phi\|_{H^s}}{C(\eps, 2\|\phi\|_{H^s})},
\] 
$\Phi$ is a map on $B_{T_{\eps,1}}$.
A similar calculation as in \eqref{dua} yields that for $u_1^\eps, u_2^\eps \in C([0,T]; H^s(\T))$, there exists $C_0(\eps, \|u_1^\eps\|_{L_T^\infty H^s}, \|u_2^\eps\|_{L_T^\infty H^s}) > 0$ such that
\[
\begin{aligned}
&\| \Phi(u_1^\eps) - \Phi(u_2^\eps) \|_{L_T^\infty H^s}
\\
&\quad
\le
C_0 (\eps, \|u_1^\eps\|_{L_T^\infty H^s}, \|u_2^\eps\|_{L_T^\infty H^s}) (T + T^{\frac12}) \| u_1^\eps - u_2^\eps\|_{L_T^\infty H^s}.
\end{aligned}
\]
We take $T_{\eps,2} >0$ with
\[
T_{\eps,2} + T_{\eps,2}^{\frac12} \le \frac{1}{2 C_0(\eps, 2\|\phi\|_{H^s}, 2\|\phi\|_{H^s})}.
\] 
Set
\[
T_\eps := \min (T_{\eps,1}, T_{\eps,2}).
\] 
Then,
$\Phi$ is a contraction map on $B_{T_\eps}$.
Thus, \eqref{RfNLS} has a solution. 

The uniqueness and the continuous dependence
follow from a standard argument.
We thus omit the details here.
\end{proof}

In this subsection,
we assume that
$s>\frac 32$ and $\phi \in H^s(\T)$.
Let $u^\eps$ be the solution to \eqref{RfNLS} obtained in Proposition \ref{prop:exRfNLS}.
Since $F$ is a polynomial, the solution $u^\eps$ belongs to $C([0,T_\eps];H^s(\T)) \cap C^\infty((0,T_\eps] \times \T)$.

Set
\[
v^{\eps}
=
\dx u^{\eps}.
\]
For simplicity, we write
\begin{align*}
\Ta^{\eps} &= F(u^{\eps},v^{\eps}, \cj {u^{\eps}} ,\cj {v^{\eps}}), &
\Ta_{\zeta}^{\eps} &= F_{\zeta} (u^{\eps},v^{\eps}, \cj {u^{\eps}} ,\cj {v^{\eps}}), &
& \ \dots.
\end{align*}
Then,
$u^\eps$ satisfies
\begin{equation}
\dt u^{\eps} + i D^\al u^\eps -\eps \dx^2 u^\eps
= \Ta^\eps .
\label{ueps}
\end{equation}
By the assumption \eqref{Fbez}, 
$v^{\eps}$ satisfies
\begin{equation}
\begin{aligned}
&\dt v^{\eps} + iD^\al v^\eps - \eps \dx^2 v^\eps
\\
&=
i (P_{\neq 0} \Im \Ta^{\eps}_{\o}) \dx v^\eps
+(\Re \Ta^\eps_\o) \dx v^\eps
+\Ta^{\eps}_{\cj \o} \cj{\dx v^{\eps}}
+\Rl^{\eps},
\end{aligned}
\label{veps}
\end{equation}
where 
\[
\Rl^\eps = \Ta_{\zeta}^\eps v^\eps
+\Ta_{\cj \zeta}^\eps \cj{v^\eps}. 
\] 
Note that $\Rl^\eps$ is the polynomial in $u^\eps, v^\eps, \cj{u^\eps}, \cj{v^\eps}$.

In what follows,
we assume that $r \ge 1$ and $\phi \in H^s(\T) \cap H^r(\T)$.%
\footnote{Since the proof of Proposition \ref{prop:exRfNLS} is a standard contraction argument,
the persistence of regularity holds.
In particular,
$T_\eps$ is independent of $\| \phi \|_{H^r}$ and
the solution $u^\eps$ belongs to $C([0,T_\eps]; H^r(\T))$
even when $r>s$ and $\phi \in H^r(\T)$.}
Let
$\al >2$ and $\eps \in (0,1)$.
For
$n \in \N$ and $t \in [0, T_\eps]$,
we define the correction term
\begin{align}
L_{n,\eps}^r(t) 
&:=
c_n \int_\T \big( \dx^{-1} \Im \Ta_\o^\eps (t) \big)^n
\big| \jb{D}^{r-1-\frac {(\al-2)n}2} v^\eps(t) \big|^2 dx,
\label{eq:Ln}
\end{align}
where
\begin{equation}
c_n:= \frac{2^n}{\al^n n!}.
\label{eq:Lncon1}
\end{equation}
We have the following bound.

\begin{lemma}
\label{lem:ene1L}
Let $\al >2$ and $N \in \N$ satisfy
\begin{align}
N - 1 < \frac 1{\al-2} \le N.
\label{conN}
\end{align}
Moreover, let $n \in \{ 1, \dots, N \}$.
Then,
there exists $a_n>0$
such that
\[
\begin{aligned}
|L_{n,\eps}^r (t)|
&\le
\frac 1{10 n^2} \| v^\eps (t) \|_{H^{r-1}}^2 
+
a_n \| u^\eps (t) \|_{H^{r-1}}^2 \| \Ta_\o^\eps (t) \|_{L^2}^{2N}
\end{aligned}
\]
for $r \ge 1$, $\eps \in (0,1)$, and $ t \in [0,T_\eps]$.
\end{lemma}

\begin{proof}
For simplicity, we suppress the time dependence in this proof.
Let $n \in \{ 1,2, \dots, N \}$.
An interpolation argument with \eqref{conN} shows that
\[
\begin{aligned}
\big\| \jb{D}^{r-1-\frac {(\al-2)n}2} v^\eps \big\|_{L^2}
&\le
\big\| \jb{D}^{r-1-\frac {(\al-2)N}2} v^\eps \big\|_{L^2}^{\frac nN} \| \jb{D}^{r-1} v^\eps \|_{L^2}^{1-\frac nN}
\\
&\le \| v^\eps \|_{H^{r- \frac 32}}^{\frac nN} \| v^\eps \|_{H^{r-1}}^{1-\frac nN}
\\
&\le \| u^\eps \|_{H^{r-1}}^{\frac n{2N}} \| v^\eps \|_{H^{r-1}}^{1-\frac n{2N}}.
\end{aligned}
\]
By \eqref{eq:Ln}, Proposition \ref{prop:bili}, and Young's inequality,
there exists $a_n >0$ such that
\begin{equation}
\begin{aligned}
|L_{n,\eps}^r|
&\le
c_n \| \dx^{-1} \Im \Ta_\o^\eps \|_{L^\infty}^n
\big\| \jb{D}^{r-1-\frac {(\al-2)n}2} v^\eps \big\|_{L^2}^2
\\
&\le c_n
\| \dx^{-1} \Im \Ta_\o^\eps \|_{L^\infty}^n \| u^\eps \|_{H^{r-1}}^{\frac nN} 
\| v^\eps \|_{H^{r-1}}^{2(1-\frac n{2N})}
\\
&\le
\frac1{10 n^2} \| v^\eps\|_{H^{r-1}}^2 + a_n \| u^\eps\|_{H^{r-1}}^2 \| \dx^{-1} \Im \Ta_\o^\eps \|_{L^\infty}^{2N}.
\end{aligned}
\label{est:Lnes}
\end{equation}
It follows from \eqref{antder1} that
\[
\| \dx^{-1} \Im \Ta_\o^\eps \|_{L^\infty}
\les
\| \Ta_\o^\eps \|_{L^1}
\les
\| \Ta_\o^\eps \|_{L^2}.
\]
Hence,
we obtain the desired bound.
\end{proof}

We define the modified energy.

\begin{definition}
\label{def:modE}
Let $\al >2$, $r \ge 1$, and $\eps \in (0,1)$.
Set
\[
a := \sum_{n=1}^N a_n,
\]
where $a_n >0$ is the constant appeared in Lemma \ref{lem:ene1L}.
Then, we define 
\begin{equation}
\begin{aligned}
E_\eps^r (t)
&:= \bigg( \|u^\eps(t)\|_{H^{r-1}}^2 + \| v^\eps(t)\|_{H^{r-1}}^2
\\
&\qquad
+
\sum_{n=1}^N
L_{n,\eps}^r (t) + a \| u^\eps(t)\|_{H^{r-1}}^2 \| \Ta_\o^\eps(t) \|_{L^2}^{2N}
\bigg)^{\frac12}
\end{aligned}
\label{modene1}
\end{equation}
for $t \in [0,T_\eps]$,
where $L_{n,\eps}^r (t)$ is given in \eqref{eq:Ln}.
\end{definition}

Because of the presence of the lower order terms,
the modified energy $E^r_\eps (t)$ is coercive (without a smallness of $u^\eps$).
In fact,
since
\[
\sum_{n=1}^\infty \frac 1{10 n^2}
= \frac 1{10} \frac{\pi^2}6 < \frac 12,
\]
Lemma \ref{lem:ene1L} and Definition \ref{def:modE}
yield that
\begin{equation}
\begin{aligned}
&
\|u^\eps(t)\|_{H^{r-1}}^2 + \frac12  \| v^\eps(t)\|_{H^{r-1}}^2
\\
&
\le
E_\eps^r(t)^2
\\
&
\le
\|u^\eps(t)\|_{H^{r-1}}^2 + \frac 32  \| v^\eps(t)\|_{H^{r-1}}^2
+
2a
\| u^\eps(t)\|_{H^{r-1}}^2
\| \Ta_\o^\eps(t) \|_{L^2}^{2N}
\end{aligned}
\label{ene1}
\end{equation}
for $t \in [0,T_\eps]$.

\begin{remark}
\label{rem:gemod}
\rm
For $\al=2$,
the correction terms in \eqref{modene1} correspond to that in \cite{KoOk25b}.
Indeed,
with \eqref{eq:Ln} and \eqref{eq:Lncon1},
we formally take $\al=2$ and $N=\infty$ in \eqref{modene1} to obtain the following:
\begin{align*}
\| v^\eps(t)\|_{H^{r-1}}^2
+
\sum_{n=1}^\infty
L_{n,\eps}^s (t)
&=
\int_\T \sum_{n=0}^\infty \frac 1{n!} \big( \dx^{-1} \Im \Ta_\o^\eps (t) \big)^n
\big| \jb{D}^{r-1} v^\eps(t) \big|^2 dx
\\
&=
\int_\T
\exp \big( \dx^{-1} \Im \Ta_\o^\eps (t) \big)
\big| \jb{D}^{r-1} v^\eps(t) \big|^2 
dx
\\
&=
\Big\|
\exp \Big( -\frac i2 \dx^{-1} \Ta_\o^\eps (t) \Big)
\jb{D}^{r-1} v^\eps(t) \Big\|_{L^2}^2 
\\
&\sim
\Big\|
\exp \Big( -\frac i2 \dx^{-1} \Ta_\o^\eps (t) \Big)
v^\eps(t) \Big\|_{H^{r-1}}^2
.
\end{align*}
Here,
\[
\exp \Big( -\frac i2 \dx^{-1} \Ta_\o^\eps (t) \Big)
v^\eps(t)
\]
corresponds to the gauge transformation (essentially) used in \cite{KoOk25b}.
From this consideration,
gauge transformations can be regarded as a modified energy method that adds an infinite number of correction terms.
\end{remark}

We prove the following energy estimate.

\begin{proposition}
\label{prop:enineq}
Let $\al > 2$, $\eps \in (0,1)$, and let $s,r \in \R$ satisfy
\[
s > \regwp,
\quad
r \ge 1,
\]
where $\regwp$ is defined in \eqref{condA}.
Then, there exists $C_1(E_\eps^s(t))>0$ (which is independent of $\eps$) such that 
\[
\dt E_\eps^r(t) \le C_1(E_\eps^s(t)) (1 + E_\eps^r(t)) 
\]
for $t \in (0, T_\eps)$.
\end{proposition}

To prove Proposition \ref{prop:enineq}, we prepare two lemmas.
First, we give a simple estimate.

\begin{lemma}
\label{lem:noli}
Let $\al>2$,
$s> \frac \al 2+1$,
and $\eps \in (0,1)$.
Then,
we have
\[
\| \dt \Ta^\eps (t)\|_{H^{s-\al-1}}
\le C(\|u^\eps (t)\|_{H^s}).
\]
for $t \in (0,T_\eps)$.
\end{lemma}

\begin{proof}
Proposition \ref{prop:bili} with $s>\frac 32$ and $\al>1$
yields
\[
\| \Ta^\eps (t) \dx v^\eps (t) \|_{H^{s-\al-1}}
\les
\| \Ta^\eps (t) \|_{H^{s-1}} \| \dx v^\eps (t) \|_{H^{s-2}}
\le
C(\|u^\eps (t) \|_{H^s}).
\]
Hence,
by \eqref{veps} and $\al>2$,
we have
\begin{equation}
\begin{aligned}
\|\dt v^\eps (t) \|_{H^{s-\al-1}} 
&\le
\| v^\eps (t) \|_{H^{s-1}}
+ \eps \| v^\eps (t) \|_{H^{s-\al+1}}
\\
&\quad
+ \| (P_{\neq 0} \Im \Ta^\eps_\o) (t) \dx v^\eps (t) \|_{H^{s-\al-1}}
\\
&\quad
+ \| (\Re \Ta^\eps_\o) (t) \dx v^\eps (t) \|_{H^{s-\al-1}}
\\
&\quad
+ \| \Ta^\eps_{\cj \o} (t) \cj{\dx v^\eps (t)} \|_{H^{s-\al-1}}
\\
&\quad
+ \| \Rl^{\eps} (t) \|_{H^{s-\al-1}}
\\
&\le
C(\|u^\eps (t) \|_{H^s}).
\end{aligned}
\label{dtuv2}
\end{equation}
A direct calculation yields that
\begin{align*}
\dt \Ta^\eps(t)
&=
\Ta_\o^\eps(t) \dt v^\eps(t)
+ \Ta_{\cj \o}^\eps(t) \dt \cj{v^\eps(t)}
+ \Ta_{\zeta}^\eps(t) \dt u^\eps(t) + \Ta_{\cj \zeta}^\eps(t) \dt \cj{u^\eps(t)}.
\end{align*}
Proposition \ref{prop:bili} with \eqref{dtuv2}, $\al>2$, and $s> \frac \al 2+1$
yields that
\[
\| \Ta_\o^\eps(t) \dt v^\eps(t) \|_{H^{s-\al-1}}
\les
\| \Ta_\o^\eps(t) \|_{H^{s-1}} \| \dt v^\eps(t) \|_{H^{s-\al-1}}
\le
C(\|u^\eps (t) \|_{H^s}).
\]
Since the remaining parts are similarly handled,
we obtain the desired bound.
\end{proof}

The proof of the following lemma is somewhat lengthy,
so we will first state the claim and provide the proof in the next subsection.

\begin{lemma}
\label{lem:exdtwLn}
Let $\al > 2$ and $\eps \in (0,1)$.
Suppose that
\[
s >\regwp,
\quad
r \ge 1,
\]
where $\regwp$ is defined in \eqref{condA}.
Then, there exists $C'(E_\eps^s(t)) > 0$ such that
\[
\dt  \Big( \| v^\eps(t) \|_{H^{r-1}}^2 + \sum_{n=1}^N L_{n,\eps}^r(t) \Big)
\le
C'(E_\eps^s(t))(1+E_\eps^r(t)) \|v^\eps (t)\|_{H^{r-1}}
\]
for any $t \in (0, T_\eps)$.
\end{lemma}

By using Lemmas \ref{lem:noli} and \ref{lem:exdtwLn},
we prove Proposition \ref{prop:enineq}.

\begin{proof}[Proof of Proposition \ref{prop:enineq}]
For simplicity, we suppress the time dependence in this proof. 
From \eqref{ueps} and Corollary \ref{cor:bili3},
we have
\begin{equation}
\begin{aligned}
\frac12 \dt \big( \|u^\eps\|_{H^{r-1}}^2 \big)
&=
\Re(\dt u^\eps, u^\eps)_{H^{r-1}}
\\
&=
-\eps \|\dx u^\eps\|_{H^{r-1}}^2 + \Re( \Ta^\eps, u^\eps)_{H^{r-1}}
\\
&\le \| \Ta^\eps \|_{H^{r-1}} \|u^\eps\|_{H^{r-1}}
\\
&\le C( \| u^\eps \|_{H^s})(1 + \| u^\eps \|_{H^r}) \|u^\eps\|_{H^{r-1}}
\end{aligned}
\label{exdtu}
\end{equation}
for $s>\frac 32$ and $r \ge 1$.
Moreover,
Lemma \ref{lem:noli} and Corollary \ref{cor:bili3} yield that
\begin{align*}
\frac 12 \dt  \big( \| \Ta_\o^\eps \|_{L^2}^2 \big)
=
\Re ( \dt \Ta_\o^\eps, \Ta_\o^\eps)_{L^2}
\le
\| \dt \Ta_\o^\eps \|_{H^{-\frac \al 2}}
\| \Ta_\o^\eps \|_{H^{\frac \al2}}
\le
C( \| u^\eps \|_{H^s})
\end{align*}
for $s> \regwp$.
Hence,
we obtain that
\begin{equation}
\begin{aligned}
\dt  \big(\|  u^\eps\|_{H^{r-1}}^2 \| \Ta_\o^\eps \|_{L^2}^{2N} \big)
&=
\dt (\| u^\eps\|_{H^{r-1}}^2)
\| \Ta_\o^\eps \|_{L^2}^{2N}
\\
&\quad
+
\| u^\eps\|_{H^{r-1}}^2 
N
\| \Ta_\o^\eps \|_{L^2}^{2(N-1)}
\dt ( \| \Ta_\o^\eps \|_{L^2}^2)
\\
&\le
C( \| u^\eps \|_{H^s})(1 + \| v^\eps \|_{H^{r-1}}) \|v^\eps\|_{H^{r-1}}
\end{aligned}
\label{exdtmod}
\end{equation}
for $s> \regwp$ and $r \ge 1$.

From \eqref{modene1}, \eqref{exdtu}, \eqref{exdtmod}, \eqref{ene1}, and Lemma  \ref{lem:exdtwLn},
we get
\[
\dt \big( (E_\eps^r)^2 \big) \le C(E_\eps^s)(1 + E_\eps^r) E_\eps^r,
\] 
which leads the desired bound.
\end{proof}

We prove the existence of a solution to \eqref{fNLS}.
Let $s>\regwp$, $r=s$, and $\phi \in H^s(\T)$.
It follows from \eqref{modene1} that
\[
K
:= E_s^\eps(0)
\]
is independent of $\eps>0$.
Define
\begin{equation}
T := \frac 1{C_1(2K)+1} \log \frac{1+2K}{1+K},
\label{time1}
\end{equation}
where $C_1$ is the continuous function appeared in Proposition \ref{prop:enineq}.
Note that $T$ is independent of $\eps>0$.

Set
\[
T_{\eps}^{\ast}
:=
\sup \{ t >0 \mid E_s^\eps(\tau)\le 2K \text{ for } \tau \in [0, t ]\}.
\]
Then, we have
\[
T_{\eps}^{\ast}\ge T.
\]
Indeed,
if $T_{\eps}^{\ast}<T$,
there exists $t_0 \in (0, T_{\eps}^{\ast})$ such that $E^\eps_s(t_0) = 2K$.
It follows from Proposition \ref{prop:enineq} with $r=s$ and \eqref{time1} that
\[
E_\eps^s(t_0)
\le
( 1 + E_\eps^s(0)) e^{C_1(E^s_\eps(0)) t_0} -1
\le
( 1 + K) e^{C_1(2K) T} -1
<2K.
\]
This contradicts to $E_\eps^s (t_0)=2K$.

From Proposition \ref{prop:enineq},
we also have the following.

\begin{corollary}
\label{cor:enebdur}
In addition to the assumption of Proposition \ref{prop:enineq},
we further assume that $r \ge s$ and $\phi \in H^r(\T)$.
Then,
there exists $C_2 (\| \phi \|_{H^s})>0$ such that
\[
\| u^\eps \|_{L_T^\infty H^r}
\le C_2 (\| \phi \|_{H^s}) (1+\| \phi \|_{H^r}),
\]
for $\eps \in (0,1)$,
where $T$ is defined in \eqref{time1}
.
\end{corollary}

By Corollary \ref{cor:enebdur} with $r=s$,
there exists $u \in L^{\infty}([0,T];H^s(\T))$
and a sequence $\{ \eps_j \}_{j \in \N} \subset (0,1)$ such that
\[
\lim_{j \to \infty} \eps_j =0
\]
and $\{ u^{\eps_j} \}_{j \in \N}$ converges weak$\ast$ to $u$ in $L^\infty ([0,T]; H^s(\T))$.
It follows that $u$ is a solution to \eqref{fNLS} in the sense of Definition \ref{def:sol}.
Furthermore,
if $\phi \in H^r(\T)$ with $r \ge s$,
then there exists a further subsequence $\{ \eps_j' \}_{j \in \N} \subset \{ \eps_j \}_{j \in \N}$ such that$\{ u^{\eps_j'} \}_{j \in \N}$ converges weak$\ast$ to $u$ in $L^\infty ([0,T]; H^r(\T))$.
In particular,
it holds that $u \in L^\infty([0,T]; H^r(\T))$.

\subsection{Proof of Lemma \ref{lem:exdtwLn}}

In this subsection,
we prove Lemma \ref{lem:exdtwLn}.
First, we define
\begin{equation}
\begin{aligned}
K_{n,\eps}^r(t)
&:=
- \al c_n \Im \int_\T
\dx \big( \big( \dx^{-1} \Im \Ta_\o^\eps (t) \big)^n \big)
\\
&\quad
\times \big( \jb{D}^{r-1-\frac{(\al-2)(n-1)}2} \dx v^\eps(t) \big) 
\jb{D}^{r-1-\frac {(\al-2)(n-1)}2} \cj{v^\eps(t)} dx
\end{aligned}
\label{eq:Kn}
\end{equation}
for $r \ge 1$, $n \in \N$,
where $c_n$ is defined in \eqref{eq:Lncon1}.
We show the following lemma.

\begin{lemma}
\label{lem:modi}
With the same assumption as in Lemma \ref{lem:exdtwLn},
we have
\[
\begin{aligned}
&K_{n,\eps}^r(t) + \dt L_{n,\eps}^r(t) - K_{n+1,\eps}^r(t)
\\
&\le
C(E_\eps^s(t))(1 + E_\eps^r(t))
\|v^\eps(t)\|_{H^{r-1}}
+ \frac \eps{2N} \| \dx v^\eps(t) \|_{H^{r-1}}^2
\end{aligned}
\]
for any $n = 1, \cdots, N$ and $t \in (0,T_\eps)$,
where $N$ is defined in \eqref{conN}.
\end{lemma}

\begin{proof}
For simplicity, we suppress the time dependence in this proof.
Fix $n =1,\dots,N$.
It follows from \eqref{eq:Ln} that
\begin{equation}
\begin{aligned}
&\dt L_{n,\eps}^r
\\
&=
2 c_n \Re \int_\T ( \dx^{-1} \Im \Ta_\o^\eps )^n
\big( \jb{D}^{r-1-\frac {(\al-2)n}2} v^\eps \big) \jb{D}^{r-1-\frac {(\al-2)n}2} \cj{\dt v^\eps} dx
\\
&\quad
+
c_n \int_\T \dt ( ( \dx^{-1} \Im \Ta_\o^\eps )^n )
\big| \jb{D}^{r-1-\frac {(\al-2)n}2} v^\eps \big|^2 dx
\\
&=:
J^r + R_1^r.
\end{aligned}
\label{dtL}
\end{equation}

First,
we show that
\begin{equation}
|R_1^r| 
\le
C( \| u^\eps \|_{H^s} ) \| v^\eps\|_{H^{r-1}}^2
\label{remi1a}
\end{equation}
for $s>\regwp$.
When $\al \in (2, 3]$,
Proposition \ref{prop:bili} yields that
\begin{align*}
|R_1^r| 
&\les
\| \dt ((\dx^{-1} \Im \Ta_\o^\eps)^n) \|_{H^{-\al+\frac 52+\dl}} \big\| |\jb{D}^{r-1-\frac {(\al-2)n}2} v^\eps |^2 \big\|_{H^{\al-\frac 52 -\dl}}
\\
&\les
\| \dt ((\dx^{-1} \Im \Ta_\o^\eps)^n) \|_{H^{-\al+\frac 52+\dl}}
\| \jb{D}^{r-1-\frac {(\al-2)n}2} v^\eps\|_{H^{\frac {\al-2}2}}^2
\\
&\les
\| \dt ((\dx^{-1} \Im \Ta_\o^\eps)^n) \|_{H^{-\al+\frac 52+\dl}}
\| v^\eps\|_{H^{r-1}}^2
\end{align*}
for $0<\dl \ll 1$.
Moreover,
Proposition \ref{prop:bili} and Lemma \ref{lem:noli} imply that
\begin{align*}
&\| \dt ((\dx^{-1} \Im \Ta_\o^\eps)^n) \|_{H^{-\al+\frac 52+\dl}}
\\
&\les
\| \dt \dx^{-1} \Im \Ta_\o^\eps \|_{H^{-\al+\frac 52+\dl}} \| (\dx^{-1} \Im \Ta_\o^\eps)^{n-1} \|_{H^{\frac 12+\dl}}
\\
&\le
C( \| u^\eps \|_{H^s} )
\end{align*}
for $s>\regwp$ and $0<\dl \ll 1$.
Hence, we have \eqref{remi1a} for $\al \in (2, 3]$ and $s>\regwp$.

When $\al>3$, 
Proposition \ref{prop:bili} yields that
\begin{align*}
|R_1^r|
&\les
\| \dt ((\dx^{-1} \Im \Ta_\o^\eps)^n) \|_{H^{-\frac \al2 +1+\dl}} 
\big \| |\jb{D}^{r-1-\frac {(\al-2)n}2} v^\eps|^2 \big \|_{H^{\frac \al2 -1-\dl}}
\\
&\les
\| \dt ((\dx^{-1} \Im \Ta_\o^\eps)^n) \|_{H^{-\frac \al2 +1+\dl}} 
\| \jb{D}^{r-1-\frac {(\al-2)n}2} v^\eps\|_{H^{\frac {\al-2}2}}^2
\end{align*}
for $0<\dl \ll 1$.
Moreover,
Proposition \ref{prop:bili} and
Lemma \ref{lem:noli} imply that
\begin{align*}
\| \dt ((\dx^{-1} \Im \Ta_\o^\eps)^n) \|_{H^{-\frac \al2 +1+\dl}}
&\les
\| \dt \dx^{-1} \Ta_\o^\eps \|_{H^{-\frac \al2 +1+\dl}} \| (\dx^{-1} \Ta_\o^\eps)^{n-1} \|_{H^{\frac \al 2-1}}
\\
&\le
C( \| u^\eps \|_{H^s} )
\end{align*}
for $s>\regwp$.
Hence,
we  have \eqref{remi1a} for $\al>3$ and $s>\regwp$.

Second, we consider $J^r$ in \eqref{dtL}.
From \eqref{veps}, $\cj v^\eps$ satisfies
\[
\dt \cj {v^\eps}
=
i \cj{D^\al v^\eps} + \eps \cj{\dx^2 v^\eps} - i (P_{\neq 0} \Im \Ta_\o^\eps) \cj {\dx v^\eps}
+ (\Re \Ta_\o^\eps) \cj {\dx v^\eps} + \cj{\Ta_{\cj \o}^\eps} \dx v^\eps + \cj{\Rl^\eps}.
\]
Hence, we have
\begin{align}
J^r
&=
-2 c_n \Im \int_\T (\dx^{-1}  \Im \Ta_\o^\eps)^n \big( \jb{D}^{r-1-\frac {(\al-2)n}2} v^\eps \big) \jb{D}^{r-1-\frac {(\al-2)n}2} D^\al \cj{v^\eps}dx
\notag
\\
&\quad
+
2c_n \eps \Re \int_\T (\dx^{-1}  \Im \Ta_\o^\eps)^n  \big( \jb{D}^{r-1-\frac {(\al-2)n}2} v^\eps \big) \jb{D}^{r-1-\frac {(\al-2)n}2} \cj{\dx^2 v^\eps}dx
\notag
\\
&\quad
+
2c_n \Im \int_\T (\dx^{-1}  \Im \Ta_\o^\eps)^n  (P_{\neq 0} \Im \Ta_\o^\eps) \big( \jb{D}^{r-1-\frac {(\al-2)n}2} v^\eps \big)
\notag
\\
&\hspace*{80pt}
\times
\jb{D}^{r-1-\frac {(\al-2)n}2} \cj{\dx v^\eps}dx
\notag
\\
&\quad
+
2c_n \Re \int_\T (\dx^{-1}  \Im \Ta_\o^\eps)^n (\Re \Ta_\o^\eps) \big( \jb{D}^{r-1-\frac {(\al-2)n}2} v^\eps \big)
\notag
\\
&\hspace*{80pt}
\times
\jb{D}^{r-1-\frac {(\al-2)n}2} \cj{\dx v^\eps}dx
\notag
\\
&\quad
+
2c_n \Re \int_\T (\dx^{-1}  \Im \Ta_\o^\eps)^n \cj{\Ta_{\cj \o}^\eps} \big( \jb{D}^{r-1-\frac {(\al-2)n}2} v^\eps \big) \jb{D}^{r-1-\frac {(\al-2)n}2} \dx v^\eps dx
\notag
\\
&\quad
+
R_2^r
\notag
\\
&=:
\sum_{\l=1}^5 J_{\l}^r + R_2^r,
\label{Jr}
\end{align}
where
\begin{equation*}
\begin{aligned}
R_2^r
&:=
2c_n \Im \int_\T (\dx^{-1}  \Im \Ta_\o^\eps)^n \big( \jb{D}^{r-1-\frac {(\al-2)n}2} v^\eps \big) 
\\
&\hspace*{60pt}
\times
\big[\jb{D}^{r-1-\frac {(\al-2)n}2},  \Im \Ta_\o^\eps \big]  \cj{\dx v^\eps}dx
\\
&\quad
+
2c_n \Re \int_\T (\dx^{-1}  \Im \Ta_\o^\eps)^n \big( \jb{D}^{r-1-\frac {(\al-2)n}2} v^\eps \big)
\\
&\hspace*{60pt}
\times
\big[\jb{D}^{r-1-\frac {(\al-2)n}2}, \Re \Ta_\o^\eps \big]\cj{\dx v^\eps}dx
\\
&\quad
+
2c_n \Re \int_\T (\dx^{-1}  \Im \Ta_\o^\eps)^n \big( \jb{D}^{r-1-\frac {(\al-2)n}2} v^\eps \big) 
\big[\jb{D}^{r-1-\frac {(\al-2)n}2}, \cj{\Ta_{\cj \o}^\eps} \big] \cj {\dx v^\eps}dx
\\
&\quad
+
2c_n \Re \int_\T (\dx^{-1}  \Im \Ta_\o^\eps)^n \big( \jb{D}^{r-1-\frac {(\al-2)n}2} v^\eps \big) 
\jb{D}^{r-1-\frac {(\al-2)n}2} \cj{\Rl^\eps}dx.
\end{aligned}
\end{equation*}

Third,
we estimate $R_2^r$. We rewrite $R_2^r$ as follows:
\[
R_2^r =: R_{2,1}^r + R_{2,2}^r + R_{2,3}^r + R_{2,4}^r.
\]
Note that
we have
\begin{equation}
\big\| \big[ \jb{D}^{r-1-\frac {(\al-2)n}2},  \Im \Ta_\o^\eps \big] \dx v^\eps \big\|_{L^2}
\le
C( \| u^\eps \|_{H^s})
(1+\|v^\eps\|_{H^{r-1}})
\label{est:commC}
\end{equation}
for $s> \frac 52$ and $r \ge 1$.
Indeed,
when $r \ge \frac {(\al-2)n}2+1$,
Proposition \ref{prop:commC} and Corollary \ref{cor:bili3} yield that
\begin{align*}
&\big\| \big[ \jb{D}^{r-1-\frac {(\al-2)n}2},  \Im \Ta_\o^\eps \big] \dx v^\eps \big\|_{L^2}
\\
&\les
\| \Ta_\o^\eps \|_{H^{s-1}} \|v^\eps\|_{H^{r-1}} + \|\Ta_\o^\eps\|_{H^{r-1}} \|v^\eps\|_{H^{s-1}}
\\
&\le
C( \| u^\eps \|_{H^s})
(1+\|v^\eps\|_{H^{r-1}})
\end{align*}
for $s> \frac 52$.
When $1 \le r < \frac {(\al-2)n}2+1$,
Proposition \ref{prop:commne} and Corollary \ref{cor:bili3} yield that
\[
\big\| \big[ \jb{D}^{r-1-\frac {(\al-2)n}2},  \Im \Ta_\o^\eps \big] \dx v^\eps \big\|_{L^2}
\les
\| \Ta_\o^\eps \|_{H^{s-1}} \|v^\eps\|_{L^2}
\le
C( \| u^\eps \|_{H^s})
\|v^\eps\|_{H^{r-1}}
\]
for $s> \frac 52$.

With \eqref{est:commC},
Corollary \ref{cor:bili3} yields that
\begin{equation}
\begin{aligned}
|R_{2,1}^r|
&\le
2c_n
\| (\dx^{-1}  \Im \Ta_\o^\eps)^n \|_{L^\infty}
\| \jb{D}^{r-1} v^\eps\|_{L^2}
\\
&\qquad \times
\big\| \big[\jb{D}^{r-1-\frac {(\al-2)n}2},  \Im \Ta_\o^\eps \big] \cj{\dx v^\eps} \big\|_{L^2}
\\
&\le
C( \| u^\eps \|_{H^s})
(1+\|v^\eps\|_{H^{r-1}})
\|v^\eps\|_{H^{r-1}}
\label{remi21}
\end{aligned}
\end{equation}
for $s>\frac 52$ and $r \ge 1$.
The same calculation as in \eqref{remi21} yields that
\begin{equation}
|R_{2,2}^r| + |R_{2,3}^r|
\le
C( \| u^\eps \|_{H^s})
(1+\|v^\eps\|_{H^{r-1}})\|v^\eps\|_{H^{r-1}}.
\label{remi2223}
\end{equation}
From Corollary \ref{cor:bili3},
we have
\begin{equation}
\begin{aligned}
|R_{2,4}^r| 
&\le 
\| (\dx^{-1}  \Im \Ta_\o^\eps)^n \|_{L^\infty} \|v^\eps\|_{H^{r-1}} \|\Rl^\eps\|_{H^{r-1}}
\\
&\le
C( \| u^\eps \|_{H^s}) (1 + \|v^\eps\|_{H^{r-1}}) \|v^\eps\|_{H^{r-1}}.
\end{aligned}
\label{remi24}
\end{equation}
It follows from \eqref{remi21}--\eqref{remi24} that
\begin{equation}
|R_2^r|
\le
C( \| u^\eps \|_{H^s}) (1 + \|v^\eps\|_{H^{r-1}}) \|v^\eps\|_{H^{r-1}}
\label{remi2}
\end{equation}
for $s> \frac 52$ and $r \ge 1$.

Fourth,
we consider $J_1^r$ in \eqref{Jr}.
It follows from \eqref{Jr} that
\begin{equation}
\begin{aligned}
J_1^r
&=
-2 c_n \Im \int_\T (\dx^{-1}  \Im \Ta_\o^\eps)^n \big( \jb{D}^{r-1-\frac {(\al-2)n}2} v^\eps \big)
\\
&\hspace*{60pt}
\times
\jb{D}^{r-1-\frac {(\al-2)n}2} \jb{D}^\al \cj{v^\eps}dx
\\
&\quad
+2 c_n \Im \int_\T (\dx^{-1}  \Im \Ta_\o^\eps)^n \big( \jb{D}^{r-1-\frac {(\al-2)n}2} v^\eps \big)
\\
&\hspace*{60pt}
\times
\jb{D}^{r-1-\frac {(\al-2)n}2}
\big( \jb{D}^\al - D^\al \big) \cj{v^\eps}dx
\\
&=: J_{1,\ast}^r + R_3^r.
\end{aligned}
\label{Jr1a}
\end{equation}
Here, we note that
\[
\jb{k}^\al - |k|^\al
= \frac \al 2 \int_0^1 ( \theta +|k|^2 )^{\frac \al 2-1} d\theta
\les \jb{k}^{\al-2}
\]
for $\al>0$ and $k \in \Z$.
Hence,
we have
\[
\big\| \jb{D}^{r-1-\frac {(\al-2)n}2} \big( \jb{D}^\al - D^\al \big) \cj{v^\eps} \big\|_{H^{-\frac \al2+1}}
\les
\| v^\eps \|_{H^{r-1}}.
\]
With Proposition \ref{prop:bili},
we obtain that
\begin{equation}
\begin{aligned}
|R_3^r|
&\les
\big\| (\dx^{-1}  \Im \Ta_\o^\eps)^n \big( \jb{D}^{r-1-\frac {(\al-2)n}2} v^\eps \big) \big\|_{H^{\frac \al2-1}}
\\
&\hspace*{60pt}
\times
\big\| \jb{D}^{r-1-\frac {(\al-2)n}2} \big( \jb{D}^\al - D^\al \big) \cj{v^\eps} \big\|_{H^{-\frac \al2+1}}
\\
&\les
\big\| (\dx^{-1}  \Im \Ta_\o^\eps)^n \big\|_{H^{s-1}}
\| v^\eps \|_{H^{r-1}}^2
\\
&\le
C( \| u^\eps \|_{H^s}) \| v^\eps \|_{H^{r-1}}^2
\end{aligned}
\label{Jr1b}
\end{equation}
for $s> \max (\frac \al 2, \frac 32)$.

It follows from \eqref{Jr1a}, \eqref{eq:Kn}, and \eqref{comm}  that
\begin{align}
J_{1,\ast}^r + K_{n,\eps}^r
&=
-2 c_n \Im \int_\T
\jb{D}^{\frac \al 2+1} \Big(
(\dx^{-1}  \Im \Ta_\o^\eps)^n \big( \jb{D}^{r-1-\frac {(\al-2)n}2} v^\eps \big) 
\Big)
\notag
\\
&\hspace*{40pt}
\times
\jb{D}^{r-1-\frac {(\al-2)(n-1)}2} \cj{v^\eps}dx
\notag
\\
&\quad
- \al c_n \Im \int_\T \dx \big((\dx^{-1}  \Im \Ta_\o^\eps)^n \big) 
\big( \jb{D}^{r-1-\frac {(\al-2)(n-1)}2} \dx v^\eps \big)
\notag
\\
&\hspace*{40pt}
\times
\jb{D}^{r-1-\frac {(\al-2)(n-1)}2} \cj{v^\eps}dx
\notag
\\
&=
-2 c_n \Im \int_\T \Com{\frac \al2 + 1}{(\dx^{-1} \Im \Ta_\o^\eps)^n} \big( \jb{D}^{r-1-\frac {(\al-2)n}2} v^\eps \big)
\notag
\\
&\hspace*{40pt}
\times
\jb{D}^{r-1-\frac {(\al-2)(n-1)}2} \cj{v^\eps} dx
\notag
\\
&\quad
-
2 c_n \Im \int_\T  (\dx^{-1} \Im \Ta_\o^\eps)^n 
\big (\jb{D}^{r-1-\frac {(\al-2)n}2 + \frac \al2 +1} v^\eps \big)
\notag
\\
&\hspace*{40pt}
\times
\jb{D}^{r-1-\frac {(\al-2)(n-1)}2} \cj{v^\eps} dx
\notag
\\
&\quad
+
2 c_n \Im \int_\T \dx \big((\dx^{-1}  \Im \Ta_\o^\eps)^n \big) 
\big(\jb{D}^{r-1-\frac {(\al-2)(n-1)}2} \dx v^\eps \big)
\notag
\\
&\hspace*{40pt}
\times
\jb{D}^{r-1-\frac {(\al-2)(n-1)}2} \cj{v^\eps}dx
\notag
\\
&=:
R_4^r + R_5^r + R_6^r.
\label{J1}
\end{align}
Propositions \ref{prop:comm3} and \ref{prop:bili2} yield that
\begin{equation}
\begin{aligned}
|R_4^r| 
&\les
\Big(
\|(\dx^{-1}  \Im \Ta_\o^\eps)^n\|_{H^{\frac \al2 +1}}
+
\|(\dx^{-1}  \Im \Ta_\o^\eps)^n\|_{H^{\frac 52+\dl}}
\Big)
\\
&\qquad \times
\big\| \jb{D}^{r-1-\frac {(\al-2)n}2} v^\eps \big\|_{H^{\frac {\al-2}2}}
\|v^\eps\|_{H^{r-1}}
\\
&\le
C( \| u^\eps \|_{H^s}) \| v^\eps \|_{H^{r-1}}^2
\end{aligned}
\label{remi3}
\end{equation}
for $s > \regwp$ and $0<\dl \ll 1$.
From $\jb{D}^2= 1 - \dx^2$ and intgrating by parts, we have
\begin{equation}
\begin{aligned}
R_5^r 
&=
-2 c_n \Im \int_\T (\dx^{-1}  \Im \Ta_\o^\eps)^n
\, \big( \jb{D}^{r-1-\frac {(\al-2)(n-1)}2} \jb{D}^2 v^\eps \big)
\\
&\hspace*{60pt}
\times
\jb{D}^{r-1-\frac {(\al-2)(n-1)}2} \cj{v^\eps} dx
\\
&=
2 c_n \Im \int_\T (\dx^{-1}  \Im \Ta_\o^\eps)^n
\, \big( \jb{D}^{r-1-\frac {(\al-2)(n-1)}2} \dx^2 v^\eps \big)
\\
&\hspace*{60pt}
\times
\jb{D}^{r-1-\frac {(\al-2)(n-1)}2} \cj{v^\eps} dx
\\
&
=
- R_6^r
.
\end{aligned}
\label{remi4}
\end{equation}
From \eqref{Jr1a}--\eqref{remi4}, we obtain
\begin{equation}
|J_1^r + K_{n,\eps}^r|
\le
C( \| u^\eps \|_{H^s}) \| v^\eps \|_{H^{r-1}}^2
\label{J1Kn}
\end{equation}
for $s > \regwp$.

Fifth,
we treat $J_2^r$ in \eqref{Jr}.
Since 
\[
2 \Re (f \cj{\dx^2 f})
= \dx^2 (| f |^2 ) - 2|\dx f|^2
\]
for $f \in C^2(\T)$,
we obtain that
\begin{align}
J_2^r
&=
\eps c_n \int_\T (\dx^{-1}  \Im \Ta_\o^\eps)^n
\dx^2 
\big( \big| \jb{D}^{r-1-\frac{(\al-2)n}2} v^\eps \big|^2 \big) dx
\notag
\\
&\quad
-2 \eps c_n \int_\T (\dx^{-1} \Im \Ta_\o^\eps)^n 
\big| \jb{D}^{r-1-\frac{(\al-2)n}2} \dx v^\eps \big|^2  dx
\notag
\\
&=
\eps c_n \int_\T \big( \dx^2 (\dx^{-1}  \Im \Ta_\o^\eps)^n) \big)
\big| \jb{D}^{r-1-\frac{(\al-2)n}2} v^\eps \big|^2 dx
\notag
\\
&\quad
-2 \eps c_n \int_\T (\dx^{-1}  \Im \Ta_\o^\eps)^n
\big| \jb{D}^{r-1-\frac{(\al-2)n}2} \dx v^\eps \big|^2  dx
\notag
\\
&=: R_7^r + I^r.
\label{J2}
\end{align}
Since Corollary \ref{cor:bili3} yields
\[
\| \dx^2 ((\dx^{-1}  \Im \Ta_\o^\eps)^n)  \|_{L^\infty}
\les
\| (\dx^{-1}  \Im \Ta_\o^\eps)^n \|_{H^{s}}
\les
\| \Ta_\o^\eps \|_{H^{s-1}}^n
\le
C( \| u^\eps \|_{H^s})
\]
for $s>\frac 52$,
we have
\begin{equation}
|R_7^r|
\le
C( \| u^\eps \|_{H^s})
\| v^\eps \|_{H^{r-1}}^2
.
\label{remi6}
\end{equation}
For $I^r$,
from \eqref{conN} and Young's inequality,
we have
\begin{equation}
\begin{aligned}
|I^r|
&\le
2 \eps c_n
\| \dx^{-1}  \Im \Ta_\o^\eps \|_{L^\infty}^n
\big\| \jb{D}^{r-1-\frac {(\al-2)n}2} \dx v^\eps \big\|_{L^2}^2
\\
&
\le
\eps
C( \| u^\eps \|_{H^s})
\big\| \jb{D}^{r-1-\frac {(\al-2)N}2} \dx v^\eps \big\|_{L^2}^{2\frac nN} 
\| \jb{D}^{r-1} \dx v^\eps \|_{L^2}^{2(1-\frac nN)}
\\
&
\le
\eps
C( \| u^\eps \|_{H^s})
\big\| \jb{D}^{r- \frac 32} \dx v^\eps \big\|_{L^2}^{2\frac nN} 
\| \jb{D}^{r-1} \dx v^\eps \|_{L^2}^{2(1-\frac nN)}\\
&
\le
\eps
C( \| u^\eps \|_{H^s})
\big\| \jb{D}^{r-2} \dx v^\eps \big\|_{L^2}^{\frac nN} 
\| \jb{D}^{r-1} \dx v^\eps \|_{L^2}^{2-\frac n{N}}
\\
&\le
\frac \eps{2N} \| \dx v^\eps\|_{H^{r-1}}^2
+
C( \| u^\eps \|_{H^s})
\| v^\eps\|_{H^{r-1}}^2
\end{aligned}
\label{Ir2}
\end{equation}
for $s> \frac 52$ and $r \in \R$.
Combining \eqref{J2}--\eqref{Ir2}, we obtain that
\begin{equation}
\begin{aligned}
|J_2^r|
\le
C( \| u^\eps \|_{H^s})
\| v^\eps \|_{H^{r-1}}^2
+ \frac \eps{2N} \| \dx v^\eps\|_{H^{r-1}}^2
\end{aligned}
\label{J2r}
\end{equation}
for $s > \frac 52$ and $r \in \R$.

Sixth,
we estimate $J_3^r$ in \eqref{Jr}.
By definition,
it holds that 
\[
\dx (\dx^{-1} \Im \Ta_\o^\eps) = P_{\neq 0} \Im \Ta_\o^\eps.
\]
Therefore, it follows that
\[
(\dx^{-1}  \Im \Ta_\o^\eps)^n  P_{\neq 0} \Im \Ta_\o^\eps
= \frac1{n+1} \dx \big( (\dx^{-1}  \Im \Ta_\o^\eps)^{n+1} \big).
\]
From \eqref{Jr} and \eqref{eq:Kn}, we have
\begin{align}
J_3^r
&=
-\frac{2 c_n}{n+1} \Im \int_\T \dx \big( (\dx^{-1}  \Im \Ta_\o^\eps)^{n+1} \big)
\notag 
\\
&\hspace*{60pt}
\times
\big( \jb{D}^{r-1-\frac {(\al-2)n}2} \dx v^\eps \big) \, \jb{D}^{r-1-\frac {(\al-2)n}2} \cj{v^\eps} dx
\notag
\\
&= K_{n+1,\eps}^r.
\label{J3}
\end{align}

Finally, we consider $J_4^r$ and $J_5^r$ in \eqref{Jr}.
From a simple calculation, we have
\begin{equation*}
\begin{aligned}
J_4^r
&=
c_n \int_\T (\dx^{-1}  \Im \Ta_\o^\eps)^n (\Re \Ta_\o^\eps) \,
\dx \big( \big| \jb{D}^{r-1-\frac {(\al-2)n}2} v^\eps \big|^2 \big) dx
\\
&=
-c_n \int_\T \dx \big( (\dx^{-1}  \Im \Ta_\o^\eps)^n \Re \Ta_\o^\eps \big) 
\big| \jb{D}^{r-1-\frac {(\al-2)n}2} v^\eps \big|^2 dx.
\end{aligned}
\end{equation*}
Similarly, we obtain that
\begin{equation*}
\begin{aligned}
J_5^r
&=
c_n \Re \int_\T (\dx^{-1}  \Im \Ta_\o^\eps)^n \cj{\Ta_{\cj \o}^\eps} \,
\dx \big( (\jb{D}^{r-1-\frac {(\al-2)n}2} v^\eps)^2 \big) dx
\\
&=
-c_n \Re \int_\T \dx \big( (\dx^{-1}  \Im \Ta_\o^\eps)^n \cj {\Ta_{\cj \o}^\eps} \big) 
\big( \jb{D}^{r-1-\frac {(\al-2)n}2} v^\eps \big)^2 dx.
\end{aligned}
\end{equation*}
By Proposition \ref{prop:bili}, we obtain that
\begin{equation}
\begin{aligned}
|J_4^r|
&\le
\| \dx \big( (\dx^{-1}  \Im \Ta_\o^\eps)^n \Re \Ta_\o^\eps \big) \|_{L^\infty}
\big\| \jb{D}^{r-1-\frac {(\al-2)n}2} v^\eps \big\|_{L^2}^2
\\
&\les
\| \dx^{-1}  \Im \Ta_\o^\eps \|_{H^{s-1}}^n \| \Re \Ta_\o^\eps \|_{H^{s-1}}
\| v^\eps \|_{H^{r-1}}^2
\\
&\le
C( \| u^\eps \|_{H^s})
\| v^\eps \|_{H^{r-1}}^2
\end{aligned}
\label{J4}
\end{equation}
for $s > \frac 52$.
It follows from a similar calculation as in \eqref{J4} that
\begin{equation}
|J_5^r|
\le
C( \| u^\eps \|_{H^s})
\| v^\eps \|_{H^{r-1}}^2.
\label{J5}
\end{equation}

Combining \eqref{dtL}--\eqref{Jr}, \eqref{remi2}, \eqref{J1Kn}, and \eqref{J2r}--\eqref{J5}, we get
\begin{equation*}
\begin{aligned}
&K_{n,\eps}^r+ \dt L_{n,\eps}^r - K_{n+1,\eps}^r
\\
&=
K_{n,\eps}^r + \sum_{\l=1}^5 J_\l^r - K_{n+1,\eps}^r + R_1^r + R_2^r
\\
&\le
C( \| u^\eps \|_{H^s})
\| v^\eps \|_{H^{r-1}}^2
+ \frac \eps{2N} \| \dx v^\eps\|_{H^{r-1}}^2
\end{aligned}
\end{equation*}
for $s> \regwp$ and $r \ge 1$.
With \eqref{ene1},
this concludes the proof.
\end{proof}

Next, we prove Lemma \ref{lem:exdtwLn}.

\begin{proof}
[Proof of Lemma \ref{lem:exdtwLn}]
As in the proof of Lemma \ref{lem:modi}, we suppress the time dependence in this proof. 
By \eqref{veps}, we have
\begin{equation}
\begin{aligned}
\frac12 \dt \big( \| v^\eps\|_{H^{r-1}}^2 \big)
&= 
\Re \big( \jb{D}^{r-1} \dt v^\eps, \jb{D}^{r-1} v^\eps \big)_{L^2}
\\
&=
-\eps \| \dx v^\eps\|_{H^{r-1}}^2
\\
&\quad
+
\Im \big( \jb{D}^{r-1} D^\al v^\eps, \jb{D}^{r-1} v^\eps \big)_{L^2}
\\
&\quad
-
\Im \big( \jb{D}^{r-1} ( (P_{\neq 0} \Im \Ta_\o^\eps) \dx v^\eps), \jb{D}^{r-1} v^\eps \big)_{L^2}
\\
&\quad
+
\Re \big(\jb{D}^{r-1}((\Re \Ta_\o^\eps) \dx v^\eps), \jb{D}^{r-1} v^\eps \big)_{L^2}
\\
&\quad
+
\Re \big( \jb{D}^{r-1}(\Ta_{\cj \o}^\eps \cj{\dx v^\eps}), \jb{D}^{r-1} v^\eps \big)_{L^2}
\\
&\quad
+
\Re \big( \jb{D}^{r-1} \Rl^\eps, \jb{D}^{r-1} v^\eps \big)_{L^2}
\\
&
=:
-\eps \| \dx v^\eps\|_{H^{r-1}}^2 + \sum_{j=1}^5 \I_j.
\end{aligned}
\label{exdtw}
\end{equation}

A direct calculation shows that
\begin{equation}
\begin{aligned}
\I_1 
=
\Im \int_\T | \jb{D}^{r-1} D^{\frac\al2} v^\eps|^2 dx
=0.
\end{aligned}
\label{exI1}
\end{equation}
We decompose $\I_2$ into two parts:
\begin{equation*}
\begin{aligned}
\I_2
&=
-\Im \big( \big[\jb{D}^{r-1}, P_{\neq 0} \Im \Ta_\o^\eps \big] \dx v^\eps, \jb{D}^{r-1} v^\eps \big)_{L^2}
\\
&\quad
- \Im \big( (P_{\neq 0} \Im \Ta_\o^\eps) \jb{D}^{r-1} \dx v^\eps, \jb{D}^{r-1} v^\eps \big)_{L^2}
\\
&=:\I_{2,1} + \I_{2,2}.
\end{aligned}
\end{equation*}
By Proposition \ref{prop:commC} and Corollary \ref{cor:bili3}, we have
\begin{equation}
\begin{aligned}
|\I_{2,1}| 
&\le 
\|[ \jb{D}^{r-1}, P_{\neq 0} \Im \Ta_\o^\eps] \dx v^\eps\|_{L^2} \|v^\eps\|_{H^{r-1}}
\\
&\les
(
\|P_{\neq 0} \Im \Ta_\o^\eps\|_{H^{s-1}} \|v^\eps\|_{H^{r-1}}
\\
&\hspace*{30pt}
+ \|P_{\neq 0} \Im \Ta_\o^\eps\|_{H^{r-1}} \|v^\eps\|_{H^{s-1}}
) \|v^\eps\|_{H^{r-1}}
\\
&\le
C( \| u^\eps \|_{H^s})(1+\|v^\eps\|_{H^{r-1}})\|v^\eps\|_{H^{r-1}}
\end{aligned}
\label{exI21}
\end{equation}
for $s>\frac 52$ and $r \ge 1$.
It follows from \eqref{eq:Kn} that
\begin{equation}
\I_{2,2} = \frac 12 K_{1,\eps}^r.
\label{exI22}
\end{equation}

We decompose $\I_3$ as follows:
\begin{equation*}
\begin{aligned}
\I_3
&=
\Re \big( \big[\jb{D}^{r-1}, \Re \Ta_\o^\eps \big] \dx v^\eps, \jb{D}^{r-1} v^\eps \big)_{L^2}
\\
&\quad
+
\Re \big( (\Re \Ta_\o^\eps) \jb{D}^{r-1} \dx v^\eps, \jb{D}^{r-1} v^\eps \big)_{L^2}
\\
&=:\I_{3,1} + \I_{3,2}.
\end{aligned}
\end{equation*}
As in \eqref{exI21}, we have
\begin{equation*}
\begin{aligned}
|\I_{3,1}| 
&\les 
(\|\Re \Ta_\o^\eps\|_{H^{s-1}} \|v^\eps\|_{H^{r-1}} + \|\Re \Ta_\o^\eps\|_{H^{r-1}} \|v^\eps\|_{H^{s-1}} )\| v^\eps\|_{H^{r-1}}
\\
&\le
C( \| u^\eps \|_{H^s})(1+\|v^\eps\|_{H^{r-1}}) \|v^\eps\|_{H^{r-1}}
\end{aligned}
\end{equation*}
for $s>\frac 52$ and $r \ge 1$.
From Corollary \ref{cor:bili3}, we get
\begin{equation*}
\begin{aligned}
\I_{3,2} 
&= 
\Re \int_\T (\Re \Ta_\o^\eps) (\jb{D}^{r-1} \dx v^\eps) \cj{\jb{D}^{r-1} v^\eps}dx
\\
&=
- \frac 12\int_\T \dx (\Re \Ta_\o^\eps) |\jb{D}^{r-1} v^\eps|^2 dx
\\
&\le
\|\dx \Re \Ta_\o^\eps\|_{L^\infty} \|\jb{D}^{r-1} v^\eps\|_{L^2}^2
\\
&\les
\| \Re \Ta_\o^\eps \|_{H^{s-1}} \|v^\eps\|_{H^{r-1}}^2
\\
&\le
C( \| u^\eps \|_{H^s})(1+\|v^\eps\|_{H^{r-1}})\|v^\eps\|_{H^{r-1}}.
\end{aligned}
\end{equation*}
Hence, we obtain that
\begin{equation}
|\I_3| \le
C( \| u^\eps \|_{H^s})(1+\|v^\eps\|_{H^{r-1}})\|v^\eps\|_{H^{r-1}}
\label{exI3}
\end{equation}
for $s>\frac 52$ and $r \ge 1$.

The same calculation as in \eqref{exI3} yields that
\begin{equation}
|\I_4| \le
C( \| u^\eps \|_{H^s})(1+\|v^\eps\|_{H^{r-1}})\|v^\eps\|_{H^{r-1}}.
\label{exI4}
\end{equation}

Since $\Rl^\eps$ is a polynomial in $u^\eps, v^\eps, \cj{u^\eps}, \cj{v^\eps}$,
Corollary \ref{cor:bili3} implies that
\begin{equation}
\begin{aligned}
|\I_5|
&\les 
\| \jb{D}^{r-1}\Rl^\eps\|_{L^2} \| \jb{D}^{r-1} v^\eps\|_{L^2}
\\
&\le
C( \| u^\eps \|_{H^s})(1+\|v^\eps\|_{H^{r-1}})\|v^\eps\|_{H^{r-1}}.
\end{aligned}
\label{exI5}
\end{equation}

From Lemma \ref{lem:modi}, \eqref{exdtw}--\eqref{exI5}, and \eqref{ene1},
we obtain that
\begin{equation}
\begin{aligned}
&\dt  \Big( \| v^\eps\|_{H^{r-1}}^2 + \sum_{n=1}^N L_{n,\eps}^r \Big)
\\
&
=
-2 \eps \| \dx v^\eps\|_{H^{r-1}}^2 + \sum_{n=1}^N \Big( K_{n,\eps}^r + \dt  L_{n,\eps}^r - K_{n+1,\eps}^r \Big) + K_{N+1,\eps}^r 
\\
&\quad
+ 2(\I_{2,1} + \I_3 + \I_4 + \I_5 )
\\
&\le
-\eps \| \dx v^\eps\|_{H^{r-1}}^2 + C(E_\eps^s)(1+E_\eps^r)\|v^\eps\|_{H^{r-1}} + |K_{N+1,\eps}^r|.
\end{aligned}
\label{exdtwL}
\end{equation}
It follows from
\eqref{eq:Kn}, \eqref{conN}, Proposition \ref{prop:bili}, and Corollary \ref{cor:bili3}
that
\begin{equation}
\begin{aligned}
|K_{N+1,\eps}^r|
&\les
\big\|
\dx \big( \big( \dx^{-1} \Im \Ta_\o^\eps \big)^{N+1} \big)
\cdot
\jb{D}^{r-1-\frac{(\al-2)N}2} \cj{v^\eps}
\big\|_{H^{1-\frac{(\al-2)N}2}}
\\
&\qquad
\times
\big\| \jb{D}^{r-1-\frac{(\al-2)N}2} \dx v^\eps \big\|_{H^{-1+\frac{(\al-2)N}2}}
\\
&\les
\big\|
\dx \big( \big( \dx^{-1} \Im \Ta_\o^\eps \big)^{N+1} \big)
\big\|_{H^{s-1}}
\big\| \jb{D}^{r-1-\frac{(\al-2)N}2} \cj{v^\eps} \big\|_{H^{\frac{(\al-2)N}2}}
\\
&\qquad
\times
\| v^\eps \|_{H^{r-1}}
\\
&\le
C(\| u^\eps \|_{H^s}) \|v^\eps\|_{H^{r-1}}^2
\end{aligned}
\label{exKN}
\end{equation}
for $s>\frac 32$.
With \eqref{exdtwL} and \eqref{exKN},
we obtain the desired bound.
\end{proof}

\subsection{Energy estimates for the difference}
\label{Subsec:estdiff}

In this subsection,
we prove an estimate for the difference of two solutions to \eqref{fNLS}.
We use the same notation as in Subsection \ref{SUBSEC:energy}.
In particular,
let $s>\frac 32$ and $\phi \in H^s(\T)$
and let $u^\eps$ denote the solution to \eqref{RfNLS} obtained in Proposition \ref{prop:exRfNLS}.

Let $\eps_1 ,\eps_2 \in (0,1)$ satisfy $\eps_1>\eps_2$.
We set
\[
\br u = u^{\eps_1} - u^{\eps_2},
\qquad
\br v = v^{\eps_1} - v^{\eps_2},
\]
for short.

It follows from \eqref{ueps} that
\begin{align}
\dt \br u + i D^\al \br u
&=
\eps _1 \dx^2 \br u
+ (\eps_1 - \eps_2) \dx^2 u^{\eps_2}
+
\Ta^{\eps_1} - \Ta^{\eps_2} .
\label{bru}
\end{align}
Moreover,
by \eqref{veps}, we have
\begin{equation}
\begin{aligned}
\dt \br v + i D^\al \br v
&=
\eps _1 \dx^2 \br v
+(\eps_1 - \eps_2) \dx^2 v^{\eps_2}
\\
&\quad
+i (P_{\neq 0} \Im \Ta_{\o}^{\eps_1}) \dx \br v
+(\Re \Ta_\o^{\eps_1})\dx \br v
+\Ta_{\cj \o}^{\eps_1} \cj{\dx \br v}
\\
&\quad
+i (P_{\neq 0} \Im \Ta_{\o}^{\eps_1} - P_{\neq 0} \Im \Ta_{\o}^{\eps_2}) \dx v^{\eps_2}
\\
&\quad
+ (\Re \Ta_{\o}^{\eps_1} - \Re \Ta_{\o}^{\eps_2}) \dx v^{\eps_2}
\\
&\quad
+( \Ta_{\cj \o}^{\eps_1} - \Ta_{\cj \o}^{\eps_2}) \cj{\dx v^{\eps_2}}
\\
&\quad
+ \br R
,
\end{aligned}
\label{brv}
\end{equation}
where
\begin{equation}
\br R
:=
\Ta_{\zeta}^{\eps_1} \br v
+ \Ta_{\cj \zeta}^{\eps_1} \cj {\br v}
+ (\Ta_{\zeta}^{\eps_1} - \Ta_{\zeta}^{\eps_2}) v^{\eps_2}
+ ( \Ta_{\cj \zeta}^{\eps_1} - \Ta_{\cj \zeta}^{\eps_2}) \cj {v^{\eps_2}}.
\label{brR}
\end{equation}

Define
\begin{align}
\br L_n (t)
&:=
c_n \int_\T (\dx^{-1} \Im \Ta_\o^{\eps_1} (t))^n 
\big| \jb{D}^{-\frac {(\al-2)n}2} \br v (t) \big|^2 dx,
\label{diffL}
\\
\br K_n (t)
&:=
-\al c_n \Im \int_\T \dx \big( (\dx^{-1} \Im \Ta_\o^{\eps_1})^n (t) \big)
\notag
\\
&\hspace*{40pt}
\times
\big( \jb{D}^{-\frac {(\al-2)(n-1)}2} \dx \br v (t) \big) \jb{D}^{-\frac {(\al-2)(n-1)}2} \cj{\br v (t)} dx
\label{diffK}
\end{align}
for $n \in \N$,
where $c_n$ is defined in \eqref{eq:Lncon1}.

By \eqref{est:Lnes} with $r=1$,
there exists $b_n = b_n(\|u^{\eps_1}\|_{L_T^\infty H^s})>0$ such that
\begin{align*}
\big|\br L_n (t) \big| 
&\le
\frac 1{10 n^2} \| \br v (t)\|_{L^2}^2
+
a_n \| \br u (t) \|_{L^2}^2
\| \dx^{-1} \Im \Ta_\o^{\eps_1} (t) \|_{L^\infty}^{2N}
\\
&\le
\frac 1{10 n^2} \| \br v (t)\|_{L^2}^2
+
b_n \| \br u (t) \|_{L^2}^2
\end{align*}
for $n=1,\dots, N$,
where $N$ is defined in \eqref{conN}.
Hence, we have
\begin{equation}
\bigg| \sum_{n=1}^N \br L_n (t) \bigg|
\le
\frac 12
\| \br v (t)\|_{L^2}^2
+
\sum_{n=1}^N b_n \| \br u (t) \|_{L^2}^2.
\label{diffLest}
\end{equation}

Set
\[
b := 2 \sum_{n=1}^N b_n.
\]
We then define
\begin{equation}
\br E (t) 
:=
\Big( b \|\br u (t)\|_{L^2}^2 + \| \br v (t)\|_{L^2}^2 + \sum_{n=1}^N \br L_n (t)
\Big)^{\frac12}.
\label{brEs}
\end{equation}
It follows from \eqref{diffLest} that
\begin{equation}
\frac b2 \|\br u(t)\|_{L^2}^2 + \frac 12 \| \br v(t)\|_{L^2}^2
\le
\br E (t)^2
\le
2 b \|\br u(t)\|_{L^2}^2 + 2 \| \br v(t)\|_{L^2}^2
\label{diff2a}
\end{equation}
for $t \in [0,T]$.

For the energy estimate for the difference,
we use the following lemma.

\begin{lemma}
\label{lem:dimodi}
Let $\al > 2$ and $1 > \eps_1 > \eps_2 > 0$.
Suppose that
\[
s> \regwp.
\]
Then,
we have
\[
\begin{aligned}
&\br K_n(t) + \dt \br L_n(t) - \br K_{n+1}(t)
\\
&\le
C( \| \phi \|_{H^s} )(\eps_1 - \eps_2 + \br E (t))
\| \br v(t)\|_{L^2}
+ \frac {\eps_1}{2N} \| \dx \br v(t) \|_{L^2}^2
\end{aligned}
\]
for any $n = 1, \cdots, N$ and $t \in (0,T)$,
where $N$ is defined in \eqref{conN}.
\end{lemma}

\begin{proof}
For simplicity, we suppress the time dependence in this proof.
Fix $n =1,\dots,N$.
It follows from \eqref{diffL} that
\begin{equation}
\begin{aligned}
\dt \br L_n
&=
2 c_n \Re \int_\T ( \dx^{-1} \Im \Ta_\o^{\eps_1} )^n
\big( \jb{D}^{-\frac {(\al-2)n}2} \br v \big) \jb{D}^{-\frac {(\al-2)n}2} \cj{\dt \br v}dx
\\
&\quad
+
c_n \int_\T \dt ( ( \dx^{-1} \Im \Ta_\o^{\eps_1} )^n )
\big| \jb{D}^{-\frac {(\al-2)n}2} \br v \big|^2dx
\\
&=:
\br J + \br R_1.
\end{aligned}
\label{dtbrL}
\end{equation}
As in \eqref{remi1a}, with Corollary \ref{cor:enebdur}, we have
\begin{equation}
|\br R_1| \le C(\| \phi \|_{H^s}) \|\br v\|_{L^2}^2
\label{brremi1a}
\end{equation}
for $s > \regwp$.

From \eqref{brv}, we have
\begin{align}
\br J
&=
-2 c_n \Im \int_\T \big( \dx^{-1}  \Im \Ta_\o^{\eps_1} \big)^n
\big( \jb{D}^{-\frac {(\al-2)n}2} \br v \big) \jb{D}^{-\frac {(\al-2)n}2} D^\al \cj{\br v}dx
\notag
\\
&\quad
+
2 c_n \eps_1 \Re \int_\T
\big( \dx^{-1} \Im \Ta_\o^{\eps_1} \big)^n
\big( \jb{D}^{-\frac {(\al-2)n}2} \br v \big) \jb{D}^{-\frac {(\al-2)n}2} \cj{\dx^2 \br v}dx
\notag
\\
&\quad
+
2 c_n (\eps_1 - \eps_2)
\Re \int_\T
\big( \dx^{-1}  \Im \Ta_\o^{\eps_1} \big)^n
\big( \jb{D}^{-\frac {(\al-2)n}2} \br v \big)
\notag
\\
&\hspace*{100pt}
\times
\jb{D}^{-\frac {(\al-2)n}2} \cj{\dx^2 v^{\eps_2}} dx
\notag
\\
&\quad
+
2 c_n \Im \int_\T
\big( \dx^{-1}  \Im \Ta_\o^{\eps_1} \big)^n
(P_{\neq 0} \Im \Ta_\o^{\eps_1}) \big( \jb{D}^{-\frac {(\al-2)n}2} \br v \big)
\notag
\\
&\hspace*{60pt}
\times
\jb{D}^{-\frac {(\al-2)n}2} \cj{\dx \br v}dx
\notag
\\
&\quad
+
2 c_n \Re \int_\T (\dx^{-1} \Im \Ta_\o^{\eps_1})^n (\Re \Ta_\o^{\eps_1}) 
(\jb{D}^{-\frac {(\al-2)n}2} \br v)
\notag
\\
&\hspace*{60pt}
\times
\jb{D}^{-\frac{(\al-2)n}2} \cj{\dx \br v}dx
\notag
\\
&\quad
+
2 c_n \Re \int_\T (\dx^{-1}  \Im \Ta_\o^{\eps_1})^n \Ta_{\cj \o}^{\eps_1} 
(\jb{D}^{-\frac {(\al-2)n}2} \br v)
\notag
\\
&\hspace*{60pt}
\times
\jb{D}^{-\frac {(\al-2)n}2} \dx \br v dx
\notag
\\
&\quad
+
2 c_n \Im \int_\T
\big( \dx^{-1}  \Im \Ta_\o^{\eps_1} \big)^n
\big( \jb{D}^{-\frac {(\al-2)n}2} \br v \big)
\notag
\\
&\hspace*{60pt} 
\times
\jb{D}^{-\frac {(\al-2)n}2} \big( (P_{\neq 0} \Im \Ta_\o^{\eps_1} - P_{\neq 0} \Im \Ta_\o^{\eps_2})
\cj{\dx v^{\eps_2}} \big)dx
\notag
\\
&\quad
+
2 c_n \Re \int_\T (\dx^{-1}  \Im \Ta_\o^{\eps_1})^n 
\big( \jb{D}^{-\frac {(\al-2)n}2} \br v \big)
\notag
\\
&\hspace*{60pt} 
\times
\jb{D}^{-\frac {(\al-2)n}2} \big ((\Re \Ta_\o^{\eps_1} - \Re \Ta_\o^{\eps_2})
\cj{\dx v^{\eps_2}} \big)dx
\notag
\\
&\quad
+
2 c_n \Re \int_\T (\dx^{-1}  \Im \Ta_\o^{\eps_1})^n
\big( \jb{D}^{-\frac {(\al-2)n}2} \br v \big)
\notag
\\
&\hspace*{60pt}
\times
\jb{D}^{-\frac {(\al-2)n}2} \big ((\cj{\Ta_{\cj \o}^{\eps_1}} - \cj {\Ta_{\cj \o}^{\eps_2}})
\dx v^{\eps_2} \big)dx
\notag
\\
&\quad
+
\br R_2
\notag
\\
&=:
\sum_{\l=1}^9 \br J_{\l} + \br R_2,
\label{brJ1}
\end{align}
where
\begin{equation*}
\begin{aligned}
\br R_2
&=
2c_n \Im \int_\T (\dx^{-1}  \Im \Ta_\o^{\eps_1})^n
\big( \jb{D}^{-\frac {(\al-2)n}2} \br v \big)
\big( \big[ \jb{D}^{-\frac {(\al-2)n}2},  \Im \Ta_\o^{\eps_1} \big] \cj{\dx \br v} \big) dx
\\
&\quad
+
2c_n \Re \int_\T (\dx^{-1}  \Im \Ta_\o^{\eps_1})^n
\big( \jb{D}^{-\frac {(\al-2)n}2} \br v \big)
\big( \big[ \jb{D}^{-\frac {(\al-2)n}2}, \Re \Ta_\o^{\eps_1} \big] \cj{\dx \br v} \big)dx
\\
&\quad
+
2c_n \Re \int_\T (\dx^{-1}  \Im \Ta_\o^{\eps_1})^n
\big( \jb{D}^{-\frac {(\al-2)n}2} \br v \big)
\big( \big[ \jb{D}^{-\frac {(\al-2)n}2}, \cj{\Ta_{\cj \o}^{\eps_1}} \big] \dx \br v \big) dx
\\
&\quad
+
2c_n \Re \int_\T (\dx^{-1}  \Im \Ta_\o^{\eps_1})^n (\jb{D}^{-\frac {(\al-2)n}2} \br v) \jb{D}^{-\frac {(\al-2)n}2} \cj{\br R} dx
\\
&=:
\br R_{2,1}+ \br R_{2,2} + \br R_{2,3} + \br R_{2,4}.
\end{aligned}
\end{equation*}

We consider $\br R_2$.
Proposition \ref{prop:commne} and Corollary \ref{cor:enebdur} imply that
\begin{equation}
\begin{aligned}
|\br R_{2,1}|
&\les
\| \dx^{-1} \Im \Ta_\o^{\eps_1} \|_{L^\infty}^n 
\big\| \big[ \jb{D}^{-\frac {(\al-2)n}2}, \Im \Ta_\o^{\eps_1} \big] \dx \br v \big\|_{L^2}
\| \br v\|_{L^2}
\\
&\le
C(\| \phi \|_{H^s}) \|\br v\|_{L^2}^2
\end{aligned}
\label{diremi21}
\end{equation}
for $s>\frac 52$.
The same calculation as in \eqref{diremi21} yields that
\begin{equation}
|\br R_{2,2}| + |\br R_{2,3}|
\le C(\| \phi \|_{H^s}) \|\br v\|_{L^2}^2.
\label{diremi2223}
\end{equation}
Moreover,
by \eqref{brR}, we have
\begin{equation}
|\br R_{2,4}| 
\le 
\| \dx^{-1} \Im \Ta_\o^{\eps_1} \|_{L^\infty}^n \|\br v\|_{L^2} \|\br R\|_{L^2}
\le
C(\| \phi \|_{H^s}) \|\br u\|_{H^1} \|\br v\|_{L^2}.
\label{diremi24}
\end{equation}
It follows from \eqref{diremi21}--\eqref{diremi24} that
\begin{equation}
|\br R_2| \le C(\| \phi \|_{H^s}) \|\br u\|_{H^1} \|\br v\|_{L^2}
\label{diremi2}
\end{equation}
for $s >\frac 52$.

For $\br J_\l$
with $\l \in \{1,2,4,5,6\}$,
the following estimate can be obtained by a calculation similar to that for \eqref{J1Kn} and \eqref{J2r}--\eqref{J5}, respectively:
\begin{equation}
\begin{aligned}
\br K_n + \sum_{\l \in \{1,2,4,5,6\}} \br J_\l - \br K_{n+1}
\le
C(\| \phi \|_{H^s}) \| \br v \|_{L^2}^2 + \frac{\eps_1}{2N} \| \dx \br v\|_{L^2}^2
\end{aligned}
\label{diJl}
\end{equation}
for $s > \regwp$.
 
We now  focus on estimates for the remaining parts.
For $\br J_3$, from \eqref{brJ1} and integrating by parts, we have
\begin{equation*}
\begin{aligned}
\br J_3
&=
- 2c_n (\eps_1 - \eps_2)
\\
&\quad
\times \bigg( 
\Re \int_\T \dx \big((\dx^{-1} \Im \Ta_\o^{\eps_1})^n \big)
\big( \jb{D}^{-\frac {(\al-2)n}2} \br v \big) \jb{D}^{-\frac {(\al-2)n}2} \cj {\dx v^{\eps_2}} dx
\\
&\quad
+ \Re \int_\T (\dx^{-1} \Im \Ta_\o^{\eps_1})^n \big( \jb{D}^{-\frac {(\al-2)n}2} \dx \br v \big) 
\jb{D}^{-\frac {(\al-2)n}2} \cj {\dx v^{\eps_2}} dx
\bigg)
\\
&=:\br R_3 + \br R_4.
\end{aligned}
\end{equation*}
From $s \ge 2$, we obtain that
\begin{equation}
\begin{aligned}
|\br R_3|
&\les
(\eps_1 - \eps_2)
\| \dx  \big((\dx^{-1} \Im \Ta_\o^{\eps_1})^n \big) \|_{L^\infty}
\\
&\quad
\times
\big\| \jb{D}^{-\frac {(\al-2)n}2} \br v \big\|_{L^2}
\big\| \jb{D}^{-\frac {(\al-2)n}2} \dx v^{\eps_2} \big\|_{L^2} 
\\
&\le
C(\| \phi \|_{H^s}) (\eps_1 - \eps_2) \| \br v\|_{L^2} \| v^{\eps_2}\|_{H^{1}} 
\\
&\le
C(\| \phi \|_{H^s})  (\eps_1 - \eps_2).
\end{aligned}
\label{diremi3}
\end{equation}
Similarly, we have
\begin{equation}
\begin{aligned}
|\br R_4|
\le
C(\| \phi \|_{H^s}) (\eps_1 - \eps_2) \|\br v\|_{H^{1}} \| v^{\eps_2}\|_{H^{1}}
\le
C(\| \phi \|_{H^s})  (\eps_1 - \eps_2).
\end{aligned}
\label{diremi4}
\end{equation}
It follows from \eqref{diremi3} and \eqref{diremi4} that
\begin{equation}
|\br J_3| \le C(\| \phi \|_{H^s})  (\eps_1 - \eps_2)
\label{diJ3}
\end{equation}
for $s \ge 2$.

We estimate $\br J_7$ in \eqref{brJ1}.
Note that
from Proposition \ref{prop:bili} and Corollary \ref{cor:bili3}, we have
\begin{equation}
\begin{aligned}
\|\Ta^{\eps_1} - \Ta^{\eps_2}\|_{L^2} 
&\le 
C(\|\phi\|_{H^s})(\|\br u\|_{L^2} + \|\br v\|_{L^2})
\le 
C(\|\phi\|_{H^s}) \|\br u\|_{H^{1}}
\end{aligned}
\label{difTa}
\end{equation}
for $s>\frac 32$.
With \eqref{brJ1} and \eqref{difTa},
Proposition \ref{prop:bili}  implies that
\begin{equation}
\begin{aligned}
|\br J_7|
&\les
\| \dx^{-1} \Im \Ta_\o^{\eps_1} \|_{L^\infty}^n 
\big\| \jb{D}^{-\frac {(\al-2)n}2} \br v \big\|_{L^2}
\\
&\quad
\times
\big\| \jb{D}^{-\frac {(\al-2)n}2} \big((P_{\neq 0} \Im \Ta_\o^{\eps_1} - P_{\neq 0} \Im \Ta_\o^{\eps_2}) \dx v^{\eps_2} \big) \big\|_{L^2}
\\
&\le
C(\| \phi \|_{H^s}) \| \br v \|_{L^2}
\|\Ta_\o^{\eps_1} - \Ta_\o^{\eps_2}\|_{L^2} \| \dx v^{\eps_2} \|_{H^{s-2}}
\\
&\le
C(\| \phi \|_{H^s}) \| \br v \|_{L^2} \| \br u \|_{H^1}  \| \dx v^{\eps_2} \|_{H^{s-2}}
\\
&\le
C(\| \phi \|_{H^s}) \| \br u \|_{H^1} \| \br v \|_{L^2}
\end{aligned}
\label{diJ7}
\end{equation}
for $s>\frac 52$.

It follows from \eqref{brJ1} and the same calculation as in \eqref{diJ7} that 
\begin{equation}
|\br J_8| + |\br J_9 |
\le
C(\| \phi \|_{H^s}) \| \br u \|_{H^1} \| \br v \|_{L^2}.
\label{diJ89}
\end{equation}
Combining \eqref{dtbrL}--\eqref{brJ1}, \eqref{diremi2}, \eqref{diJl}, \eqref{diJ3}, \eqref{diJ7}, and \eqref{diJ89},
we obtain Lemma \ref{lem:dimodi}.
\end{proof}

We then prove the energy estimate for the difference $\br u$.

\begin{proposition}
\label{prop:diff1}
Let $\al>2$.
Suppose that
\[
s> \regwp,
\]
where $\regwp$ is defined in \eqref{condA}.
Then,
there exists $C_3( \| \phi \|_{H^s})>0$ such that
\[
\sup_{t \in [0,T]} \br E (t)
\le
C_3(\| \phi \|_{H^s})
|\eps_1 - \eps_2|^{\frac 12}
\]
for $\eps_1, \eps_2 \in (0,1)$ with $\eps_1>\eps_2$,
where $T$ is given in \eqref{time1}.
\end{proposition}

\begin{proof}
We prove that there exists $\br C(\|\phi\|_{H^s})>0$ such that
\begin{equation}
\dt  (\br E (t)^2)
\le 
\br C(\|\phi\|_{H^s}) \big( | \eps_1 - \eps_2 | + \br E(t)^2 \big)
\label{diineq}
\end{equation}
for $t \in (0,T)$ and $1 > \eps_1 > \eps_2 > 0$. Then, by solving this diffrential inequality and $\br E(0) = 0$,
we then obtain the desired bound.

In what follows,
we suppress the time dependence for short.
By \eqref{bru}, we have
\begin{equation*}
\begin{aligned}
\frac12 \dt ( \|\br u\|_{L^2}^2 )
&=
-\eps_1 \|\dx u^{\eps_1}\|_{L^2}^2 
+ (\eps_1 - \eps_2) \Re (\dx^2 u^{\eps_2}, \br u)_{L^2} 
\\
&\quad
+\Re ( \Ta^{\eps_1} - \Ta^{\eps_2}, \br u)_{L^2}
\\
&\le
|\eps_1 - \eps_2| \|u^{\eps_2}\|_{H^{2}} \|\br u\|_{L^2}
+
\|\Ta^{\eps_1} - \Ta^{\eps_2}\|_{L^2} \|\br u\|_{L^2}.
\end{aligned}
\end{equation*}
With \eqref{difTa}, we get
\begin{equation}
\dt  (\|\br u\|_{L^2}^2)
\le
C(\|\phi\|_{H^s}) \big(|\eps_1 - \eps_2| + \|\br u\|_{H^{1}} \|\br u\|_{L^2} \big)
\label{difu}
\end{equation}
for $s \ge 2$.

It follows from \eqref{brv} that
\begin{align}
\frac12 \dt \big( \| \br v \|_{L^2}^2 \big)
&=
- \eps_1 \| \dx \br v \|_{L^2}^2
\notag
\\
&\quad
-(\eps_1-\eps_2) \Re ( \dx v^{\eps_2}, \dx \br v)_{L^2}
\notag
\\
&\quad
+
\Re \big( i (P_{\neq 0} \Im \Ta_{\o}^{\eps_1}) \dx \br v, \br v \big)_{L^2}
\notag
\\
&\quad
+\Re \big( (\Re \Ta_{\o}^{\eps_1}) \dx \br v, \br v \big)_{L^2}
\notag
\\
&\quad
+ 
\Re \big( \Ta_{\cj \o}^{\eps_1} \cj {\dx \br v}, \br v \big)_{L^2}
\notag
\\
&\quad
+
\Re \big( i (P_{\neq 0} \Im \Ta_{\o}^{\eps_1}- P_{\neq 0} \Im \Ta_{\o}^{\eps_2}) \dx v^{\eps_2} , \br v \big)_{L^2}
\notag
\\
&\quad
+
\Re \big( (\Re \Ta_{\o}^{\eps_1}- \Re \Ta_{\o}^{\eps_2}) \dx v^{\eps_2} ,  \br v \big)_{L^2}
\notag
\\
&\quad
+
\Re \big( (\Ta_{\cj \o}^{\eps_1}- \Ta_{\cj \o}^{\eps_2}) \cj{\dx v^{\eps_2}}, \br v \big)_{L^2}
\notag
\\
&\quad
+
\Re ( \br R , \br v )_{L^2}
\notag
\\
&=:
-\eps_1 \| \dx \br v \|_{L^2}^2
+
\sum_{j=1}^8 \I_j .
\label{difv}
\end{align}

From Corollary \ref{cor:enebdur}, we have
\begin{equation}
\begin{aligned}
|\I_1|
\le
|\eps_1 - \eps_2| \| \dx v^{\eps_2} \|_{L^2} \| \dx \br v \|_{L^2}
\le C(\|\phi\|_{H^s})|\eps_1 - \eps_2|
\label{difI1}
\end{aligned}
\end{equation}
for $s \ge 2$.

For $\I_2$,
it follows from \eqref{diffK} that
\begin{equation}
\I_{2} = \frac 12 \br K_{1}.
\label{difI22}
\end{equation}
For $\I_{3}$,
Corollary \ref{cor:enebdur} yields that
\begin{equation}
\begin{aligned}
\I_{3}
&= 
\Re \int_\T (\Re \Ta_\o^{\eps_1})  ( \dx \br v) \cj{ \br v}dx
=
-\frac12 \int_\T \dx (\Re \Ta_\o^{\eps_1}) \cdot | \br v|^2 dx
\\
&\les
\|\dx \Ta_\o^{\eps_1}\|_{L^\infty} \| \br v\|_{L^2}^2
\le
C(\|\phi\|_{H^s}) \|\br v\|_{L^2}^2
\end{aligned}
\label{difI32}
\end{equation}
for $s>\frac 52$.
The same calculation as in \eqref{difI32} yields that
\begin{equation}
|\I_4| \le C(\|\phi\|_{H^s}) \|\br v\|_{L^2}^2.
\label{difI4}
\end{equation}

From Proposition \ref{prop:bili} and \eqref{difTa}, we have
\begin{equation}
\begin{aligned}
|\I_5| 
&\les
\| (P_{\neq 0} \Im \Ta_{\o}^{\eps_1}- P_{\neq 0} \Im \Ta_{\o}^{\eps_2}) \dx v^{\eps_2} \|_{L^2}
\|\br v\|_{L^2}
\\
&\les
\|\Ta_{\o}^{\eps_1}- \Ta_{\o}^{\eps_2}\|_{L^2} \|\dx v^{\eps_2} \|_{H^{s-2}} \|\br v\|_{L^2}
\\
&\le
C(\|\phi\|_{H^s}) \|\br u \|_{H^{1}} \|\br v\|_{L^2}
\end{aligned}
\label{difI5}
\end{equation}
for $s>\frac 52$.
The same calculation as in \eqref{difI5} impiles that
\begin{equation}
|\I_6| + |\I_7|
\le
C(\|\phi\|_{H^s}) \|\br u \|_{H^{1}} \|\br v\|_{L^2}.
\label{difI6I7}
\end{equation}
From \eqref{brR} and \eqref{difTa}, we have
\begin{equation}
|\I_8|
\le
C(\|\phi\|_{H^s}) \|\br u \|_{H^{1}} \|\br v\|_{L^2}.
\label{difI8}
\end{equation}

Collecting \eqref{difv}--\eqref{difI8} and Lemma \ref{lem:dimodi}, we obtain that
\begin{equation}
\begin{aligned}
\dt  \Big(\| \br v\|_{L^2}^2 + \sum_{n=1}^N \br L_n \Big)
&\le
-2 \eps_1 \| \dx \br v \|_{L^2}^2
+
2 \sum_{j=1}^8 \I_j - \br K_1
\\
&\quad 
+ \sum_{n=1}^N \Big( \br K_n + \dt  \br L_n - \br K_{n+1} \Big) + \br K_{N+1}
\\
&\le
- \eps_1  \| \dx \br v \|_{L^2}^2
\\
&\quad
+ C(\|\phi\|_{H^s}) \big( |\eps_1 -\eps_2| + \|\br u \|_{H^{1}} \|\br v\|_{L^2} \big)
\\
&\quad
+ \br K_{N+1}.
\end{aligned}
\label{difvLn}
\end{equation}  

From
\eqref{diffK} and 
Corollary \ref{cor:enebdur},
the same calculation as in \eqref{exKN} yields that
\begin{equation}
\begin{aligned}
|\br K_{N+1}|
\le
C(\|\phi\|_{H^s}) \|\br v\|_{L^2}^2.
\end{aligned}
\label{difKN}
\end{equation}
Combining \eqref{difvLn}, \eqref{difKN}, and \eqref{diff2a},
we obtain that
\begin{equation}
\begin{aligned}
\dt  \Big(\| \br v\|_{L^2}^2 + \sum_{n=1}^N \br L_n \Big)
\le
C(\|\phi\|_{H^s})
\big( |\eps_1 - \eps_2|  + \br E^2 \big).
\end{aligned}
\label{difene}
\end{equation}
By \eqref{brEs}, \eqref{diff2a}, \eqref{difu}, and \eqref{difene}, we get \eqref{diineq}.
\end{proof}

From
Proposition \ref{prop:diff1} with \eqref{diff2a},
we have
\[
u = \lim_{\eps \to +0} u^\eps \in C([0,T]; H^{1}(\T)),
\]
where $u$ is the solution to \eqref{fNLS} constructed at the end of Subsection \ref{SUBSEC:energy}.
Moreover,
we also have
\[
u \in C ([0,T]; H^{s-\dl}(\T))
\]
for any $\dl>0$.
Indeed, for $\dl \in (0,s-1) $ and $t_1,t_2 \in [0,T]$,
 an interpolation yields that
\[
\| u(t_1) - u(t_2) \|_{H^{s-\dl}}
\le
\| u(t_1) - u(t_2) \|_{H^1}^{\frac \dl{s-1}} \| u(t_1) - u(t_2) \|_{H^{s}}^{1-\frac \dl{s-1}}.
\]
Since
\[
u \in C([0,T]; H^{1}(\T)) \cap L^\infty([0,T]; H^s(\T)),
\]
we obtain that $u \in C ([0,T]; H^{s-\dl}(\T))$ for $\dl>0$.

\begin{remark}
\label{rem:pers}
\rm
If $\phi \in H^r(\T)$ for some $r \ge s$,
the solution $u$ to \eqref{fNLS} constructed above satisfies
\[
u \in L^\infty([0,T]; H^r(\T)).
\]
Moreover,
we have the bound
\begin{equation}
\| u \|_{L_T^\infty H^r} \le C_2 ( \| \phi \|_{H^s}) (1+\| \phi \|_{H^r}),
\label{est:solbd1}
\end{equation}
where $C_2$ is the continuous function appeared in Corollary \ref{cor:enebdur}.
In addition,
the interpolation argument as above yields that
\[
u \in C ([0,T]; H^{r-\dl}(\T))
\]
for any $\dl>0$.
Note that $r=s+\frac \al2 + 1$ suffices for the argument in Subsection \ref{Subsec:code} below.
\end{remark}

A slight modification of the proof of Proposition \ref{prop:diff1} yields the following:

\begin{corollary}
\label{cor:diff2}
Let $\al>2$ and
\[
s>
\regwp,
\]
where $\regwp$ is defined in \eqref{condA}.
Let $u^{(j)} \in C([0,T]; H^s(\T))$ be a solution to \eqref{fNLS} in $H^s(\T)$ on $[0,T]$ with the initial data $\phi^{(j)} \in H^s(\T)$ for $j=1,2$.
Then,
there exists
\[
C_4( \| u^{(1)} \|_{L_T^\infty H^s}, \| u^{(2)} \|_{L_T^\infty H^s} )>0
\]
such that
\[
\| u^{(1)} - u^{(2)} \|_{L_T^\infty H^{1}}
\le C_4( \| u^{(1)} \|_{L_T^\infty H^s}, \| u^{(2)} \|_{L_T^\infty H^s} ) \| \phi^{(1)} - \phi^{(2)} \|_{H^{1}}.
\]
\end{corollary}

\begin{proof}
It follows from $u^{(j)} \in C([0,T]; H^s(\T))$ and Remark \ref{rem:sold} that
\[
v^{(j)} := \dx u^{(j)} \in C ([0,T]; H^{s-1}(\T)) \cap C^1([0,T]; H^{s-\al-1}(\T)).
\]
It follows from $s>\regwp$ that
\[
s-\al -1> -\frac \al2.
\]
Hence,
we have
\begin{align*}
\dt \big( \| v^{(j)} (t) \|_{L^2}^2 \big)
&= 2 \Re \big( \jb{D}^{-\frac \al2} \dt v^{(j)} (t), \jb{D}^{\frac \al 2} v^{(j)} (t) \big)_{L^2}
\end{align*}
for $t \in (0,T)$.
Hence, we can apply the same argument as in the proof of Proposition \ref{prop:diff1}
replacing $u^{\eps_1}$ and $u^{\eps_2}$ with $u^{(1)}$ and $u^{(2)}$, respectively.
\end{proof}

Corollary \ref{cor:diff2} yields the uniqueness of the solution to \eqref{fNLS}.


\subsection{Continuous dependence on initial data}
\label{Subsec:code}

To prove well-posedness,
we use the Bona-Smith approximation as in \cite{BoSm75}.
Let $s> \regwp$ and $\phi \in H^s(\T)$.
For $\mu \in \N$, we set 
\begin{equation}
\phi_\mu (x) := \sum_{\substack{k \in \Z \\ |k| \le \mu}} \ft \phi(k) e^{ikx}.
\label{BS_ini1}
\end{equation}
A simple calculation yields the following:

\begin{lemma}
\label{lem:BoSm}
The truncated function $\phi_\mu$ defined in \eqref{BS_ini1} satisfies the followings:
\begin{enumerate}
\item
$\phi_\mu \in C^{\infty}(\T)$ for any $\mu \in \N$.

\item
$\lim_{\mu \to \infty} \|\phi_\mu - \phi\|_{H^s} =0$.

\item
For $r>s$, we have $\| \phi_\mu \|_{H^r} \les \mu^{r-s} \| \phi \|_{H^s}$.

\item
For $r<s$, we have
$\|\phi_\mu - \phi \|_{H^r} \les \mu^{-(s-r)} \|\phi\|_{H^s}$.
\end{enumerate}
\end{lemma} 

Let $\mu \in \N$ and
let $u_\mu$ be the the solution to \eqref{fNLS} with the initial data
$\phi_\mu$.
Since $\phi_\mu \in C^\infty(\T)$,
Remark \ref{rem:pers} yields that
(at least)
\[
u_\mu \in C([0,T]; H^{s+\frac \al2}(\T)) \cap C^1([0,T]; H^{s- \frac \al2}(\T)).
\]
Set
\[
v_\mu = \dx u_\mu.
\]
For simplicity, we write
\begin{align*}
\Ta^{\mu}
= \Ta(u_\mu, v_\mu, \cj {u_\mu}, \cj {v_\mu}),
\quad
\Ta_{\zeta}^{\mu}
= \Ta_{\zeta} (u_\mu, v_\mu, \cj {u_\mu}, \cj {v_\mu}),
\quad
\cdots
\end{align*}

Similarly as in Subsection \ref{Subsec:estdiff},
for $\mu, \nu \in \N$ with $\mu>\nu$,
we define
\[
\wt u := u_{\mu} - u_{\nu},
\qquad
\wt v := \dx \wt u,
\]
for short.
Then, $\wt u$ satisfies
\begin{align}
\dt \wt u + i D^\al \wt u
&= \Ta^{\mu} - \Ta^{\nu}.
\label{wtu}
\end{align}
Moreover, it follows from the same calculation as in \eqref{brv} that
\begin{equation}
\begin{aligned}
\dt \wt v + i D^\al \wt v
&=
i (P_{\neq 0} \Im \Ta_{\o}^{\mu}) \dx \wt v
+(\Re \Ta_\o^{\mu})\dx \wt v
+\Ta_{\cj \o}^{\mu} \cj{\dx \wt v}
\\
&\quad
+i (P_{\neq 0} \Im \Ta_{\o}^{\mu} - P_{\neq 0} \Im \Ta_{\o}^{\nu}) \dx v_{\nu}
\\
&\quad
+ (\Re \Ta_{\o}^{\mu} - \Re \Ta_{\o}^{\nu}) \dx v_{\nu}
\\
&\quad
+ (\Ta_{\cj \o}^{\mu} - \Ta_{\cj \o}^{\nu}) \cj {\dx v_{\nu}}
+ \wt R,
\end{aligned}
\label{wtv}
\end{equation}
where 
\begin{equation}
\wt R
:=
\Ta_{\zeta}^{\mu} \wt v
+ \Ta_{\cj \zeta}^{\mu} \cj {\wt v}
+ (\Ta_{\zeta}^{\mu} - \Ta_{\zeta}^{\nu}) v_{\nu}
+ ( \Ta_{\cj \zeta}^{\mu} - \Ta_{\cj \zeta}^{\nu}) \cj {v_{\nu}}.
\label{wtR}
\end{equation}

Set
\begin{align}
\wt L_n^s (t)
&:=
c_n \int_\T (\dx^{-1} \Im \Ta_\o^{\mu} (t))^n \big| \jb{D}^{s-1-\frac {(\al-2)n}2} \wt v (t) \big|^2 dx,
\label{wtLn}
\\
\wt K_n^s (t)
&:=
-\al c_n \Im \int_\T \dx \big( (\dx^{-1} \Im \Ta_\o^{\mu} (t))^n  \big)
\notag
\\
&\hspace*{40pt}
\times
\big(\jb{D}^{s-1-\frac {(\al-2)(n-1)}2} \dx \wt v (t) \big) \jb{D}^{s-1-\frac {(\al-2)(n-1)}2} \cj{\wt v (t)} dx
\label{wtKn}
\end{align}
for $s> \frac 32$ and $n \in \N$,
where $c_n$ is defined in \eqref{eq:Lncon1}.

By \eqref{est:Lnes} with $r=s$,
there exists $\wt b_n = \wt b_n(\|u_{\mu}\|_{L_T^\infty H^s})>0$ such that
\begin{equation*}
\begin{aligned}
\big|\wt L_n^{s} (t) \big| 
&\le
\frac 1{10 n^2} \| \wt v  (t) \|_{H^{s-1}}^2
+
\wt b_n \| \wt u (t) \|_{H^{s-1}}^2
\end{aligned}
\end{equation*}
for $n=1,\dots, N$,
where $N$ is defined in \eqref{conN}.
Hence, we have
\begin{equation}
\bigg| \sum_{n=1}^N L_n^{s} (t) \bigg|
\le
\frac 12
\| \wt v (t) \|_{H^{s-1}}^2
+
\sum_{n=1}^N \wt b_n \| \wt u (t) \|_{H^{s-1}}^2.
\label{WPL1}
\end{equation}

Set
\[
\wt b := 2 \sum_{n=1}^N \wt b_n.
\]
We then define
\[
\wt E^s (t) 
:=
\Big( \wt b \|\wt u (t)\|_{H^{s-1}}^2 + \| \wt v (t)\|_{H^{s-1}}^2 + \sum_{n=1}^N \wt L_n^s (t)
\Big)^{\frac12}.
\]
By \eqref{WPL1},
we have
\begin{equation}
\frac{\wt b}2
\|\wt u(t)\|_{H^{s-1}}^2 + \frac 12 \| \wt v(t)\|_{H^{s-1}}^2
\le
\wt E^{s}(t)^2
\le
2 \wt b
\|\wt u(t)\|_{H^{s-1}}^2 + 2 \| \wt v(t)\|_{H^{s-1}}^2
\label{diff2b}
\end{equation}
for $t \in [0,T]$.

We prove the following lemma.

\begin{lemma}
\label{lem:comodi}
Let $\al > 2$ and $\mu > \nu$.
Suppose that
\[
s>\regwp,
\quad
\wt s > \frac 32
\]
where $\regwp$ is defined in \eqref{condA}.
Then,
we have
\[
\begin{aligned}
&\wt K_n^s(t) + \dt \wt L_n^s(t) - \wt K_{n+1}^s(t)
\\
&\le
C( \| \phi \|_{H^s} )
( \|\wt u(t)\|_{H^s}
+
\| \wt u (t) \|_{H^{\wt s}} \| u_\nu (t) \|_{H^{s+1}}
)
\| \wt v(t)\|_{H^{s-1}}
\end{aligned}
\]
for any $n = 1, \cdots, N$ and $t \in (0,T)$,
where $N$ is defined in \eqref{conN}.
\end{lemma}

\begin{proof}
For simplicity, we suppress the time dependence in this proof.
Fix $n =1,\dots,N$.
It follows from \eqref{wtLn} that
\begin{equation}
\begin{aligned}
&\dt \wt L_n^s
\\
&=
2 c_n \Re \int_\T ( \dx^{-1} \Im \Ta_\o^{\mu} )^n
\big( \jb{D}^{s-1-\frac {(\al-2)n}2} \wt v \big) \jb{D}^{s-1-\frac {(\al-2)n}2} \cj{\dt \wt v}dx
\\
&\quad
+
c_n \int_\T \dt ( ( \dx^{-1} \Im \Ta_\o^{\mu} )^n )
\big| \jb{D}^{s-1-\frac {(\al-2)n}2} \wt v \big|^2 dx
\\
&=:
\wt J^s + \wt R_1^s.
\end{aligned}
\label{dtwtL}
\end{equation}
As in \eqref{remi1a}, with Corollary \ref{cor:enebdur}, we have
\begin{equation}
|\wt R_1^s| \le C(\| \phi \|_{H^s}) \|\wt v\|_{H^{s-1}}^2
\label{wtremi1a}
\end{equation}
for $s > \regwp$.

From \eqref{wtv}, we have
\begin{align}
\wt J^s
&=
-2 c_n \Im \int_\T \big( \dx^{-1}  \Im \Ta_\o^{\mu} \big)^n
\big( \jb{D}^{s-1-\frac {(\al-2)n}2} \wt v \big) \jb{D}^{s-1-\frac {(\al-2)n}2} \cj{D^\al \wt v}dx
\notag
\\
&\quad
+
2 c_n \Im \int_\T
\big( \dx^{-1}  \Im \Ta_\o^{\mu} \big)^n
(P_{\neq 0} \Im \Ta_\o^{\mu}) 
\notag
\\
&\hspace*{60pt}
\times
\big( \jb{D}^{s-1-\frac {(\al-2)n}2} \wt v \big)
\jb{D}^{s-1-\frac {(\al-2)n}2} \cj{\dx \wt v}dx
\notag
\\
&\quad
+
2 c_n \Re \int_\T (\dx^{-1} \Im \Ta_\o^{\mu})^n (\Re \Ta_\o^{\mu})
\notag
\\
&\hspace*{60pt}
\times
\big( \jb{D}^{s-1-\frac {(\al-2)n}2} \wt v \big) \jb{D}^{s-1-\frac {(\al-2)n}2} \cj{\dx \wt v}dx
\notag
\\
&\quad
+
2 c_n \Re \int_\T (\dx^{-1}  \Im \Ta_\o^{\mu})^n \cj{\Ta_{\cj \o}^{\mu}} 
\big( \jb{D}^{s-1-\frac {(\al-2)n}2} \wt v \big) \jb{D}^{s-1-\frac {(\al-2)n}2} \dx \wt v dx
\notag
\\
&\quad
+
2 c_n \Im \int_\T
\big( \dx^{-1}  \Im \Ta_\o^{\mu} \big)^n
\big( \jb{D}^{s-1-\frac {(\al-2)n}2} \wt v \big)
\notag
\\
&\hspace*{60pt}
\times
\jb{D}^{s-1-\frac {(\al-2)n}2} \big( (P_{\neq 0} \Im \Ta_\o^{\mu} - P_{\neq 0} \Im \Ta_\o^{\nu}) \cj{\dx v_\nu} \big) dx
\notag
\\
&\quad
+
2 c_n \Re \int_\T (\dx^{-1}  \Im \Ta_\o^{\mu})^n 
\big( \jb{D}^{s-1-\frac {(\al-2)n}2} \wt v \big)
\notag
\\
&\hspace*{60pt}
\times
\jb{D}^{s-1-\frac {(\al-2)n}2}
\big( (\Re \Ta_\o^{\mu} - \Re \Ta_\o^{\nu}) \cj{\dx v_\nu} \big) dx
\notag
\\
&\quad
+
2 c_n \Re \int_\T (\dx^{-1}  \Im \Ta_\o^{\mu})^n
\big( \jb{D}^{s-1-\frac {(\al-2)n}2} \wt v \big)
\notag
\\
&\hspace*{60pt}
\times
\jb{D}^{s-1-\frac {(\al-2)n}2}
\big( (\cj{\Ta_{\cj \o}^{\mu}} -\cj{\Ta_{\cj \o}^{\nu}})  \dx v_\nu \big) dx
\notag
\\
&\quad
+
\wt R_2^s
\notag
\\
&=:
\sum_{\l=1}^7 \wt J_\l^s + \wt R_2^s,
\label{wtJ1}
\end{align}
where
\begin{align}
\wt R_2^s
&=
2 c_n \Im \int_\T (\dx^{-1}  \Im \Ta_\o^{\mu})^n \big( \jb{D}^{s-1-\frac {(\al-2)n}2} \wt v \big)
\notag
\\
&\hspace*{50pt}
\times
\big[ \jb{D}^{s-1-\frac {(\al-2)n}2},  P_{\neq 0} \Im \Ta_\o^{\mu} \big] \cj{\dx \wt v} dx
\notag
\\
&\quad
+
2 c_n \Re \int_\T (\dx^{-1}  \Im \Ta_\o^{\mu})^n \big( \jb{D}^{s-1-\frac {(\al-2)n}2} \wt v \big)
\notag
\\
&\hspace*{50pt}
\times
\big[ \jb{D}^{s-1-\frac {(\al-2)n}2}, \Re \Ta_\o^{\mu} \big] \cj{\dx \wt v}dx
\notag
\\
&\quad
+
2 c_n \Re \int_\T (\dx^{-1}  \Im \Ta_\o^{\mu})^n \big( \jb{D}^{s-1-\frac {(\al-2)n}2} \wt v \big)
\big[\jb{D}^{s-1-\frac {(\al-2)n}2}, \cj{\Ta_{\cj \o}^{\mu}}\big] \dx \wt v dx
\notag
\\
&\quad
+
2 c_n \Re \int_\T (\dx^{-1}  \Im \Ta_\o^{\mu})^n 
\big( \jb{D}^{s-1-\frac {(\al-2)n}2} \wt v \big) \jb{D}^{s-1-\frac {(\al-2)n}2} \cj{\wt R} dx
\notag
\\
&=:
\sum_{\l=1}^4 \wt R_{2,\l}^s.
\label{wtR2}
\end{align}
We consider $\wt R_2^s$. From Proposition \ref{prop:commC}, we have
\begin{equation*}
\big\| \big[ \jb{D}^{s-1-\frac {(\al-2)n}2},  \Im \Ta_\o^{\mu} \big] \dx \wt v \big\|_{L^2}
\les
\| \Ta_\o^{\mu} \|_{H^{s-1}} \| \wt v \|_{H^{s-1}} 
\le
C( \| \phi \|_{H^s} ) \| \wt v \|_{H^{s-1}} 
\end{equation*}
for $s > \regwp$.
Hence, we obtain
\begin{equation}
\begin{aligned}
| \wt R_{2,1}^s | 
&\les
\| \dx^{-1}  \Im \Ta_\o^{\mu} \|_{L^\infty}^n
\big\| \jb{D}^{s-1-\frac {(\al-2)n}2} \wt v \big\|_{L^2}
\\
&
\hspace*{100pt}
\times
\big\| \big[ \jb{D}^{s-1-\frac {(\al-2)n}2},  \Im \Ta_\o^{\mu} ] \dx \wt v \big\|_{L^2} 
\\
&\le
C( \| \phi \|_{H^s} ) \| \wt v \|_{H^{s-1}}^2
\end{aligned}
\label{wtR21}
\end{equation} 
for $s > \regwp$.
Similarly, we have
\begin{equation}
| \wt R_{2,2}^s| + |\wt R_{2,3}^s | 
\le
C( \| \phi \|_{H^s} ) \| \wt v \|_{H^{s-1}}^2.
\label{wtR2223}
\end{equation}
It follows from \eqref{wtR} and Corollary \ref{cor:bili3} that
\begin{equation}
\begin{aligned}
| \wt R_{2,4}^s | 
&\les
\| \dx^{-1}  \Im \Ta_\o^{\mu} \|_{L^\infty}^n \| \jb{D}^{s-1} \wt v\|_{L^2} \| \wt R \|_{H^{s-1}}
\\
&\le
C( \| \phi \|_{H^s} ) \|\wt u \|_{H^s} \| \wt v \|_{H^{s-1}}.
\end{aligned}
\label{wtR27}
\end{equation}  
Collecting \eqref{wtR2}--\eqref{wtR27}, we get
\begin{equation}
| \wt R_2^s | 
\le
C( \| \phi \|_{H^s} ) \| \wt u \|_{H^s} \| \wt v \|_{H^{s-1}}.
\label{wtR2z}
\end{equation} 
for $s>\regwp$.

For $\wt J_{\l}$
with $\l \in \{1,2,3,4\}$,
the following estimate can be obtained by a calculation similar to that for \eqref{J1Kn}, \eqref{J3}--\eqref{J5}, respectively:
\begin{equation}
\begin{aligned}
\wt K_n^s + \sum_{\l=1}^4 \wt J_\l^s - \wt K_{n+1}^s
\le
C(\| \phi \|_{H^s}) \| \wt v \|_{H^{s-1}}^2.
\end{aligned}
\label{coJl}
\end{equation}

We consider the remaining parts.
First, for $\wt J_5^s$ in \eqref{wtJ1}, we show that
\begin{equation}
|\wt J_5^s|
\le
C(\|\phi\|_{H^s})
\big( \| \wt u \|_{H^s} + \| \wt u \|_{H^{\wt s}} \| u_\nu\|_{H^{s+1}} \big) \|\wt v\|_{H^{s-1}}
\label{coJ15-3}
\end{equation}
for $s> \frac 52$ and $\wt s> \frac 32$.

Proposition \ref{prop:bili} and Corollary \ref{cor:bili3} yield that
\begin{equation}
\begin{aligned}
&\big\| (\dx^{-1} \Im \Ta_\o^{\mu})^n \jb{D}^{s-1-\frac{(\al-2)n}2} \wt v \big\|_{L^2}
\\
&\le
\big\| (\dx^{-1} \Im \Ta_\o^{\mu})^n \big\|_{H^{s-1}} \big\| \jb{D}^{s-1-\frac{(\al-2)n}2} \wt v \big\|_{L^2}
\\
&\le
C(\|\phi\|_{H^s})
\|\wt v\|_{H^{s-1}}
\end{aligned}
\label{coJ15-3a}
\end{equation}
for $s>\frac 32$.
Moreover,
Proposition \ref{prop:bili2} and Corollary \ref{cor:bili3} that
\begin{equation}
\begin{aligned}
&\| ( P_{\neq 0} \Im \Ta_\o^{\mu} - P_{\neq 0} \Im\Ta_\o^{\nu} ) \dx v_{\nu} \|_{H^{s-1}} 
\\
&\les
\| \Ta_\o^\mu - \Ta_\o^\nu \|_{H^{s-1}} \| \dx v_{\nu} \|_{H^{\frac 12+\dl}}
+
\| \Ta_\o^\mu - \Ta_\o^\nu \|_{H^{\frac 12+\dl}} \| \dx v_{\nu} \|_{H^{s-1}}
\\
&\le
C(\|\phi\|_{H^s}) \big( \| \wt u \|_{H^s} + \| \wt u \|_{H^{\wt s}} \| u_\nu\|_{H^{s+1}} \big)
\end{aligned}
\label{coJ15-3b}
\end{equation}
for $s>\frac 52$, $\wt s > \frac 32$, and $0<\dl \ll 1$.
With \eqref{wtJ1}, \eqref{coJ15-3a}, and \eqref{coJ15-3b},
we have
\begin{align*}
|\wt J_5^s|
&\le
\big\| (\dx^{-1} \Im \Ta_\o^{\mu})^n \jb{D}^{s-1-\frac{(\al-2)n}2} \wt v \big\|_{L^2}
\\
&\quad
\times
\| ( P_{\neq 0} \Im \Ta_\o^{\mu} - P_{\neq 0} \Im\Ta_\o^{\nu} ) \dx v_{\nu} \|_{H^{s-1}}
\\
&\le
C(\|\phi\|_{H^s}) \big( \| \wt u \|_{H^s} + \| \wt u \|_{H^{\wt s}} \| u_\nu\|_{H^{s+1}} \big) \|\wt v\|_{H^{s-1}}
\end{align*}
for $s>\frac 52$, $\wt s > \frac 32$,
which shows \eqref{coJ15-3}.

With \eqref{wtJ1},
the same calculation as in \eqref{coJ15-3} yields that
\begin{equation}
|\wt J_6^s| + |\wt J_7^s|
\le
C(\|\phi\|_{H^s})
\big( \| \wt u \|_{H^s} + \| \wt u \|_{H^{\wt s}} \| u_\nu\|_{H^{s+1}} \big) \|\wt v\|_{H^{s-1}}.
\label{coJ1617}
\end{equation}
Combining \eqref{dtwtL}--\eqref{wtJ1}, \eqref{wtR2z}--\eqref{coJ15-3}, and \eqref{coJ1617},
we obtain the desired bound.
\end{proof}

The following energy estimate  plays a key role.

\begin{proposition}
\label{prop:code2}
Let
\[
s> \regwp,
\quad
\wt s > \frac 32,
\]
where $\regwp$ is defined in \eqref{condA}.
Then, there exists $C( \|\phi\|_{H^s})>0$
such that
\[
\dt \wt E^s (t)
\le
C( \|\phi\|_{H^s})
\big(
\wt E^s (t)
+
\| \wt u (t) \|_{H^{\wt s}} \| u_\nu (t) \|_{H^{s+1}}
\big)
\]
for $t \in (0,T)$ and $\mu,\nu \in \N$ with $\mu>\nu$.
\end{proposition}

\begin{proof}

We prove that there exists $C(\|\phi\|_{H^s})>0$ such that
\begin{equation}
\dt  ( \wt E^s (t)^2)
\le 
C(\|\phi\|_{H^s}) \big( \wt E^s(t) + \| \wt u (t) \|_{H^{\wt s}} \| u_\nu (t) \|_{H^{s+1}} \big)
\wt E^s(t).
\label{conineq}
\end{equation}
We suppress time dependence for short.
From \eqref{wtu}, Proposition \ref{prop:bili},  and \eqref{est:solbd1}, we have
\begin{equation}
\begin{aligned}
\dt \big( \|\wt u\|_{H^{s-1}}^2 \big)
&\le
2 \|\Ta^\mu - \Ta^\nu\|_{H^{s-1}} \|\wt u\|_{H^{s-1}}
\\
&\le
C(\|\phi\|_{H^s}) \|\wt u\|_{H^{s}} \|\wt u\|_{H^{s-1}} .
\end{aligned}
\label{codeu}
\end{equation}

By \eqref{wtv}, we have
\begin{align}
\frac12 \dt \big( \| \wt v\|_{H^{s-1}}^2 \big)
&=
\Re \big( \jb{D}^{s-1}\big (i (P_{\neq 0} \Im \Ta_{\o}^{\mu}) \dx \wt v \big), \jb{D}^{s-1}\wt v \big)_{L^2}
\notag
\\
&\quad
+\Re \big( \jb{D}^{s-1}\big ((\Re \Ta_{\o}^{\mu}) \dx \wt v \big), \jb{D}^{s-1}\wt v \big)_{L^2}
\notag
\\
&\quad
+ 
\Re \big( \jb{D}^{s-1}(\Ta_{\cj \o}^{\mu} \cj {\dx \wt v}), \jb{D}^{s-1} \wt v \big)_{L^2}
\notag
\\
&\quad
+
\Re \big( \jb{D}^{s-1} \big( i (P_{\neq 0} \Im \Ta_{\o}^{\mu}- P_{\neq 0} \Im \Ta_{\o}^{\nu}) \dx v_{\nu} \big), \jb{D}^{s-1}\wt v \big)_{L^2}
\notag
\\
&\quad
+
\Re \big( \jb{D}^{s-1} \big( (\Re \Ta_{\o}^{\mu}- \Re \Ta_{\o}^{\nu}) \dx v_{\nu} \big), \jb{D}^{s-1}\wt v \big)_{L^2}
\notag
\\
&\quad
+
\Re \big( \jb{D}^{s-1} \big( (\Ta_{\cj \o}^{\mu}- \Ta_{\cj \o}^{\nu}) \cj{\dx v_{\nu}}\big) , \jb{D}^{s-1}\wt v \big)_{L^2}
\notag
\\
&\quad
+
\Re \big( \jb{D}^{s-1} \wt R , \jb{D}^{s-1} \wt v \big)_{L^2}
\notag
\\
&=:
\sum_{j=1}^7 \I_j .
\label{normv}
\end{align}

We decompose $\I_1$ as follows:
\begin{equation*}
\begin{aligned}
\I_1 
&= 
-\Im \big( \big[ \jb{D}^{s-1}, P_{\neq 0} \Im \Ta_\o^\mu \big] \dx \wt v, \jb{D}^{s-1} \wt v \big)_{L^2}
\\
&\quad
-\Im \big( (P_{\neq 0} \Im \Ta_\o^\mu) \jb{D}^{s-1} \dx \wt v, \jb{D}^{s-1} \wt v \big)_{L^2}
\\
&=: \I_{1,1} + \I_{1,2}.
\end{aligned}
\end{equation*}
Proposition \ref{prop:commC} and Corollary \ref{cor:bili3} imply that
\[
|\I_{1,1}|
\les
\| P_{\neq 0} \Im \Ta_\o^\mu \|_{H^{s-1}} \|\wt v\|_{H^{s-1}} \| \wt v\|_{H^{s-1}}
\le
C(\|\phi\|_{H^s}) \|\wt v\|_{H^{s-1}}^2
\]
for $s>\frac 52$.
By \eqref{wtKn}, we have
\[
\I_{1,2} = \frac 12 \wt K_1^s.
\]
Hence, we obtain that
\begin{equation}
\Big| \I_1 - \frac 12 \wt K_1^s \Big| \le C(\|\phi\|_{H^s}) \|\wt v\|_{H^{s-1}}^2.
\label{coI1}
\end{equation}

With \eqref{normv},
the same calculation as in \eqref{exI3} yields that
\begin{equation}
|\I_2| + |\I_3| \le C(\|\phi\|_{H^s}) \| \wt u \|_{H^s} \|\wt v\|_{H^{s-1}}
\label{coI2I3}
\end{equation}
for $s>\frac 52$.

We consider $\I_4$ in \eqref{normv}.
It follows from \eqref{coJ15-3b} that
\begin{equation}
\begin{aligned}
|\I_4|
&\le
\big\| \big( P_{\neq 0} \Im \Ta_\o^\mu - P_{\neq 0} \Im \Ta_\o^\nu \big) \dx v_{\nu} \big\|_{H^{s-1}}
\| \wt v \|_{H^{s-1}}
\\
&\le
C(\|\phi\|_{H^s}) \big( \| \wt u \|_{H^s} + \| \wt u \|_{H^{\wt s}} \| u_\nu\|_{H^{s+1}} \big) \|\wt v\|_{H^{s-1}}
\end{aligned}
\label{coI4}
\end{equation}
for $s>\frac 52$ and $\wt s > \frac 32$.
The same calculation as in \eqref{coI4} yields that
\begin{equation}
|\I_5| + |\I_6| \le C(\|\phi\|_{H^s})
\big( \| \wt u \|_{H^s} + \| \wt u \|_{H^{\wt s}} \| u_\nu\|_{H^{s+1}} \big) \|\wt v\|_{H^{s-1}}.
\label{coI5I6}
\end{equation}
By \eqref{normv}, \eqref{difTa}, and \eqref{wtR}, we obtain 
\begin{equation}
|\I_7|
\le
\|R_2^\mu - R_2^\nu\|_{H^{s-1}} \| \wt v\|_{H^{s-1}}
\le
C(\|\phi\|_{H^s}) \| \wt u \|_{H^s} \|\wt v\|_{H^{s-1}}.
\label{coI7}
\end{equation}

It follows from Lemma \ref{lem:comodi}, \eqref{normv}--\eqref{coI7} that
\begin{equation}
\begin{aligned}
&\dt  \Big(\| \wt v\|_{H^{s-1}}^2 + \sum_{n=1}^N \wt L_n^{s} \Big)
\\
&\le
2 \sum_{j=1}^7 \I_j - \wt K_1^{s}
+\Big(\sum_{n=1}^N \wt K_n^{s} + \dt  \wt L_n^{s} - \wt K_{n+1}^{s} \Big) + \wt K_{N+1}^{s}
\\
&\le
C(\|\phi\|_{H^s})
\big( \| \wt u \|_{H^s} + \| \wt u \|_{H^{\wt s}} \| u_\nu\|_{H^{s+1}} \big) \|\wt v\|_{H^{s-1}} 
+ \wt K_{N+1}^{s}
\end{aligned}
\label{cowtv}
\end{equation} 
for $s > \regwp$.
Moreover,
with \eqref{wtKn},
the same calculation as in \eqref{exKN} yields that
\begin{equation}
\begin{aligned}
|K_{N+1}^{s}|
&\le
C(\|\phi\|_{H^s}) \|\wt v\|_{H^{s-1}}^2.
\end{aligned}
\label{coKn}
\end{equation}
Combining \eqref{cowtv}, \eqref{coKn}, and \eqref{diff2b},
we get
\begin{equation}
\dt  \Big(\| \wt v\|_{H^{s-1}}^2 + \sum_{n=1}^N \wt L_n^{s} \Big) 
\le
C(\|\phi\|_{H^s})
\big(\wt E^s + \| \wt u \|_{H^{\wt s}} \| u_\nu\|_{H^{s+1}} \big) \wt E^s .
\label{cowtv2}
\end{equation}
From \eqref{codeu} and \eqref{cowtv2}, we get \eqref{conineq},
which completes the proof.
\end{proof}

Remark \ref{rem:pers} and Lemma \ref{lem:BoSm} yield that
\[
\|u_\nu\|_{L_T^\infty H^{s+1}}
\le
C(\|\phi\|_{H^s})(1 + \|\phi_\nu\|_{H^{s+1}})
\le
C(\|\phi\|_{H^s})
\nu.
\]
Moreover,
from Corollary \ref{cor:diff2} and Lemma \ref{lem:BoSm} with $\mu > \nu$,
we have
\begin{equation}
\begin{aligned}
\| \wt u \|_{L_T^\infty H^{\wt s}}
&\le
\| \wt u \|_{L_T^\infty H^1}^{\frac{s-\wt s}{s-1}} \| \wt u \|_{L_T^\infty H^s}^{\frac{\wt s-1}{s-1}}
\le
C(\|\phi\|_{H^s})
\|\phi_\mu - \phi_\nu\|_{H^1}^{\frac{s-\wt s}{s-1}}
\\
&\le
C(\|\phi\|_{H^s})
\nu^{-(s-\wt s)}
\end{aligned}
\label{BSintaa}
\end{equation}
for $\wt s \in (1, s)$.
Hence,
Proposition \ref{prop:code2} yields that
\[
\wt E_s(t)
\le
\big(\wt E_s(0) + \nu^{1-(s-\wt s)} \big) e^{C(\| \phi \|_{H^s}) t}
\]
for $s>\regwp$ and $\wt s \in (\frac{3}2, s)$.
From \eqref{diff2b},
we have
\begin{equation}
\| u_\mu - u_\nu \|_{L_T^\infty H^s}
\le
C(\| \phi \|_{H^s})
\big( \| \phi_\mu - \phi_\nu \|_{H^s} + \nu^{1-(s-\wt s)} \big)
\label{BSdiffa}
\end{equation}
for $\mu>\nu$.
Here,
we can choose $\wt s$ such that
\begin{equation}
\frac{3}2 < \wt s < s-1,
\label{wtscond1}
\end{equation}
which follows from $s> \frac 52$.
Then,
$\{ u_\mu \}_{\mu \in \N}$ converges to $u$ constructed at the end of Subsection \ref{SUBSEC:energy} in $C([0,T]; H^s(\T))$.
In particular, we have $u \in C([0,T]; H^s(\T))$.

Finally, we prove the well-posedness result in Theorem \ref{thm:equiv}.

\begin{proof}[Proof of the well-posedness for \eqref{fNLS}]
Let $\{ \phi^{(\l)} \}_{\l \in \N} \subset H^s(\T)$ converge to $\phi$ in $H^s(\T)$.
Moreover, let $\phi^{(\l)}_\mu$ be the truncation of $\phi^{(\l)}$ defined in \eqref{BS_ini1}.
Set $u^{(\l)}$ and $u^{(\l)}_\mu$ as the solutions to \eqref{fNLS} with the initial data  $\phi^{(\l)}$ and $\phi^{(\l)}_\mu$, respectively.

First, we prove
\begin{equation}
\lim_{\l \to \infty} u_\mu^{(\l)} = u_\mu
\label{WPumu}
\end{equation}
in $C([0,T];H^s(\T))$ for $\mu \in \N$.
Let $\wt s$ satisfy \eqref{wtscond1}.
Then,
the same calculation as in the proof of Proposition \ref{prop:code2} yields that
\begin{align*}
&\|u_\mu^{(\l)} - u_\mu\|_{L_T^\infty H^s}
\\
&\le
\Big( \|\phi_\mu^{(\l)} - \phi_\mu \|_{H^s}
\\
&\qquad
+
\|u_\mu^{(\l)} - u_\mu\|_{L_T^\infty H^{\wt s}}
\big( \|u_\mu^{(\l)}\|_{L_T^\infty H^{s+1}} + \|u_\mu\|_{L_T^\infty H^{s+1}} \big) 
\Big) e^{C(\|\phi\|_{H^s})T}.
\end{align*}
From Remark \ref{rem:pers}, we have
\begin{align*}
\|u_\mu^{(\l)}\|_{L_T^\infty H^{s+1}} + \|u_\mu\|_{L_T^\infty H^{s+1}}
&\le
C(\|\phi_\mu\|_{H^s})(1 + \|\phi_\mu^{(\l)}\|_{H^{s+1}} + \|\phi_\mu\|_{H^{s+1}})
\\
&\le
C(\|\phi\|_{H^s}) \mu. 
\end{align*}
The same calculation as in \eqref{BSintaa} yields that
\[
\|u_\mu^{(\l)} - u_\mu\|_{L_T^\infty H^{\wt s}}
\le
C(\|\phi\|_{H^s})
\|\phi_\mu^{(\l)} - \phi_\mu \|_{H^1}^{\frac{s-\wt s}{s-1}}
.
\]
Hence, we have
\begin{equation*}
\|u_\mu^{(\l)} - u_\mu\|_{L_T^\infty H^s}
\le
C(\|\phi\|_{H^s})
\Big( \|\phi_\mu^{(\l)} - \phi_\mu \|_{H^s}
+
\mu
\|\phi_\mu^{(\l)} - \phi_\mu \|_{H^1}^{\frac{s-\wt s}{s-1}}
\Big).
\end{equation*}
Since $\phi_\mu^{(\l)}$ and $\phi_\mu$ are truncations of $\phi^{(\l)}$ and $\phi$, respectively,
we obtain \eqref{WPumu}.

The triangle inequality and \eqref{BSdiffa} show that
\begin{equation}
\begin{aligned}
&\| u^{(\l)} - u\|_{L_T^\infty H^s}
\\
&\le
\| u^{(\l)} - u_\mu^{(\l)}\|_{L_T^\infty H^s} + \|u_\mu^{(\l)} - u_\mu\|_{L_T^\infty H^s}
+\| u_\mu - u\|_{L_T^\infty H^s}
\\
&\le
C(\|\phi\|_{H^s})
\big( \|\phi^{(\l)} - \phi_\mu^{(\l)}\|_{H^s} + \|\phi_\mu - \phi\|_{H^s} + \mu^{1-(s-\wt s)} \big) 
\\
&\quad
+ \|u_\mu^{(\l)} - u_\mu\|_{L_T^\infty H^s}.
\end{aligned}
\label{WPdifu}
\end{equation}
By \eqref{BS_ini1}, we obtain that
\begin{equation}
\begin{aligned} 
\|\phi^{(\l)} - \phi_\mu^{(\l)}\|_{H^s}
&\le
\|\phi^{(\l)} - \phi\|_{H^s} + \|\phi - \phi_\mu \|_{H^s} + \|\phi_\mu - \phi_\mu^{(\l)}\|_{H^s}
\\
&\le
2 \|\phi^{(\l)} - \phi\|_{H^s} + \|\phi - \phi_\mu \|_{H^s}. 
\end{aligned}
\label{WPphi}
\end{equation}

From \eqref{WPumu}--\eqref{WPphi}, we obtain that
\[
\lim_{\l \to \infty} u^{(\l)} = u.
\]
in $C([0,T];H^s(\T))$.
Thus, the well-posedness result holds.
\end{proof}

\section{Non-existence}
\label{SEC:nonex}

In this section, we prove non-existence of a solution to \eqref{fNLS} when \eqref{Fbez2} holds.
The main result in this section is the following.

\begin{theorem}
\label{thm:NE2}

Let $\al > 2$ and  $s \in \R$ satisfy
\[
s> \regwp,
\]
where $\regwp$ is defined in
\eqref{condA}.
Assume that
$u \in C([0,T]; H^s (\T))$ is a solution to \eqref{fNLS}
for some $T>0$.

\begin{enumerate}
\item
If
\[
\int_\T \Im F_\o ( \phi,\dx \phi,\cj{\phi},\cj{\dx \phi} ) dx >0,
\]
then we have
$P_- \phi \in H^{s+\dl}(\T)$
for any  $\dl \in (0,\min (\al-2,1))$.

\item
If
\[
\int_\T \Im F_\o ( \phi,\dx \phi,\cj{\phi},\cj{\dx \phi} ) dx <0,
\]
then we have
$P_+ \phi \in H^{s+\dl} (\T)$
for any  $\dl \in (0,\min (\al-2,1))$.
\end{enumerate}

The same conclusion holds for a solution $u \in C([-T,0]; H^s(\T))$ by interchanging $P_- \phi$ and $P_+ \phi$.
\end{theorem}

Once Theorem \ref{thm:NE2} is obtained,
Theorem \ref{thm:nonexi} follows from choosing $\phi \in H^s(\T)$ with
$P_\pm \phi \notin H^{s+\dl}(\T)$ and
\[
\int_\T \Im F_\o ( \phi(x),\dx \phi(x),\cj{\phi(x)},\cj{\dx \phi(x)} ) dx
\neq 0.
\]
From the assumption of Theorem \ref{thm:nonexi},
we can choose such $\phi$ by an approximation argument.
See Lemma 4.2 in \cite{KoOk25b}.

\subsection{A priori estimates}
\label{subsec:apriori}

Let $u$ be the solution to \eqref{fNLS} in $C([0,T]; H^s (\T))$.
Set
\[
v = \dx u.
\]
For simplicity,
we write
\[
\Ta = F(u, v, \cj{u}, \cj{v}),
\quad
\Ta_{\zeta} = F_{\zeta} (u, v, \cj{u}, \cj{v}),
\quad
\dots.
\]
Then, $v$ satisfies
\begin{equation}
\begin{aligned}
\dt v + i D^\al v
&=
\Ta_{\zeta} v
+
\Ta_{\o} \dx v
+
\Ta_{\cj \zeta} \cj v
+
\Ta_{\cj \o} \cj{\dx v}
\\
&=
(P_0 \Ta_{\o}) \dx v
+
(P_{\neq 0} \Ta_{\o}) \dx v
+
\Ta_{\cj \o} \cj{\dx v}
+ R,
\end{aligned}
\label{eqv}
\end{equation}
where $R$ denotes a polynomial in
$u,v,\cj{u},\cj{v}$.

We can remove $\Re (P_0 \Ta_{\o}) \dx v$ from the first term on the right-hand side of \eqref{eqv}.
Indeed,
we define
\begin{align*}
\wt u  (t,x)
&:= u \bigg( t, x - \int_0^t \Re P_0 \Ta_\o(t') dt' \bigg),
\\
\wt v &:= \dx \wt u,
\quad
\wt \Ta := F( \wt u, \wt v, \cj{\wt u}, \cj{\wt v}).
\end{align*}
Then,
$\wt v$ satisfies
\begin{align*}
\dt \wt v + i D^\al \wt v
=
i \Im (P_0 \wt \Ta_{\o}) \dx \wt v
+
(P_{\neq 0} \wt \Ta_{\o}) \dx \wt v
+
\wt \Ta_{\cj \o} \cj{\dx \wt v}
+ \wt R.
\end{align*}

In the following, let $\wt u$ and $\wt v$ be simply denoted as $u$ and $v$.
That is, we assume that $v$ satisfies the following.
\begin{equation}
\begin{aligned}
\dt v + i D^\al v
=
i(\Im P_0 \Ta_{\o}) \dx v
+
(P_{\neq 0} \Ta_{\o}) \dx v
+
\Ta_{\cj \o} \cj{\dx v}
+ R.
\end{aligned}
\label{eqv2}
\end{equation}
Set
\[
\ft V(t,k)
=
e^{i|k|^\al t} \ft v(t,k).
\]
With \eqref{eqv2}, $V$ satisfies
\begin{equation}
\begin{aligned}
&\dt \ft V (t,k)
\\
&=
- (\Im P_0 \Ta_\o(t)) k \ft V(t,k)
+ 
e^{i|k|^\al t} \Ft \big[ (P_{\neq 0} \Ta_{\o}) {\dx v} \big] (t,k)
\\
&\quad
+ 
e^{i|k|^\al t} \Ft \big[ \Ta_{\cj \o} \cj{\dx v} \big] (t, k)
+
e^{i|k|^\al t} \ft{R}(t,k)
\\
&=:
- (\Im P_0 \Ta_\o(t)) k \ft V(t,k)
+
\Nl_1 (t,k)
+
\Nl_2 (t,k)
+
\Nl_3 (t,k).
\label{ftfW}
\end{aligned}
\end{equation}
We decompose $\Nl_1$ into three parts as follows:
\begin{equation}
\begin{aligned}
\Nl_1(t,k)
&=
i \sum_{k_1+k_2=k} e^{i(|k|^\al - {|k_2|}^\al)t}
\ft {P_{\neq 0} \Ta_{\o}} (t, k_1) k_2 \ft V (t,k_2)
\\
&=
i \bigg(
\sum_{D_1(k)}
+
\sum_{D_2(k)}
\bigg)
e^{i(|k|^\al - |k_2|^\al)t}
\ft {P_{\neq 0} \Ta_{\o}}(t, k_1) k_2 \ft V (t,k_2)
\\
&=:
\Nl_{1,1} (t,k)
+
\Nl_{1,2} (t,k),
\end{aligned}
\label{decN1}
\end{equation}
where
\begin{align*}
D_1(k)&:= \Big\{(k_1,k_2) \in \Z^2 \mid k_1 + k_2 = k, \, |k_1| \ge \frac{|k_2|}2 \Big\},
\\
D_2(k)&:= \Big\{(k_1,k_2) \in \Z^2 \mid k_1 + k_2 = k, \, |k_1| < \frac{|k_2|}2 \Big\}.
\end{align*}
Note that
\[
|k_2| \neq |k|
\]
for $(k_1,k_2) \in D_2(k)$ with $k_1 \neq 0$.
We further decompose $\Nl_{1,2}$ into two parts:
\begin{equation}
\begin{aligned}
\Nl_{1,2}
&= \dt \Ml_1 + \Kl_1,
\\
\Ml_1 (t,k)
&:=
\sum_{D_2(k)}
\frac{e^{i(|k|^\al - {|k_2|}^\al)t}}{|k|^\al - |k_2|^\al}
\ft{P_{\neq 0} \Ta_{\o}}(t,k_1) k_2 \ft V (t,k_2),
\\
\Kl_1 (t,k)
&:=
-\sum_{D_2(k)}
\frac{e^{i(|k|^\al - |k_2|^\al)t}}{|k|^\al - |k_2|^\al}
\dt \Big( \ft{P_{\neq 0} \Ta_{\o}}(t,k_1) k_2 \ft V (t,k_2) \Big).
\end{aligned}
\label{decN2}
\end{equation} 

For $\Nl_2$ in \eqref{ftfW}, we apply a similar calculation as in \eqref{decN1} and \eqref{decN2}. Namely, we decompose $\Nl_2$ as follows:
\begin{equation}
\begin{aligned}
\Nl_2(t,k)
&=
i \sum_{k_1+k_2=k} e^{i(|k|^\al + {|k_2|}^\al)t}
\ft{\Ta_{\cj \o}} (t, k_1) k_2 \cj{\ft V (t,-k_2)}
\\
&=
i \bigg(
\sum_{D_1(k)}
+
\sum_{D_2(k)}
\bigg)
e^{i(|k|^\al + {|k_2|}^\al)t}
\ft{\Ta_{\cj \o}} (t, k_1) k_2 \cj{\ft V (t,-k_2)}
\\
&=:
\Nl_{2,1} (t,k)
+
\Nl_{2,2} (t,k).
\end{aligned}
\label{decN3}
\end{equation}
We further decompose $\Nl_{2,2}$ into two parts:
\begin{equation}
\begin{aligned}
\Nl_{2,2}
&= \dt \Ml_2 + \Kl_2,
\\
\Ml_2 (t,k)
&:=
\sum_{D_2(k)}
\frac{e^{i(|k|^\al + {|k_2|}^\al)t}}{|k|^\al + |k_2|^\al}
\ft{ \Ta_{\cj \o}}(t,k_1) k_2 \cj{\ft V (t,-k_2)},
\\
\Kl_2 (t,k)
&:=
-
\sum_{D_2(k)}
\frac{e^{i(|k|^\al + |k_2|^\al)t}}{|k|^\al + |k_2|^\al}
\dt \Big(
\ft{ \Ta_{\cj \o}}(t,k_1) k_2 \cj{\ft V (t,-k_2)} \Big).
\end{aligned}
\label{decN4}
\end{equation} 

From \eqref{ftfW}--\eqref{decN4},
we have
\begin{equation}
\begin{aligned}
\dt \ft V (t,k)
&=
- (\Im P_0 \Ta_\o (t)) k \ft V (t,k)
\\
&\quad
+ 
\sum_{j=1}^2
\Big( \Nl_{j,1} (t,k)
+
\dt \Ml_j (t,k)
+
\Kl_j (t,k) \Big)
+
\Nl_3(t,k).
\end{aligned}
\label{ftfw2}
\end{equation}
The following is the main estimate in this subsection.

\begin{proposition}
\label{prop:FGH}
Let $\al > 1$ and  $s \in \R$ satisfy
\[
s>
\regwp,
\]
where $\regwp$ is defined in \eqref{condA}.
Then,
with the notation above,
there exists $C( \| u \|_{L_T^\infty H^s})>0$ such that the following estimates hold:
\begin{align}
\| \Nl_{1,1} \|_{L_T^\infty \l^2_{s-1}}
+ \|\Nl_{2,1}  \|_{L_T^\infty \l^2_{s-1}}
+ \| \Nl_3 \|_{L_T^\infty \l^2_{s-1}}
&\le
C( \| u \|_{L_T^\infty H^s}),
\label{F1F2}
\\
\| \Ml_1 \|_{L_T^\infty \l^2_{s+ \al-3}} + \| \Ml_2 \|_{L_T^\infty \l^2_{s+\al-3}}
&\le
C( \| u \|_{L_T^\infty H^s}),
\label{Mest}
\\
\|\Kl_1\|_{L_T^\infty \l^2_{s+\min(-1,\al-4)}}
+
\|\Kl_2\|_{L_T^\infty \l^2_{s+\min(-1,\al-4)}}
&\le
C( \| u \|_{L_T^\infty H^s}).
\label{Kest}
\end{align}
\end{proposition}

\begin{proof}
For simplicity, we suppress the time dependence in this proof.
First, we prove \eqref{F1F2}. 
Note that $k_1,k_2 \in \Z$ with $|k_1| \ge \frac{|k_2|}2$ implies that $\jb{k_1+k_2} \les \jb{k_1}$.
It follows from \eqref{decN1} and H\"older's inequality that
\[
\jb{k}^{s-1} |\Nl_{1,1}(k)|
\les
\sum_{k_2 \in \Z}
\jb{k-k_2}^{s-1} \big| \ft {P_{\neq 0}\Ta_{\o}}(k-k_2) \big|
\jb{k_2}
\big|\ft V (k_2)\big|.
\]
Hence,
Proposition \ref{prop:bili} and $s> \frac 52$ yields that
\begin{equation}
\begin{aligned}
\|\Nl_{1,1}\|_{\l^2_{s-1}}
&\les
\| \Ta_\o \|_{H^{s-1}}
\sum_{k_2 \in \Z}
\jb{k_2} |\ft V(k_2)|
\\
&\les
\| \Ta_\o \|_{H^{s-1}} \| V \|_{H^{s-1}}
\le
C( \| u \|_{L_T^\infty H^s}).
\end{aligned}
\label{F11}
\end{equation}

From \eqref{decN3},
the same calculation as in \eqref{F11}, we have
\begin{equation}
\|\Nl_{2,1}\|_{\l^2_{s-1}}
\le
C( \| u \|_{L_T^\infty H^s}).
\label{F21}
\end{equation}
Recall that $R$ is a polynomial in $u,v, \cj u, \cj v$.
From \eqref{ftfW},
Proposition \ref{prop:bili} and $s>\frac 32$
yield that
\begin{equation}
\| \Nl_3 \|_{\l^2_{s-1}}
\le
C( \| u \|_{L_T^\infty H^s}).
\label{F3}
\end{equation} 
Combining \eqref{F11}--\eqref{F3}, we get \eqref{F1F2}
if
$s> \frac 52$.

Next, we prove \eqref{Mest}.
When $\al>1$ and $|k_1| < \frac{|k_2|}2$,
we have
\begin{equation}
\begin{aligned}
\big| |k_1+k_2|^\al - |k_2|^\al \big|
&= \bigg| k_1 \al \int_0^1 |\theta k_1 + k_2|^{\al-2} (\theta k_1 + k_2) d\theta \bigg|
\\
&\sim
|k_1| |k_2|^{\al-1}.
\end{aligned}
\label{D2e1}
\end{equation}
It follows from \eqref{decN2} and \eqref{D2e1} that
\begin{align*}
\jb{k}^{s+\al-3} |\Ml_1 (k)| 
&\les
\sum_{k_1 \in \Z}
\jb{k_1}^{-1}
\big| \ft{P_{\neq 0} \Ta_\o}(k_1)\big|
\jb{k-k_1}^{s-1} \big|\ft V(k-k_1)\big|.
\end{align*}
Hence,
we obtain that
\begin{equation}
\begin{aligned}
\| \Ml_1 \|_{\l^2_{s+\al-3}}
&\les
\sum_{k_1 \in \Z}
\jb{k_1}^{-1}
\big| \ft{P_{\neq 0} \Ta_\o}(k_1)\big|
\| V \|_{H^{s-1}}
\\
&\les
\| \Ta_{\o} \|_{H^{s-1}}
\| V \|_{H^{s-1}}
\le
C( \| u \|_{L_T^\infty H^s})
\end{aligned}
\label{M1}
\end{equation}
if $s> \frac 32$.

From \eqref{decN4},
the same calculation as in \eqref{M1} yields that
\begin{equation}
\| \Ml_2 \|_{\l^2_{s+\al-3}}
\le
C( \| u \|_{L_T^\infty H^s}).
\label{M2}
\end{equation}
From \eqref{M1} and \eqref{M2}, we get \eqref{Mest}
if $s> \frac 32$.

Finally, we consider \eqref{Kest}. It follows from \eqref{decN2} and \eqref{D2e1} that
\begin{equation}
\begin{aligned}
&\jb{k}^{s+\min(-1,\al-4)} |\Kl_1 (k)|
\\
&\les
\sum_{D_2(k)}
\jb{k_1}^{-1} \jb{k_2}^{s- \max (\al-1,2)}
\Big| \dt \Big( \ft{P_{\neq 0} \Ta_{\o}}(k_1) \ft V(k_2) \Big) \Big|
\\
&\les
\sum_{D_2(k)}
\biggl(
\jb{k_1}^{-\al+1}
\big| \dt  \ft{P_{\neq 0} \Ta_{\o}}(k_1) \big|
\jb{k_2}^{s-1}
\big|\ft V(k_2)\big|
\\
&\hspace*{60pt}
+
\jb{k_1}^{-1}
\big|\ft{P_{\neq 0} \Ta_{\o}}(k_1) \big| \jb{k_2}^{s-2}
\big|\dt \ft V(-k_2)\big|
\biggr)
\\
&=:
\Kl_{1,1}(k) + \Kl_{1,2}(k).
\label{pfKest}
\end{aligned}
\end{equation}
For  $\Kl_{1,1}$,
Minkowski's integral inequality and Lemma \ref{lem:noli}
implies that
\begin{equation}
\begin{aligned}
\| \Kl_{1,1} \|_{\l^2}
&\les
\sum_{k_1 \in \Z}
\jb{k_1}^{-\al+1}
\big| \dt (\ft{P_{\neq 0} \Ta_{\o}}) (k_1) \big|
\| V \|_{H^{s-1}}
\\
&\les
\| \dt \Ta_{\o} \|_{H^{-\al+\frac 32+\eps}}
\| V \|_{H^{s-1}}
\\
&\le
C(\| u \|_{L_T^\infty H^s})
\end{aligned}
\label{estK11}
\end{equation}
for $s > \regwp$ and $0< \eps \ll 1$.

For $\Kl_{1,2}$,
it follows from \eqref{pfKest} that
\begin{align*}
\| \Kl_{1,2} \|_{\l^2}
&\les
\| \Ta_{\o} \|_{L^2}
\| \dt V \|_{H^{s-2}}
.
\end{align*}
The first equality in \eqref{ftfW} yields that
\begin{align*}
\| \dt V \|_{H^{s-2}}
&\les
|P_0 \Ta_\o| \| V \|_{H^{s-1}}
+
\| ( {P_{\neq 0} \Ta_{\o}} ) \dx v \|_{H^{s-2}}
\\
&\quad
+
\|  \Ta_{\cj \o} \cj {\dx v} \|_{H^{s-2}}
+
\| R \|_{H^{s-2}}.
\end{align*}
For the second part on the right-hand side,
Proposition \ref{prop:bili} and $s> \frac 32$
yield that
\[
\| ( {P_{\neq 0} \Ta_{\o}} ) \dx v \|_{H^{s-2}}
\les
\| {P_{\neq 0} \Ta_{\o}} \|_{H^{s-1}} \| \dx v \|_{H^{s-2}}
\le
C( \| u \|_{L_T^\infty H^s})
.
\]
The third part on the right-hand side is similarly handled.
Hence, we have that
\begin{equation}
\|\Kl_{1,2} \|_{\l^2}
\le
C( \| u \|_{L_T^\infty H^s})
\label{estK12}
\end{equation}
if
$s> \frac 32$.

From \eqref{pfKest}--\eqref{estK12}, we obtain
\begin{equation}
\| \Kl_1 \|_{\l^2_{s+\min (-1,\al-4)}}
\le 
C( \| u \|_{L_T^\infty H^s})
\label{estK1}
\end{equation}
if
$s> \regwp$.

Since
\[
\frac{1}{|k|^\al + |k_2|^\al}
\le \frac1{\jb{k_1} \jb{k_2}^{\al-1}}
\]
for $|k_1| < \frac{|k_2|}2$,
the same calculation as in \eqref{estK1} yields that
\begin{equation}
\| \Kl_2 \|_{\l^2_{s+\min(-1,\al-4)}}
\le
C( \| u \|_{L_T^\infty H^s}).
\label{estK2}
\end{equation} 
Combining \eqref{estK1} and \eqref{estK2}, we get \eqref{Kest}
if
$s> \regwp$.
This concludes the proof.
\end{proof}

\subsection{Proof of Theorem \ref{thm:NE2}}
We use the same notations as in Subsection \ref{subsec:apriori}.
We only consider the case (i) in Theorem \ref{thm:NE2},
since the case (ii) follows from a straightforward modification.
In what follows, we assume that
\begin{equation}
\Im P_0 \Ta_\o(0) >0.
\label{NI:phim2}
\end{equation}

From \eqref{ftfw2}, we have
\begin{align*}
&\ft V(t,k)
\\
&=
e^{- k\int^t_0 \Im P_0 \Ta_\o(\tau) d\tau} \ft V(0)
\\
&\quad
+\sum_{j=1}^2 \Big( \Ml_j(t,k) - e^{-k\int^t_0 \Im P_0 \Ta_\o(\tau) d\tau} \Ml_j(0,k) \Big) 
\\
&\quad
+
\int^t_0 e^{-k \int^t_{t'} \Im P_0 \Ta_\o(\tau) d\tau}
\bigg( \Nl_3(t',k) - \sum_{j=1}^2 \big( k (\Im P_0 \Ta_\o(t')) \Ml_j (t',k)
\\
&\hspace*{100pt}
+ \Nl_{j,1} (t',k)+ \Kl_j (t',k) \big) \bigg)dt'
\end{align*}
for $t \in [0,T]$ and $k \in \Z$.
It holds that
\begin{equation}
\begin{aligned}
\ft V(0,k)
&=
e^{k\int^t_0 \Im P_0 \Ta_\o(\tau) d\tau} \ft V(t,k)
\\
&\quad
- \sum_{j=1}^2 \Big( e^{k\int^t_0 \Im P_0 \Ta_\o(\tau) d\tau} \Ml_j(t,k) - \Ml_j(0,k) \Big)
\\
&\quad
-
\int^t_0 e^{k \int^{t'}_0 \Im P_0 \Ta_\o(\tau) d\tau}
\bigg(\Nl_3(t',k)
\\
&\hspace*{70pt}
-
\sum_{j=1}^2 \Big( k (\Im P_0 \Ta_\o(t')) \Ml_j (t',k)
\\
&\hspace*{70pt}
+ \Nl_{j,1} (t',k) + \Kl_j (t',k) \Big)\bigg) dt'.
\end{aligned}
\label{ftV0k1}
\end{equation}

From the continuity of $P_0 \Ta_\o$ and \eqref{NI:phim2},
there exists $ T^{\ast} \in (0,T]$ such that
\begin{equation}
\Im P_0 \Ta_\o (t)
\ge
\frac {\Im P_0 \Ta_\o (0)}2
\label{intM}
\end{equation}
for any $t\in [0,T^{\ast}]$.
Then,
we note that
\begin{equation}
\begin{aligned}
&\bigg\|
\int^{T^{\ast}}_0 e^{k \int^{t'}_0 \Im P_0 \Ta_\o(\tau) d\tau} \ft{P_- f} (t',k) dt'
\bigg\|_{\l^2}
\\
&\le
\bigg\|
\int^{T^{\ast}}_0 e^{ -\frac{|k|}2 \Im P_0 \Ta_\o(0) t'} |\ft{P_- f} (t',k)| dt'
\bigg\|_{\l^2}
\\
&\le
\int^{T^{\ast}}_0
\Big\|
\Big( \frac 2{|k| \Im P_0 \Ta_\o(0) t'} \Big)^{1-\eps} \ft{P_- f}(t',k) \Big\|_{\l^2} dt'
\\
&\le
\Big( \frac{2 \sqrt 2}{\Im P_0 \Ta_\o(0)} \Big)^{1-\eps}
\frac{(T^\ast)^\eps}{\eps}
\| P_- f \|_{L_{T^\ast}^\infty H^{-1+\eps}}
\end{aligned}
\label{eintM}
\end{equation}
for $\eps \in (0,1)$ and $f \in L^\infty ([0,T^\ast]; H^{-1+\eps}(\T))$.

Let $\dl>0$ be specified later.
From
\eqref{ftV0k1} with $t=T^\ast$
and
\eqref{eintM},
we obtain that
\begin{equation}
\begin{aligned}
\| P_-V (0) \|_{H^{s-1+\dl}}
&\les
\Big\| \jb{k}^{s-1+\dl} e^{k\int^{T^\ast}_0 \Im P_0 \Ta_\o(\tau) d\tau} \ft{P_-V} (T^\ast,k) \Big\|_{\l^2}
\\
&\quad
+
\| \Nl_3 \|_{L_{T^\ast}^\infty \l^2_{s-2+\dl+\eps}}
\\
&\quad
+
\sum_{j=1}^2
\big(
\| \Ml_j \|_{L_{T^\ast}^\infty \l^2_{s-1+\dl+\eps}}
+ \| \Nl_{j,1} \|_{L_{T^\ast}^\infty \l^2_{s-2+\dl+\eps}}
\\
&\hspace*{40pt}
+ \| \Kl_j \|_{L_{T^\ast}^\infty \l^2_{s-2+\dl+\eps}}
\big)
\label{kft0}
\end{aligned}
\end{equation}
for $\eps \in (0,1)$.
Proposition \ref{prop:FGH} implies that
\begin{equation}
\begin{aligned}
&\| \Nl_3 \|_{L_{T^\ast}^\infty \l^2_{s-2+\dl+\eps}}
+
\sum_{j=1}^2
\big(
\| \Ml_j \|_{L_{T^\ast}^\infty \l^2_{s-1+\dl+\eps}}
\\
&\hspace*{100pt}
+ \| \Nl_{j,1} \|_{L_{T^\ast}^\infty \l^2_{s-2+\dl+\eps}}
+ \| \Kl_j \|_{L_{T^\ast}^\infty \l^2_{s-2+\dl+\eps}}
\big)
\\
&\le
C(\dl, \| u \|_{L_T^\infty H^s})
\end{aligned}
\end{equation}
provided that
\begin{align*}
0<\dl <\min (\al-2,1),
\quad
0< \eps< \min (\al-2,1) - \dl. 
\end{align*}

From \eqref{intM},
a direct calculation shows that
\begin{equation}
\begin{aligned}
&\Big\| \jb{k}^{s-1+\dl} e^{k\int_0^{T^\ast} \Im P_0 \Ta_\o(\tau) d\tau} \ft{P_-V} (T^\ast,k) \Big\|_{\l^2}
\\
&\les
\big\| \jb{k}^{s-1+\dl} e^{-|k| \frac{\Im P_0 \Ta_\o(0)}2 T^\ast} \ft{P_-V} (T^\ast,k) \Big\|_{\l^2}
\\
&\les
\| V \|_{L_T^\infty H^{s-1}}
\end{aligned}
\label{ftV0aa}
\end{equation}
for $\dl>0$.
Therefore,
\eqref{kft0}--\eqref{ftV0aa},
we obtain that
$P_- V(0) \in H^{s-1+\dl}(\T)$
provided that
\[
0<\dl <\min (\al-2,1).
\]
Thus,
we have
\begin{align*}
\| P_- \phi \|_{H^{s+\dl}}
\le
C( \dl, \| u \|_{L_T^\infty H^s}).
\end{align*}
This shows (i) in Theorem \ref{thm:NE2}.

\mbox{}

\noindent
{\bf 
Acknowledgements.}
This work was
supported by JSPS KAKENHI Grant numbers
JP23K03182 and JP25K07070.

\end{document}